\pdfoutput=1
\documentclass[12pt]{article}

\usepackage{amsmath}
\usepackage{graphicx}
\usepackage{enumerate}
\usepackage{natbib}
\usepackage{url} 

\usepackage{amssymb}
\usepackage{tabularx}

\RequirePackage[colorlinks,citecolor=blue,urlcolor=blue]{hyperref}
\usepackage[ruled]{algorithm2e}
\usepackage{algorithmic}
\usepackage{graphicx,psfrag,epsf}
\usepackage{epstopdf}
\usepackage{subfloat}
\usepackage{rotating}
\usepackage{subfigure}
\usepackage{longtable}
\usepackage{engord}
\usepackage{color}
\usepackage{amssymb}
\usepackage{float}
\usepackage{hhline}
\usepackage{version}
\usepackage{fancyheadings}
\usepackage{appendix}
\usepackage{epstopdf}
\usepackage{multirow}
\usepackage[ruled]{algorithm2e}
\usepackage{caption}
\usepackage{xcolor}

\usepackage{amsthm}

\usepackage{soul}

\newtheorem{defi}{Definition}[section]
\newtheorem{lemma}{Lemma}[section]
\newtheorem{thm}{Theorem}[section]

\newtheorem{proposition}{Proposition}[section]
\newtheorem{assumption}{Assumption}[section] 
\newtheorem{remark}{Remark}[section] 

\newcommand{\M}{{\cal M}}

\newcommand{\I}{{\cal I}} 

\def\trans{^{\rm T}}

\let\sss= \scriptscriptstyle

\numberwithin{equation}{section}

\newcommand{\Yao}[1]{{\textcolor{black}{#1}}}
\newcommand{\Xia}[1]{{\textcolor{black}{#1}}} 
\newcommand{\ZF}[1]{{\textcolor{black}{#1}}} 

\newcommand{\RVSD}[1]{{\textcolor{black}{#1}}}

\begin{document}

\def\spacingset#1{\renewcommand{\baselinestretch}%
{#1}\small\normalsize} \spacingset{1}


\title{\bf Random Fixed Boundary Flows}

\maketitle
 
\author{\noindent Zhigang Yao\\
    Department of Statistics and Data Science \\ 21 Lower Kent Ridge Road\\
National University of Singapore, Singapore 117546 \\\\ 
Center of Mathematical Sciences and Applications  \\ 20 Garden Street\\
Harvard University\\
Cambridge, USA   02138\\
\noindent  email: \texttt{zhigang.yao@nus.edu.sg or zhigang.yao@cmsa.fas.harvard.edu}\\

\noindent     Yuqing Xia\\
School of Data Science  \\ 
Zhejiang University of Finance and Economics\\
Hangzhou, China 310018  \\
\noindent email: \texttt{xiayq0121@zufe.edu.cn}\\

 \noindent    Zengyan Fan\\
School of Science and Technology\\ 463 Clementi Road\\ 
Singapore University of Social Sciences, Singapore 599494  \\ 
\noindent  email: \texttt{zyfan@suss.edu.sg}
}
 
\vspace{0.1 in}

\begin{abstract}
\RVSD{We consider fixed boundary flow with canonical interpretability as principal components extended on non-linear Riemannian manifolds. We aim to find a  flow with fixed starting and ending points for noisy multivariate data sets lying on an embedded non-linear Riemannian manifold. 
In geometric term, the  fixed boundary flow is defined as an optimal curve that moves in the data cloud with two fixed end points. At any point on the flow, we maximize the
inner product of the vector field, which is calculated locally, and the tangent vector of the flow. The rigorous definition derives from an optimization problem using the intrinsic metric on the manifolds. For random data sets, we name the fixed boundary flow the random fixed boundary flow  and analyze its limiting behavior under noisy observed samples. We construct a high level algorithm to compute the random fixed boundary flow and the convergence of the algorithm is provided. We show that the fixed boundary flow yields a concatenate of three segments, of which one coincides with the usual principal flow when the manifold is reduced to the Euclidean space. {\color{black} We further prove that the random fixed boundary flow converges largely to the population fixed boundary flow with high probability.} We illustrate how the random fixed boundary flow can be used and interpreted, and showcase its application in real data sets.}
\end{abstract}

\noindent%


\section{Introduction}
\addtolength{\textheight}{.5in}%
Most existing statistical methods assume a linear dependency between features.
As the dimensionality of features increases, the representation of the features in a high-dimensional space becomes more complex and it thus becomes more challenging to understand the relationships between features. In many applications, modern data structures are often complex and not necessarily linear. Indeed it is often the case that there is a lower-dimensional structure, namely a manifold embedded in the high-dimensional {\color{black}ambient space \citep{fefferman2016testing,pmlr-v75-fefferman18a},} as in the examples of geometric shapes in the shape space \citep{turk1994zippered,dryden2016statistical,kilian2007geometric,Bradley2013} and graphs in computer graphics \citep{phillips1997feret,Gross2005,arjovsky2017wasserstein}.

A series of methods that aim to recover the underlying structure of the lower-dimensional manifold have been developed over the past two decades. These methods, usually called manifold learning, are focused mostly on mapping data in a $d$-dimensional space into a set of points close to an $m$-dimensional ($m \ll d$) manifold. Among them, is a method known as known as the Principal Component Analysis (PCA), which is commonly used to reduce the feature dimension in the Euclidean space. To address features lying in a non-linear space (i.e., a manifold), methods such as LLE \citep{roweis2000nonlinear}, Isomap \citep{tenenbaum2000global}, MDS \citep{cox2000multidimensional}, and LTSA \citep{zhang2004principal}, which determine the low-dimensional embedding, preserving local properties of the data, may be preferable. A comprehensive review of such work appears in \citet{ma2011manifold}.

Another line of research relating to statistics on manifolds is centered on the extension of existing methods defined in the Euclidean space to the manifold space. The manifold space can be the actual physical space that the data lies on or the learnt manifold created through the manifold learning methods. In recent decades, numerous non-linear approaches have been developed to analyze the data on the manifold directly \citep{Jupp1987,fletcher2004principal,Huckemann2006,Kumi2007,Fletcher2007,Kenobi2010,Jung2012,Eltzner2017}. Throughout the paper, we focus on the known manifold, based on the assumption that the manifold embedding is known.

Next, we will mainly review the ``curve fitting'' methods on manifolds. 
A geodesic is a generalization of the straight lines in the standard Euclidean space to the manifold. The principal geodesic analysis \citep{fletcher2004principal}, which extends the PCA to the manifold, was proposed to describe the non-linear variability of data on a manifold. The principal curves, proposed in \citet{hastie1989principal}, are flexible one-dimensional curves that pass through the middle of data points. Having said that, principal curves are able to better capture the non-linear variation of data in comparison to all other regression lines in the Euclidean space.
\citet{ozertem2011locally} redefined principal curves and surfaces in terms of the gradient and the Hessian of the probability density estimate, based on the consideration that every point on the principal surface should be at the local maximum of the probability density in the local orthogonal subspace, and not the expected value as in \citet{hastie1989principal}. For applications in classification tasks, \citet{ladicky2011locally} proposed a new curve fitting method to find the smooth decision boundary with bounded curvature.

A recent piece of work on principal flows \citep{panaretos2014principal} works as an extension of the principal curves on Riemannian manifolds. Therefore, the principal flows are also flexible one-dimensional curves, which pass through the Fr\'echet mean of the data points. The principal flows are able to capture the non-geodesic pattern of variation both locally and globally.
Instead of handling curves with an explicit parameterization,
\citet{liu2017level} combine the level set method with the principal flow algorithm to obtain a fully implicit formulation, so that the obtained co-dimension one surface on the manifold fits the data set well.


\RVSD{When the data comes with multiple paths, it would be quite natural to want to isolate one of the paths in particular - that with a fixed direction. All the methods outlined above/earlier fail to determine flows with fixed directions. Hence, we propose flows with a fixed direction, each determined by fixed data boundaries, namely their start and end points. 
For example, we consider seismological events that took place in the Sea of Japan between 1904 and 2015, with the epicentres plotted as green dots in Figure \ref{Fig:plate}(b). 
From the information of tectonic plates shown in Figure \ref{Fig:plate}(a), we observe that the seismological events in this analysis tended to occur around the tectonic plate boundaries (shown as black curves with triangles in Figure \ref{Fig:plate}(a)). Specifically, we deduce that the seismological events occurred frequently along the boundaries of four tectonic plates: the North American plate, the Eurasian plate, the Philippine Sea plate and the Pacific plate.
Given these seismological events, the principal flow passes through the Fr\'echet mean and captures local variations that depend on the value for the scale parameter, $h$. Since there are a greater number of seismological events along the boundary of the Pacific plate, the resulting Fr\'echet mean appears around the Pacific plate and the principal flow starts moving from the Fr\'echet mean. In Figure \ref{Fig:plate}(b), the red curve represents the principal flow of the earthquake data for a scale parameter of $400$ miles. 
We observe that the principal flow moves along the boundary of the Pacific plate (red curve in Figure \ref{Fig:plate}(c)). 
When we focus on the seismological events caused by the Philippine Sea plate, the principal flow will not be of interest in terms of finding a boundary. In this sense, the trend along the boundary highlighted in blue in Figure \ref{Fig:plate}(c) would be more appropriate.  Although we could derive a flow similar to that shown in blue by selecting the data with latitudes and longitudes around the Philippine Sea plate, it is hard to accurately determine which data points to include in practice. Hence, we propose fixed boundary flows, where the flow will be automatically determined by using boundary points that are chosen by users manually. If we select start and end points around the Philippine Sea plate, the obtained fixed boundary flow for a scale parameter of $400$ miles is shown in blue in Figure \ref{Fig:plate}(b). Furthermore, we observe that the fixed boundary flow starts from the fixed starting point, moves along the boundary of the Philippine Sea plate (blue curve in Figure \ref{Fig:plate}(c)) and terminates at the fixed ending point.}

\begin{figure}[h]
\begin{center}
\subfigure[]{
\includegraphics[width=1.5in]{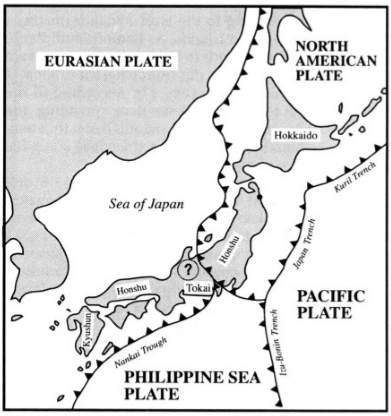}}
\subfigure[]{
\includegraphics[width=1.45in]{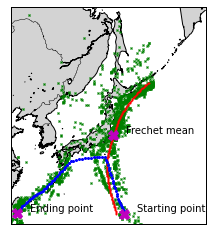}}
\subfigure[]{
\includegraphics[width=1.5in]{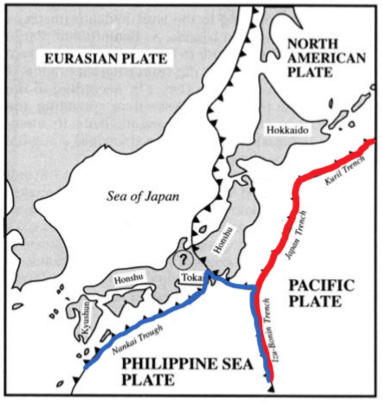}}
\end{center}
\caption{Seismology events in the Sea of Japan. (a) Map of tectonic plates around the Sea of Japan. Black curves with triangles: plate boundaries; (b) Earthquake epicentre data of the Sea of Japan (green dots). Red curve: principal flow for a scale parameter of $400$ miles; blue curve: fixed boundary flow for a scale parameter of $400$ miles; (c) Map of tectonic plates with highlighted boundaries that correspond to principal flows (in red) and fixed boundary flow (in blue).}\label{Fig:plate}
\end{figure}

In order to clarify the aforementioned concepts and parameters,
we hereby review the technique proposed for principal flows in brief and demonstrate that this technique comes up short when considering boundary constraints. 
Throughout this paper we work within the context of a complete Riemannian manifold $\M$ of dimension $m$, and $\M$ is isometrically embedded into the Euclidean space $(\mathbb{R}^d,\|\cdot\|)$ with $m<d$. The related preliminaries in Riemannian geometry can be found in the Supplementary Materials. \Xia{Given data points $\{x_i\}_{i=1}^n$ on the Riemannian manifold,}
the methodology for the principal flow seeks to solve for a curve on the manifold that passes through the Fr\'echet mean 
of the data, such that the tangent vector along the curve locally follows the direction of maximal variation of the data in a local tangent space. \Xia{As we will define later, the vector field characterizes the direction of maximal variation and the scale parameter characterizes how locally or globally we wish to describe a path of maximal variation}. A flow with a large scale parameter captures the global trend while a flow with a reduced scale parameter describes the finer structure.

Mathematically, the principal flow finds a curve $\gamma : [0,r] \to \mathcal{M}$ starting at a Fr\'echet mean $\bar{x}$ and maximizing
\begin{align}\label{obj:prin_flow}
\int_{0}^r \lambda_1(\gamma(t)) \langle \dot{\gamma}(t), W_{n,h}(\gamma(t)) \rangle dt.
\end{align}
where $\lambda_1(x)$ and $W_{n,h}(x)$ are the first eigenvalue and the first unit eigenvector of the local tangent covariance matrix $\Sigma_{n,h}(x)$, respectively. The definition of $\Sigma_{n,h}(x)$ is reviewed in the Supplementary Materials, and we remark that $\Sigma_{n,h}(x)$ is computed with the projections of data points onto $T_x\M$, which implies that the first eigenvector $W_{n,h}(x)$ also locates in the tangent space at $T_x\M$. The subscript $n$ of $\Sigma_{n,h}$ indicates that $\Sigma_{n,h}(x)$ is calculated from the data points of cardinality $n$, while the subscript $h$ of $\Sigma_{n,h}(x)$ indicates the locality. Specifically, $\Sigma_{n,h}(x)$ is computed using the data points in $B_d(x,h)$, the Euclidean ball centered at $x$ of radius $h$.  With a different $x$, the eigenvectors form the vector field $W = \{W_{n,h}(x)\}$. \Xia{To avoid confusion, we will omit the subscript $n$ and $h$ of $W_{n,h}(x)$ hereafter. The first eigenvalue is assumed to be simple throughout, which guarantees the uniqueness of $W$.} We note that the projection onto the local tangent space might be impossible in practice, in the case that either $\M$ or the formula of the local tangent space is unavailable. Under these circumstances, we might omit the projection step in computing $\Sigma_{n,h}(x)$ and use the local covariance matrix in ambient space instead.



{\color{black} 
Let us think of a simple example: noisy ``C''-shaped data in $\M = \mathbb{R}^2$ as shown in Figure \ref{Fig:lambda1}. Furthermore, by setting $h = \infty$, we will use this example to demonstrate that determining fixed boundary flows is not a simple extension of the work of principal flows. The first eigenvalue $\lambda_1(x)$ in \eqref{obj:prin_flow} varies with $x$ with Figure \ref{Fig:lambda1} visualizing its changes. One may see that the first eigenvalue reaches its trough at the Fr\'echet mean, which in turn implies that the first eigenvalue would have been increasing along any direction after its departure from $\bar{x}$. By differentiating $\lambda_1(x)$ (see derivation given in the Supplementary Materials) we have
\begin{align*}
    d \lambda_1(x) = 2\langle W(x)W(x)^T(x-\bar{x}), dx \rangle,
\end{align*}
and that $\lambda_1(x)$ increases most rapidly along its gradient, that is, $W(x)$ and $-W(x)$. Therefore, maximizing either $\lambda_1(\gamma(t))$ or $\langle  \dot{\gamma}(t), W(\gamma(t)) \rangle$ at $\gamma(t)=\bar{x}$ 
locally leads to two half-lines along $W(\bar{x})$ and $-W(\bar{x})$ starting from $\bar{x}$. Hence, maximizing the optimization problem \eqref{obj:prin_flow}, \Xia{which is the product of $\lambda_1(\gamma(t))$ and $\langle  \dot{\gamma}(t), W(\gamma(t)) \rangle$} locally, leads to the principal flow along $W(\bar{x})$ through $\bar{x}$, as represented by the dashed line on the left panel.

\begin{figure}[h]
    \centering
    \includegraphics[width=0.6\textwidth]{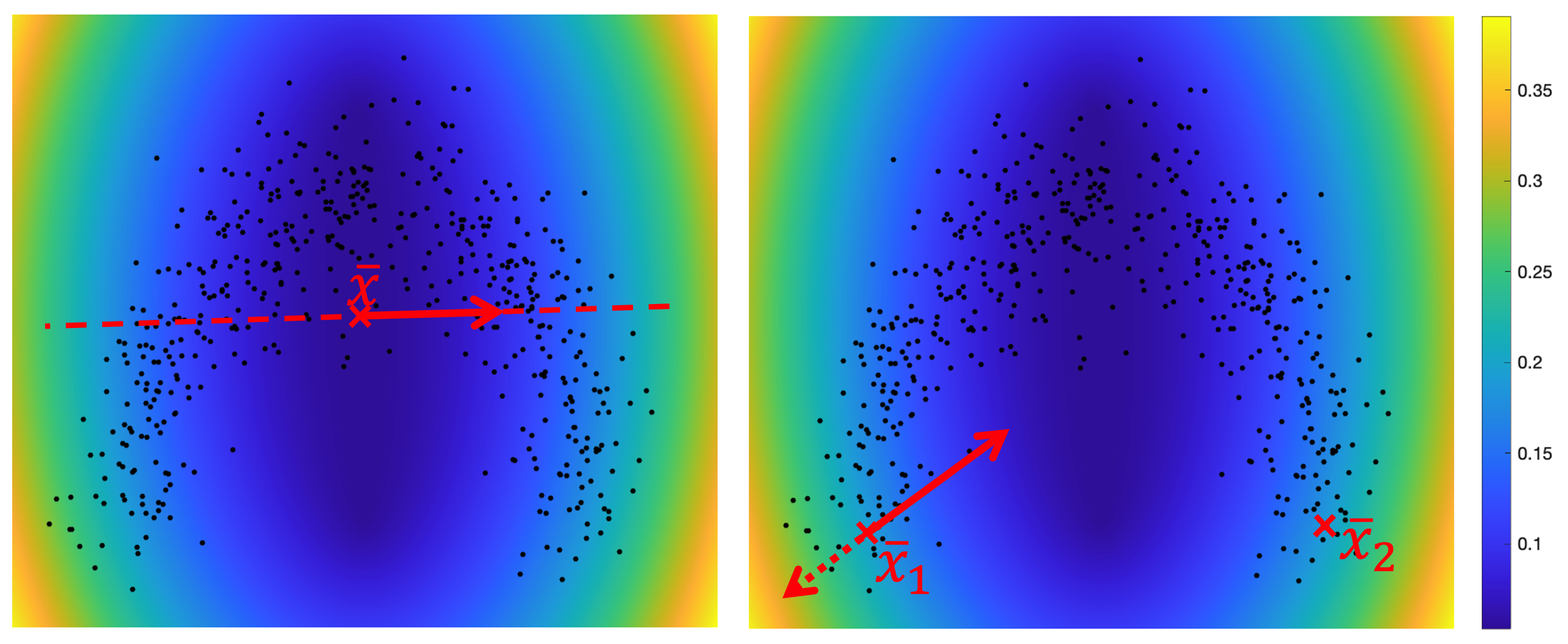}
    \caption{Distribution of the first eigenvalue $\lambda_1(x)$ for $x \in \M$ with $h=\infty$. 
    The black points represent sample points; the red arrow represents the direction of $W(x)$; the red cross on the left panel represents the Fr\'echet mean $\bar{x}$ of the sample and the dashed line on the left panel represents the principal flow; two red crosses on the right panel represent the fixed boundary $\bar{x}_1$ and $\bar{x}_2$; the dashed arrow on the right panel represents the opposite direction of $W(x)$; the color bar represents the magnitude of $\lambda_1(x)$.
}
\label{Fig:lambda1}
\end{figure}

Things are very different when one considers the fixed boundary flows, which begin from the fixed starting point $\bar{x}_1$, move along the data points and end at the fixed ending point $\bar{x}_2$. From the right panel of Figure \ref{Fig:lambda1}, we observe that the \Xia{$\lambda_1(x)$} is large at the boundary and will decrease when a curve moves towards the data cloud's center from $\bar{x}_1$. Furthermore, from the differentiation form, $\lambda_1(x)$ decreases most rapidly along its gradient $W(\bar{x}_1)$, as shown by the red arrow  in the right panel of Figure \ref{Fig:lambda1} and increases most rapidly along $-W(\bar{x}_1)$, the dashed arrow shown on the right panel of Figure 2. Therefore, if one maximizes the inner product $\langle  \dot{\gamma}(t), W(\gamma(t)) \rangle$ \Xia{at $\gamma(t) = \bar{x}_1$} in \eqref{obj:prin_flow}, in favor of the curve moving along the vector field $W(\gamma(t))$, \Xia{one should take $\dot{\gamma}(t) = W(\gamma(t))$. This means the curve would move along the red arrow in the right panel of Figure \ref{Fig:lambda1}, which makes $\lambda_1(\gamma(t))$ decrease most rapidly.} While if one maximizes $\lambda_1(\gamma(t))$, \Xia{one should take $\dot{\gamma}(t) = -W(\gamma(t))$, in favor of the curve moving along $-W(\gamma(t))$ since it is the gradient of $\lambda_1(\gamma(t))$. However, taking $\dot{\gamma}(t) = -W(\gamma(t))$ makes $\langle  \dot{\gamma}(t), W(\gamma(t)) \rangle$ decrease fastest.} From this point of view, we conclude that maximizing $\lambda_1(\gamma(t))$ and the inner product $\langle  \dot{\gamma}(t), W(\gamma(t)) \rangle$ in \eqref{obj:prin_flow} is mutually conflicting. Such conflict makes the fixed boundary flows unique, unlike the principal flows, meaning one cannot, therefore, simply extend the optimization problem of principal flows to fixed boundary flows.

We are now motivated to consider the fixed boundary flows that capture the manifold data variation in a way that differs from the principal flows. To achieve this, we initialize an optimization problem to capture a smooth \ZF{flow} for non-random data lying on the manifold that starts and ends at pre-defined points in Section \ref{SEC:FBF}. For each point of the flow, its tangent vector is close to the vector field at that point. When noise presents, the data follows from the underlying distribution of the population flow on the manifold and it is thus non-deterministic. And so too are the fixed boundary flows. The random fixed boundary flows, generalizing the fixed boundary flows, are proposed in Section \ref{SEC:RFBF_ALGORITHM}. An efficient algorithm to determine the random fixed boundary flow, with its convergence of the random fixed boundary flow, is outlined in Section \ref{SEC:CONV_RFBF_ALGORITHM}. In Sections \ref{SEC:simulation} and \ref{SEC:realdata}, we illustrate that the random fixed boundary flow is able to capture patterns of variation in \Xia{synthetic, seismic and real-world image data.} Several statistical properties and theories of the fixed boundary flow are examined in Section \ref{SEC:FBF_EUCLIDEAN}.


\section{Fixed boundary flows}\label{SEC:FBF} 
Fixing two boundary points produces an infinite number of flows. To begin with, we describe the class of curves that provide the candidates of the fixed boundary flows. Given $\bar{x}_1$ and $\bar{x}_2$, we define the class as:
\begin{eqnarray}
\Gamma(\bar{x}_1,\bar{x}_2)&=\{&\gamma: [0,r]\to \mathcal{M}: \gamma\in C^2([0,r]), r\in(0,\Xia{C\Delta}], \nonumber\\
&& \gamma(0)=\bar{x}_1, \gamma(r)=\bar{x}_2, \gamma(s)\neq \gamma(s^\prime)~\text{for}~s\neq s^\prime, \nonumber \\
&& \ell(\gamma[0,t])=t, \text{for~all}~ t\in[0,r]\},\label{curve_set_original}
\end{eqnarray}
where $W(\gamma(t))$ is the value of the vector field $W$, calculated form local data $\{x_i\}_{i=1}^n$ at $\gamma(t)$, and $\ell(\gamma[0,t])$ denotes the length of the parametric flow \Xia{$\gamma[0,t]$} 
from $\gamma(0)$ to $\gamma(t)$, for all \Xia{$0 < t \leq r$}. \Xia{Here, $\Delta = d(\bar{x}_1, \bar{x}_2)$ denotes the geodesic distance between $\bar{x}_1$ and $\bar{x}_2$ and $C>1$ is a given constant. The choice of $C$ controls the size of $\Gamma(\bar{x}_1, \bar{x}_2)$.  Since $t\in [0,\Xia{C\Delta}]$, the length of the flows in the class $\Gamma(\bar{x}_1,\bar{x}_2)$ is less than $C\Delta$. A smaller C filters out the flows that (1) are far away from the data cloud by restricting the length of flows in the class $\Gamma(\bar{x}_1, \bar{x}_2)$, and (2) overfits the data (this is because overfitted flows tend to go through all data points which will increase its length). We assume $\Delta < 1$ without loss of generalization, otherwise the manifold $\M$ should be rescaled.} For any flow $\gamma \in \Gamma(\bar{x}_1,\bar{x}_2)$, we could determine its moving direction and vector field at every point. The moving directions and vector fields vary with different points and different flows. 
To \Xia{follow the direction of highest variation, }
\Yao{we aim to find a flow with a moving direction that matches the vector field as much as possible at any given point on the flow. 
From  the classical mechanics perspective, we seek a flow with fixed starting and ending points, that best approximate the vector field globally. 
Conventional local Euclidean approaches fail to achieve this without being able to accommodate the boundary conditions globally, while forcing the flow to stay on the manifold. We term such an optimal flow the fixed boundary flow; that is, it is defined as a smooth flow $\gamma$ on the manifold $\mathcal{M}$, starting and ending at the fixed points, with a derivative vector $\dot{\gamma}$ that is maximally compatible with the vector field $W$, calculated from local data.}

\begin{defi}(Fixed boundary flow at scale $h$)\label{Definition:FBF}
Let $\bar{x}_1, \bar{x}_2\in B$, where $B$ is the neighborhood that contains the data $\{x_i\}_{i=1}^n$ on the manifold. Assume that $\Sigma_{n,h}(x)$ have distinct first and second eigenvalues for any $x\in B$. A fixed boundary flow of $\{x_i\}_{i=1}^n$ with given $\bar{x}_1$ and $\bar{x}_2$ is the curve satisfying
\begin{equation}
\gamma = \arg\sup _{\gamma\in\Gamma(\bar{x}_1,\bar{x}_2)}\int_0^{\ell(\gamma)}\langle \dot{\gamma},W(\gamma(t))\rangle dt, \label{opt_original}
\end{equation}
where $W(\gamma(t))$ is the vector field over the neighborhood of $\gamma(t)$ for $0\le t\le C\Delta$.
\end{defi}
\noindent
The fixed boundary flow is the solution of the optimization problem defined in \eqref{opt_original}.

\section{Random fixed boundary flows}\label{SEC:RFBF}
Besides being high-dimensional, the data on the manifold is usually noisy, representing some underlying distribution. One accessible way to illustrate the noisy data is shown in the following assumption.


\begin{defi}(Population flow)
\Xia{Given boundary points $\bar{x}_1$ and $\bar{x}_2$, there exists a population flow $\gamma^{\ast} \subset \M$ under unit speed parameterization, depending on the continuous vector field $W^{\ast}$ distribution. Assume that $\gamma^{\ast}$ passes through $\bar{x}_1$ and $\bar{x}_2$, which means $\gamma^{\ast}(0) \neq \bar{x}_1$ and $\gamma^{\ast}(r^{\ast}) \neq \bar{x}_2$ with $r^{\ast} = \ell(\gamma^{\ast})$. }
\end{defi}


\begin{assumption} \label{assumption:noise}
The data points $\{x_i\}_{i=1}^n \subset \mathcal{M}$ satisfy
\begin{equation}
\Xia{x_i = \gamma^*(t_i) + \xi_i, \quad {\rm for} \  i = 1, \ldots, n}
\label{data_generation}
\end{equation}
where $t_1 \leq \cdots \leq t_n$ are ordered indices sampled from uniform distribution between $[0, r^*]$ on $\gamma^*$  and $\{\xi_i\}_{i=1}^n \sim N(0, \sigma^2I_d)$ are i.i.d. Gaussian noises.
\end{assumption}


\noindent

\begin{remark}
\Xia{Also, $\gamma^*(t_1) \neq \bar{x}_1$ and $\gamma^*(t_n) \neq \bar{x}_2$. {\color{black}As shown in Figure \ref{Fig:gamma_star}}, $\bar{x}_1$ and $\bar{x}_2$ are chosen to be at the
inner end of the data cloud so that there are enough samples in the neighborhood of $\bar{x}_1$ and $\bar{x}_2$. Section \ref{SEC:CONV_RFBF_ALGORITHM} will further formulate the relationship between $\bar{x}_1$($\bar{x}_2$) and the end points of the population flow.} 
\end{remark}


Under Assumption \ref{assumption:noise}, the relation between the fixed boundary flow and $\gamma^{\ast}$ is summarized in the following theorem. 


\begin{thm}\label{THM:RANDFBF}
Suppose Assumption \ref{assumption:noise} holds,
the vector field $W$ is calculated at scale $h$ and $T = \{t: \|\gamma^\ast(t)-\gamma^\ast(0)\| > h/2 \ \mbox{and} \ \|\gamma^\ast(t)-\gamma^\ast(r^*)\| > h/2\}$. For any $t \in T$ and given $\delta > 0$, there exist constants $C$ and $n_0$ such that if $n \geq n_0$,  then $\langle  \dot{\gamma}^\ast(t), W(\gamma^\ast(t) \rangle \geq 1 - \frac{C}{2}h^2$ with probability $1-\delta$.
\end{thm}

\noindent The proof of Theorem \ref{THM:RANDFBF} is given in Appendix B in the Supplementary Materials.
From Theorem \ref{THM:RANDFBF}, we observe that 
the inner product $\langle  \dot{\gamma}^\ast(t), W(\gamma^*(t)\rangle$ is close to its maximum, that is, $1$ with sufficiently small $h$. This means that the integrand in the optimization problem (\ref{opt_original}) achieves a very large value along \Xia{the main segment of the flow $\gamma^*$. Note that we choose to work with $\gamma$ simply because there might not be enough samples at the two ends of $\gamma^*$.  Here, boundary $\bar{x}_1$ and $\bar{x}_2$ are at the  inner end of the data cloud so that the main segment $\gamma^*(\bar{x}_1, \bar{x}_2)$ is $h/2$ away from $\gamma^*(0)$ and $\gamma^*(r^*)$. Hence, $\gamma^*(\bar{x}_1, \bar{x}_2)$ well approximates the optimal solution to (\ref{opt_original}), {\color{black} as illustrated by Figure \ref{Fig:gamma_star}.}  The theoretical analysis will focus on the main segment of  $\gamma^*$. 
For convenience,  we use $\gamma^\ast$ simplifying  $\gamma^{\ast}(\bar{x}_1,\bar{x}_2)$ for the rest of the paper.}


\begin{figure}[h]
\begin{center}
\includegraphics[width=2.5in]{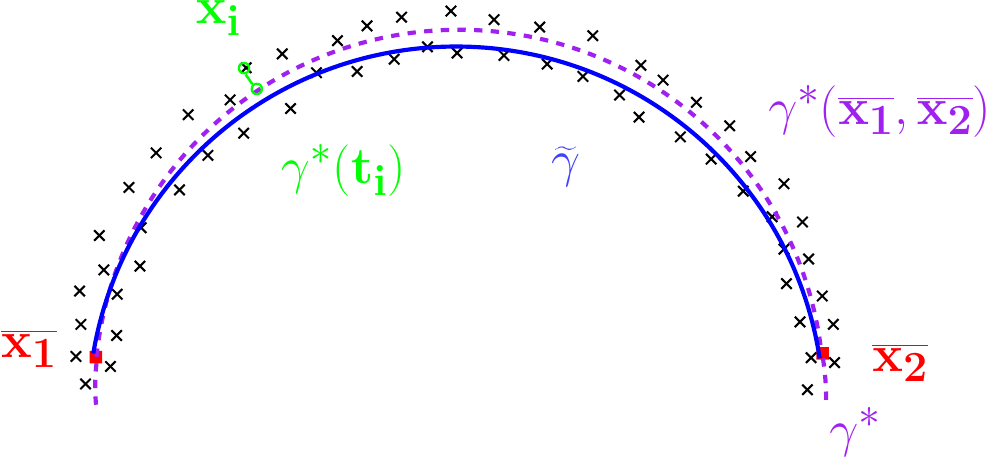}
\end{center}
\caption{Illustration of the population flow $\gamma^{\ast}$ and random fixed boundary flow $\tilde{\gamma}$. The main section of $\gamma^{\ast}$ is plotted along the purple dashed line and $\tilde{\gamma}$ is shown on the blue curve. A selected data point $x_i$ and its corresponding segment $\gamma^{\ast}(t_i)$ are highlighted in green.}\label{Fig:gamma_star}
\end{figure}

Now, let us turn to the random fixed boundary flows. Under  Assumption \ref{assumption:noise} and given fixed boundaries, a random fixed boundary flow is the empirical flow, $\tilde{\gamma}$, computed from the data points with the fixed boundary.
Our focus here is twofold. First is to determine the random fixed boundary flows through an efficient algorithm without intensive computation. Second is to investigate the distance property between the random fixed boundary flow and the population flow $\gamma^\ast$, where a theoretical analysis of the bound of the Hausdorff distance is derived from the geometry property of the underlying manifold. 

\subsection{Determination of random fixed boundary flows}\label{SEC:RFBF_ALGORITHM}



The aim of the proposed approach is to determine the random fixed boundary flow via a discrete flow with the fixed boundary. Furthermore, each point of the discrete flow moves along the direction of the vector field, which captures the localized variation maximally. From this perspective, the proposed approach attains an approximate solution for the original optimization problem in \eqref{opt_original}.

Given the fixed boundary points $\bar{x}_1$ and $\bar{x}_2$, 
the implementation begins with a discrete flow $\tilde{\gamma}^{\sss (0)}$ starting at $\bar{x}_1$ and ending at $\bar{x}_2$, with a user-defined resolution $N$. The choice of the initial flow $\tilde{\gamma}^{\sss (0)}$ can be a geodesic on the manifold $\mathcal{M}$ or a straight line from $\bar{x}_1$ to $\bar{x}_2$ in the ambient space, neither derailing the convergence of the algorithm, as we will show. The initial flow is denoted by $\tilde{\gamma}^{\sss (0)}(t_i)$, with $0 = t_0 < t_1 < ... < t_{2N} = 1$ and satisfied $\tilde{\gamma}^{\sss (0)}(t_0)=\bar{x}_1$ and $\tilde{\gamma}^{\sss (0)}(t_{2N})=\bar{x}_2$. 
Then, the proposed approach will iteratively update the flow $\tilde{\gamma}^{\sss (k)}(t_i)$ from $k=1$ until the convergence criterion is met. Gradually, at each point, the flow $\tilde{\gamma}^{\sss (k)}(t_i)$ is determined to maximize the localized variation of the data. Hence, user-defined values for the scale parameter $h$, shrinkage constant $\rho$ and stopping criterion constant $\epsilon$, are each needed during iterations. 

\RVSD{
During the iterations, we update the discrete flow $\tilde{\gamma}^{\sss (k)}(t_{i})$, $i=0,1,\ldots,2N$, for $k=1,2,\ldots$ by maximizing the optimization problem \eqref{opt_original}. There are four core steps with this aim in mind: choosing scale parameter, calculating local covariance matrix, determining vector field, and updating. Here, we elaborate on each of these core steps, as shown in Figure \ref{Fig:vector_field_update}. }

\begin{figure}[h]
\begin{center}
\subfigure[]{
\includegraphics[width=1.8in]{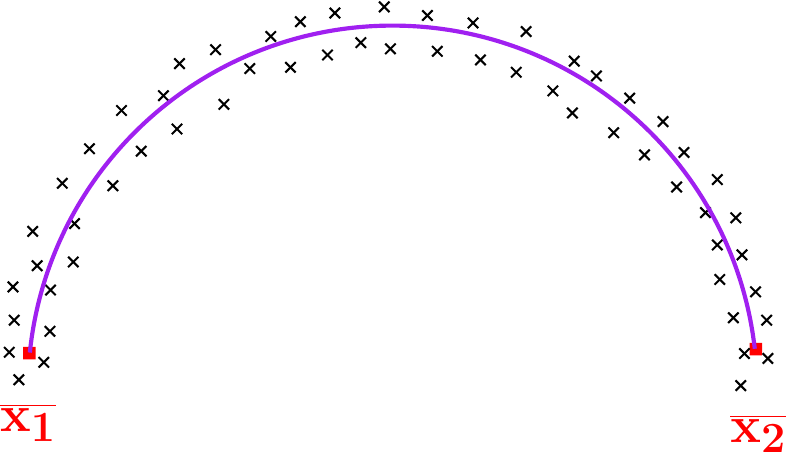}}
\subfigure[]{
\includegraphics[width=1.8in]{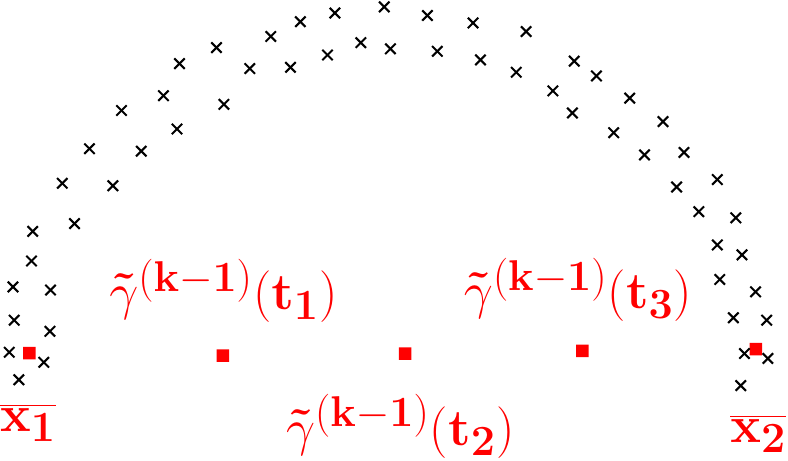}}
\subfigure[]{
\includegraphics[width=1.8in]{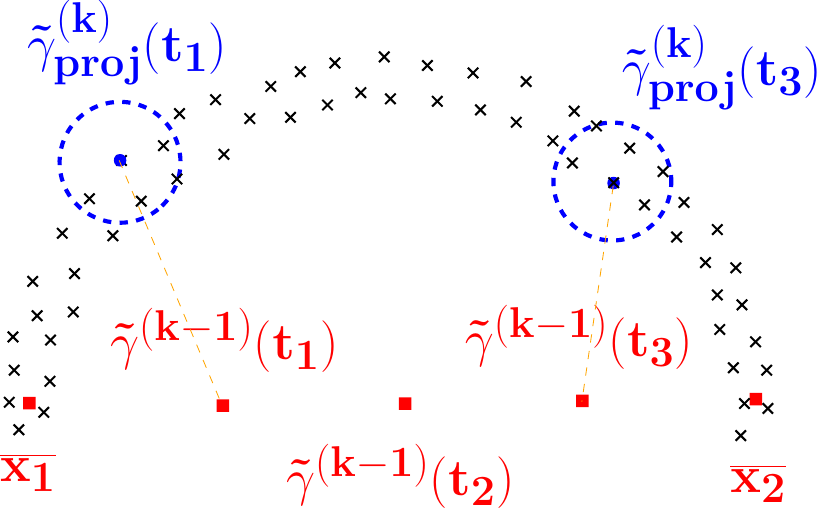}}
\subfigure[]{
\includegraphics[width=1.8in]{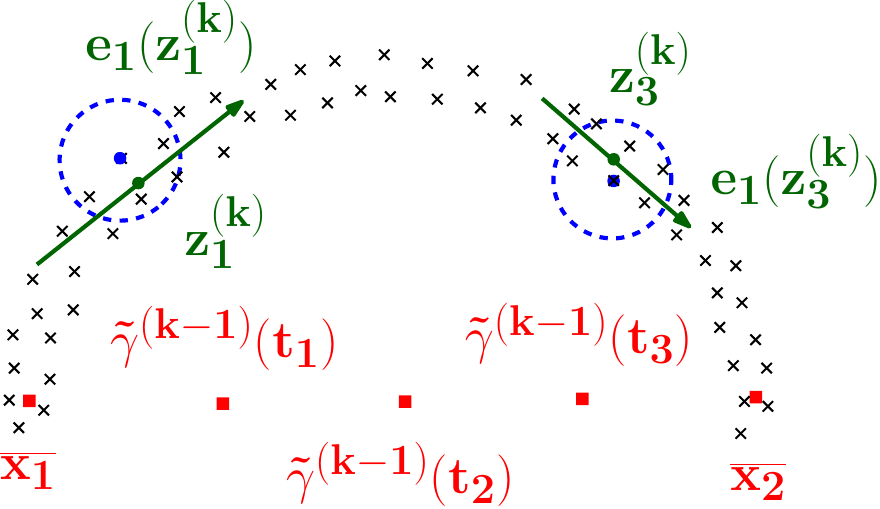}}
\subfigure[]{
\includegraphics[width=1.8in]{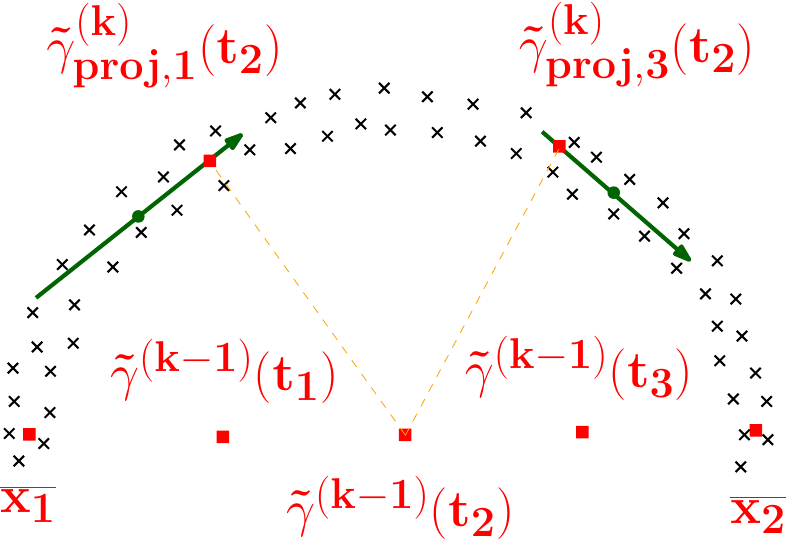}}
\subfigure[]{
\includegraphics[width=1.8in]{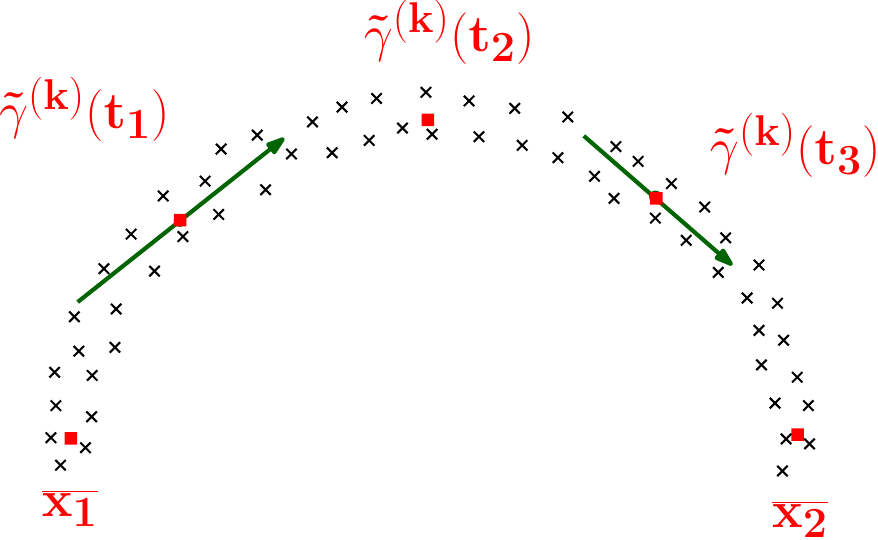}}
\end{center}
\caption{Determination of random fixed boundary flows. (a) Random data (black cross) with boundary points in red and the population flow in purple; (b) Flow $\tilde{\gamma}^{\sss (k-1)}(t_{i})$, $i=0,1,\ldots,4$ in red; (c) Projecting points $\tilde{\gamma}^{\sss (k-1)}(t_{2j+1})$ at $\tilde{\gamma}_{\text{proj}}^{\sss (k)}(t_{2j+1})$ (blue dots), $j=0,1$, and the local points are selected by the dotted circle in blue; (d) Calculating the local covariance matrix $\Sigma_{h^{\sss (k)}}(z^{\sss (k)}_{2j+1})$ at the mean points $z^{\sss (k)}_{2j+1}$ (green dots) and computing the first eigenvector $e_1(z^{\sss (k)}_{2j+1})$ (green arrows); (e) Projecting $\tilde{\gamma}^{\sss (k-1)}(t_{2})$ onto $e_1(z^{\sss (k)}_1)$ and $e_1(z^{\sss (k)}_3)$ at $\tilde{\gamma}_{\text{proj},1}^{\sss (k)}(t_{2})$ and $\tilde{\gamma}_{\text{proj},3}^{\sss (k)}(t_{2j})$ (red squares); (f) Updating $\tilde{\gamma}^{\sss (k)}(t_{i})$, $i=0,1,\ldots,4$, in red squares.
}\label{Fig:vector_field_update}
\end{figure}

\begin{itemize}
\item[(1)] \RVSD{\textbf{Choosing scale parameter:} we choose an appropriate scale parameter $h^{\sss (k)}=\rho^kh$, where $h \leq 1$ and $\rho\in (0,1]$ is a shrinkage constant. In our study, we let $\rho=0.9$. One may note that the shrinkage constant $\rho$ makes the scale parameter $h^{\sss (k)}$ decrease during the iterations. Hence, the scale parameter $h^{\sss (k)}$ guarantees the capture of the local variation. }

\item[(2)] \RVSD{\textbf{Calculating local covariance matrix:} the local covariance matrix is determined by using the discrete flow $\tilde{\gamma}^{\sss (k-1)}$ that we have obtained from the previous iteration. Specifically, we use the points $\tilde{\gamma}^{\sss (k-1)}(t_{2j+1})$, $j=0,1,\ldots,N-1$, with odd indices to calculate the local covariance matrix. Determining the local covariance matrix is a vital step for the following updating step of the discrete flow. 
We note that the points $\tilde{\gamma}^{\sss (k-1)}(t_{2j+1})$, $j=0,1,\ldots,N-1$, may not lie inside the data cloud. To capture the local variation accurately, we propose to project these points back inside the data cloud. To this aim, we first project these points to the nearest data points. As the nearest data points might be outliers, we further select the local data points within the distance of $h^{\sss (k)}$ from the nearest data points and obtain the mean points. Eventually, the projected points $\tilde{\gamma}_{\text{proj}}^{\sss (k)}(t_{2j+1})$, $j=0,\ldots,N-1$ are the nearest data points to the mean points. 
Then, the projected points $\tilde{\gamma}_{\text{proj}}^{\sss (k)}(t_{2j+1})$ are used to select the local data points to further compute the local covariance matrix. Denote by 
$\{y_l\}_{l=1}^{n_{2j+1, k}}$ the data points in the neighborhood of the Euclidean ball $B_d(\tilde{\gamma}_{\rm proj}^{{\sss (k)}}(t_{2j+1}),h^{{\sss (k)}})$ with center $\tilde{\gamma}_{\rm proj}^{{\sss (k)}}(t_{2j+1})$ and radius $h^{{\sss (k)}}$. 
Eventually, the local covariance matrix is computed at the mean  $z^{{\sss (k)}}_{2j+1}$ of the local data points $\{y_l\}_{l=1}^{n_{2j+1, k}}$ and can be calculated by 
\begin{equation}
\Sigma_{h^{\sss (k)}}(z^{{\sss (k)}}_{2i+1})=\frac{1}{n_{2j+1, k}}\sum_{l=1}^{n_{2j+1, k}}(y_l-z^{\sss (k)}_{2j+1})\otimes(y_l-z^{\sss (k)}_{2j+1}),
\end{equation}
where $a\otimes a=aa^T$. }

\item[(3)] \RVSD{\textbf{Determining vector field:} following on from the local covariance matrix $\Sigma_{h^{\sss (k)}}(z^{{\sss (k)}}_{2j+1})$ that we obtained in the previous step, the vector field is to determine at the mean points $z^{{\sss (k)}}_{2j+1}$. Denote by $W(\tilde{\gamma}^{\sss (k)}(t_{2j+1}))$ the vector field at point $\tilde{\gamma}^{\sss (k)}(t_{2j+1})$, $j=0,1,\ldots,N-1$. 
By now, the local variation is captured by the local covariance matrix $\Sigma_{h^{\sss (k)}}(z^{{\sss (k)}}_{2j+1})$. Hence, the direction along the first eigenvector $e_1(z^{{\sss (k)}}_{2j+1})$ shows the maximum variation. We let the vector field 
$W(\tilde{\gamma}^{\sss (k)}(t_{2j+1}))=e_1(z^{{\sss (k)}}_{2j+1})$.}

\item[(4)] \RVSD{\textbf{Updating:} as the boundary points are fixed, we first let $\gamma^{\sss (k)}(t_0)=\bar{x}_1$ and $\gamma^{\sss (k)}(t_{2N})=\bar{x}_2$. 
For the points with odd indices, we let $\tilde{\gamma}^{\sss (k+1)}(t_{2j+1})=z^{{\sss (k)}}_{2j+1}$, $j=0,1,\ldots,N-1$. 
The vector field $W(\tilde{\gamma}^{\sss (k)}(t_{2j+1}))$ is used to update the points with even indices. Specifically, we map the points $\tilde{\gamma}^{\sss (k-1)}(t_{2j})$ to the directions of the two adjacent vector fields $W(\tilde{\gamma}^{\sss (k)}(t_{2j-1}))$ and $W(\tilde{\gamma}^{\sss (k)}(t_{2j+1}))$ for $j=1,2,\ldots,N-1$. 
We denote the two projected points by $\tilde{\gamma}_{{\rm proj},2j-1}^{\sss (k)}(t_{2j})$ and $\tilde{\gamma}_{{\rm proj},2j+1}^{\sss (k)}(t_{2j})$. 
Hence, we update the points $\tilde{\gamma}^{\sss (k)}(t_{2j})$ by using the mean point of $\tilde{\gamma}_{{\rm proj},2j-1}^{\sss (k)}(t_{2j})$ and $\tilde{\gamma}_{{\rm proj},2j+1}^{\sss (k)}(t_{2j})$. }
\end{itemize}

\RVSD{
It is crucial to ensure that the random fixed boundary flow always moves along the direction that maximizes the vector field. Therefore, a stop criterion is necessary to the implementation. According, we terminate the iteration process when the optimization function $f(\tilde{\gamma}^{\sss (k)})$ does not change too much. 
Lastly, interpolation and projection will be implemented to ensure that the points $\tilde{\gamma}(t_i)$ on the resulting random fixed boundary flow are equidistant and lie on the manifold. The detailed algorithm is summarized in Algorithm \ref{RFBF algorithm}. The convergence of the random fixed boundary flow will be investigated in Section \ref{SEC:CONV_RFBF_ALGORITHM}. }

\begin{center}
\scalebox{0.7}{
\begin{algorithm}[H]
\RVSD{
\textbf{Input:} Starting point $\bar{x}_1$, ending point $\bar{x}_2$, resolution $N$, scale parameter $h \leq 1$, shrinkage constant $\rho\in (0,1]$, and stopping criterion $\epsilon$.}
\begin{itemize}
\item[1.] Choose an initial discrete flow $\tilde{\gamma}^{\sss (0)}(t_i)$, $i=0, 1,\cdots, 2N$, with $\tilde{\gamma}^{\sss (0)}(t_0) = \bar{x}_1$ and $\tilde{\gamma}^{\sss (0)}(t_{2N}) = \bar{x}_2$; 
\item[2.] Set $k=1$;
\begin{itemize}
\item[(1)] Choose scale parameter $h^{\sss (k)}=\rho^kh$;
\item[(2)] Calculate local covariance matrix $\Sigma_{h^{\sss (k)}}$ by using the points $\tilde{\gamma}^{{\sss (k-1)}}(t_{2j+1})$, $j=0,1,\ldots,N-1$. For each $j$,
\begin{itemize}
\item[(a)] project $\tilde{\gamma}^{{\sss (k-1)}}(t_{2j+1})$ to the data point at $\tilde{\gamma}_{\rm proj}^{{\sss (k)}}(t_{2j+1})$;
\item [(b)] select data points $\{y_l\}_{l=1}^{n_{2j+1,k}}$ in the neighborhood $B_d(\tilde{\gamma}_{\rm proj}^{{\sss (k)}}(t_{2j+1}),h^{{\sss (k)}})$, and compute the mean point $z^{{\sss (k)}}_{2j+1}$ of $\{y_l\}_{l=1}^{n_{2j+1,k}}$;
\item[(c)] 
using $\{y_l\}_{l=1}^{n_{2j+1,k}}$, compute the local covariance matrix $\Sigma_{h^{\sss (k)}}(z^{{\sss (k)}}_{2j+1})$ at $z^{{\sss (k)}}_{2j+1}$;
\end{itemize}
\item[(3)] Determine vector field $W(\tilde{\gamma}^{\sss (k)}(t_{2j+1}))=e_1(z^{{\sss (k)}}_{2j+1})$ for $j=0,1,\ldots,N-1$;
\item[(4)] Update the flow $\tilde{\gamma}^{{\sss (k)}}(t_{i})$, $i=0,1,\ldots,2N$;
\end{itemize}
\item[3.] If 
$|f(\tilde{\gamma}^{\sss (k)})-f(\tilde{\gamma}^{\sss (k-1)})|\le \epsilon$, output the discrete flow $\tilde{\gamma}^{\sss (k)}(t_i)$, $i=0,1,\cdots, 2N$. Otherwise, let $k \to k+1$ and go to step 2;
\item[4.] Output the discrete flow 
$\tilde{\gamma}(t_i)$, $i=0,1,\cdots, 2N$, after interpolation and mapping the discrete flow $\tilde{\gamma}^{\sss (k)}(t_i)$ onto the manifold. 
\end{itemize}
\caption{Obtaining Random Fixed Boundary Flow $\tilde{\gamma}$} \label{RFBF algorithm}
\end{algorithm}
}
\end{center}

\subsection{Convergence of the random fixed boundary flow} \label{SEC:CONV_RFBF_ALGORITHM}

In the following statement, $\I(x,h) = \{i:x_i \in B_d(x,h)\}$ and $|\I(x,h)|$ denotes the cardinality of $\I(x,h)$.
We use upper $C$, $C_0,  C_1, \cdots$ or lower $c, c_0, c_1, \cdots$ to denote constants greater or less than 1. Here, a constant means a value independent of $h$, $h^{{\sss (k)}}$ and $x$. Values of $C$ and $c$ with various subscripts may differ from line to line.

Recalling our Assumption \ref{assumption:noise}, samples are blurred by Gaussian noise.
Hence, by Gaussian concentration, the maximal distance between a point $x_i$ and $\gamma^\ast$ is bounded above by $\sigma(\sqrt{d} + \sqrt{\ln(n^C)})$ with probability at least $1-n^{-C}$. If $\sigma$ is sufficiently small such that
\[
 \sqrt{d} \leq 1/(2\sqrt{\sigma}), \quad  {\rm and} \quad n \leq \exp \{\frac{1}{4C\sigma}\},
\]
we can further bound the maximal distance between a point $x_i$ and $\gamma^\ast$ above by $\sqrt{\sigma}$, with probability $1-n^{-C}$, since the following holds
\begin{align}\label{bound:sample}
 \max_i \ d(x_i, \gamma^\ast) \leq \sigma(\sqrt{d}+\sqrt{\ln(n^C}) \leq \sigma( \frac{1}{2\sqrt{\sigma}} + \frac{1}{2\sqrt{\sigma}}) \leq \sqrt{\sigma}.
\end{align}
This inequality above shows that the samples mainly lie in the tube $T_* = \{x : d(x,\gamma^\ast) \leq \sqrt{\sigma}\}$ surrounding $\gamma^\ast$. Considering a point $x$ satisfying $d(x,\gamma^\ast) \leq \epsilon$ with $\epsilon > \sqrt{\sigma}$, the intersection $ B_d(x, 2\epsilon) \cap T_* $ cannot be ignored, hence the following assumption $B_d(x, 2\epsilon) \cap X \neq \emptyset$ holds true. 

\begin{assumption} \label{assumption:sample}
  For any $\epsilon > \sqrt{\sigma}$, if $x$ satisfies $d(x,\gamma^\ast) \leq \epsilon$, then $ B_d(x, 2\epsilon) \cap X \neq \emptyset$.
\end{assumption}

Note that Step 3(a) of Algorithm \ref{RFBF algorithm} projects points to the data cloud by finding its nearest samples in $X$. Assumption \ref{assumption:sample} bounds the distance between the given point and the projected point above, which essentially leads to the convergence.
Algorithm \ref{RFBF algorithm} selects decreasing scales $h^{{\sss (k)}} = \rho h^{{\sss (k-1)}} $ with a given $\rho \in (0,1]$ in each iteration, until the scale is less than $4\sqrt{\sigma}$ or the objective function hardly changes. Each iteration takes the output discrete flow of the previous iteration as input, updates the vector field with a smaller scale and outputs a discrete flow using the \Yao{updated} vector field. \Xia{Theorems \ref{THM:K ITERATION} - \ref{prop:dist_xonM} with full proofs in Appendix C in the Supplementary Materials, together prove that the random fixed boundary flow converges to the population flow $\gamma^{\ast}$, given certain conditions of the initial discrete flow.}

Specifically, Theorem \ref{THM:K ITERATION} exploits the $k$-th iteration and bounds $d_H(\tilde{\gamma}^{\sss (k+1)}, \gamma^{\ast})$ above when (a) its input discrete flow is sufficiently close to $\gamma^{
\ast}$, (b) the points in the discrete flow are sufficiently dense, and (c) the points with odd indices are not too close to the two ends of the population flow. Note that (c) is needed since the vector field near the two ends does not follow the population flow. This means that the fixed boundaries $\bar{x}_1$ and $\bar{x}_2$ should be chosen not too close to the two ends, $\gamma^\ast(0)$ and  $\gamma^\ast(r^*)$, in practice. Theorem \ref{THM:GREEDY ALGORITHM} proves that imposing constraints on the initial discrete flow, that is the input discrete flow for $k=0$, also leads to the upper bound of $d_H(\tilde{\gamma}^{\sss (k+1)}, \gamma^{\ast})$. Theorem \ref{prop:dist_xonM} proves the convergence of the random fixed boundary flow, as the projection of $\tilde{\gamma}^{\sss (K)}$ onto $\M$. 

\begin{thm} \label{THM:K ITERATION}
\RVSD{
Suppose the discrete curve at the $k$-th iteration satisfies the following conditions:
\begin{itemize}
    \item[(a)] 
    $d_H(\{\tilde{\gamma}^{{\sss (k)}}(t_{i})\}_{i=0}^{2N^{\sss (k)}},\gamma^\ast) \leq C_1 h^{{\sss (k)}}$,
    \item[(b)] 
    $\|\tilde{\gamma}^{{\sss (k)}}(t_{i+1}) - \tilde{\gamma}^{{\sss (k)}}(t_i)\| \leq C_2 h^{{\sss (k)}}$ for any $i<2N^{{\sss (k)}}$, 
    \item[(c)] $\|\tilde{\gamma}^{{\sss (k)}}(t_{2j+1})-\gamma^\ast(0)\|\geq (2C_1+\Xia{3.5})h^{{\sss (k)}}$ and $\|\tilde{\gamma}^{{\sss (k)}}(t_{2j+1})-\gamma^\ast(r^*)\|\geq (2C_1+\Xia{3.5})h^{{\sss (k)}}$  for any $j = 0, 1, \cdots, N^{{\sss (k)}}-1$,
\end{itemize}
For any given $\delta > 0$, there exists $C$ such that any point in the polyline
\begin{align*}
\tilde{\gamma}^{\sss (k+1)} = & \{ s\big( \gamma_{{\rm proj},2j-1}^{{\sss (k)}}(t_{2j}),\gamma_{{\rm proj},2j+1}^{{\sss (k)}}(t_{2j}) \big) \}_{j =1}^{N^{{\sss (k)}}-1} 
\ \cup \ s\big(\bar{x}_1, \gamma_{{\rm proj},1}^{{\sss (k)}}(t_0) \big) \\
& \ \cup \ s\big(\gamma_{{\rm proj},2N^{{\sss (k)}}-1}^{{\sss (k)}}(t_{2N^{{\sss (k)}}}) , \bar{x}_2 \big) 
\ \cup \ \{ s\big( \gamma_{{\rm proj},2j+1}^{{\sss (k)}}(t_{2j}),\gamma_{{\rm proj},2j+1}^{{\sss (k)}}(t_{2j+2}) \big) \}_{j=0}^{N^{{\sss (k)}}-1} 
\end{align*}
is also within Hausdorff distance $C_\delta{h^{{\sss (k)}}}^2$ to $\gamma^\ast$ with probability $1-\delta$. 
}
\end{thm}

\RVSD{
We only sketch the proof of Theorem \ref{THM:K ITERATION}. Recalling Algorithm \ref{RFBF algorithm}, the polyline $\tilde{\gamma}^{\sss (k+1)}$ is composed of segments passing $\{\tilde{\gamma}^{\sss (k+1)}(t_{2j+1})\}_{j=0}^{N^{\sss (k)}-1}$ and along $\{W(\tilde{\gamma}^{\sss (k+1)}(t_{2j+1}))\}_{j=0}^{N^{\sss (k)}-1}$. 
Lemma \Xia{3.1 in the Supplementary Materials} proves $\{\tilde{\gamma}^{\sss (k+1)}(t_{2j+1})\}_{j=0}^{N^{\sss (k)}-1}$ are within Hausdorff distance $O({h^{{\sss (k)}}}^2)$ to $\gamma^\ast$. Lemma \Xia{3.2} proves the vector field at $\tilde{\gamma}^{\sss (k+1)}(t_{2j+1})$ approximate the tangent vector of $\gamma^\ast$ in the order of $h^{\sss (k)}$. Based on these, Lemma \Xia{3.3 in the Supplementary Materials} proves that the segments which pass $\{\tilde{\gamma}^{\sss (k+1)}(t_{2j+1})\}_{j=0}^{N^{\sss (k)}-1}$ along $\{W(\tilde{\gamma}^{\sss (k+1)}(t_{2j+1}))\}_{j=0}^{N^{\sss (k)}-1}$ approximates $\gamma^\ast$ in the order of ${h^{\sss (k)}}^2$. Hence, the polyline $\tilde{\gamma}^{\sss (k+1)}$ which is composed of these segments is also within Hausdorff distance $O({h^{{\sss (k)}}}^2)$ to $\gamma^\ast$. 
}

\begin{thm}\label{THM:GREEDY ALGORITHM}
\RVSD{
If the conditions (a)-(c) in Theorem \ref{THM:K ITERATION} hold for $k=0$ and the constants in Theorem \ref{THM:K ITERATION} further satisfies
\begin{align*}
C_\delta h^{\sss (0)} \leq C_1 \rho, \quad  \|\gamma^\ast(0) - \gamma^\ast(r^*)\| > (4C_1+7)h^{{\sss (0)}},  \quad C_2 > 4C_1+7
\end{align*}
then for any given $k > 0$, 
 $ d_H(\tilde{\gamma}^{\sss (k+1)}, \gamma^\ast) \leq C{h^{{\sss (k)}}}^2$ with probability $(1-\delta)^{k+1}$.
 }
\end{thm}

\RVSD{
According to the stopping criteria, $h^{{\sss (K)}} = O(\sqrt{\sigma})$ when Algorithm \ref{RFBF algorithm} stops. Hence, the final polyline $\tilde{\gamma}^{\sss (K)}$ satisfies
\[
d_H(\tilde{\gamma}^{\sss (K)}, \gamma^\ast) = O({h^{{\sss (k)}}}^2) = O((\sqrt{\sigma})^2) = O(\sigma).
\]
Note that the interpolation step generates a discrete curve containing $\tilde{\gamma}^{\sss (K)}$, and the projection step will not  change the order of the Hausdorff distance as Theorem \ref{prop:dist_xonM} has proved. To be precise, the final discrete curve of Algorithm \ref{RFBF algorithm} is located in a tube along the population curve $\gamma^\ast$ with a radius in order of $\sqrt{\sigma}$. }

\begin{thm}\label{prop:dist_xonM}
\RVSD{
  If $d(x, \gamma^\ast) = O(h)$, then $d(\tilde{x}, \gamma^\ast) = O(h)$, where $\tilde{x}$ is the projection of $x$ onto $\M$.
}
\end{thm}

}

\section{Simulations}\label{SEC:simulation}

\RVSD{
To illustrate the performance of random fixed boundary flows, we studied several random data sets generated on two manifolds, a unit sphere and a right-circular unit cone. The two manifolds are in $\mathbb{R}^3$ with intrinsic dimension $d=2$. In the simulation, the boundary points were selected manually from the given data set so that there are enough data points around the boundary points to calculate the local variation. 
To generate the random fixed boundary flows, we applied the proposed algorithm with different values of the scale parameter $h$. 
Here, we note that the random fixed boundary flow is a discrete curve with derivatives that approximately capture the direction of the maximum local variation depending on $h$. Throughout the numerical studies in sections \ref{SEC:simulation} and \ref{SEC:realdata}, we use RFBFs to denote random fixed boundary flows.}

In the first part of the simulation, we evaluate the performance of the RFBFs on the unit sphere. 
The noisy data sets are randomly generated from three population flows, which are plotted in purple in Figure \ref{Fig:Simulateddata} (a)-(c). Specifically, Gaussian noise is added to the points on the population flows with a constraint such that the perturbed points remain on the test manifold. In this manner, we generated three noisy data sets, each representing different types of variation on the unit sphere.
The first data set is concentrated around a ``C''-shaped curve on the unit sphere, thus presenting a variation pattern along the geodesic.
After that, we considered two data sets from two non-convex closed flows. In this setting, the simulated data sets present local variation patterns along the non-convex flows.  In particular, the second data set is generated from a quarter of the six-fold star-shaped flow, and the third data set is concentrated around a half of the two-fold star-shaped flow.

\RVSD{
To obtain RFBFs, the initial flows used in our analysis are straight lines connecting $\bar{x}_1$ and $\bar{x}_2$. One may use other initial flows, for example, the geodesic from $\bar{x}_1$ to $\bar{x}_2$. 
Given a set of randomly generated data, we obtained a RFBF with a specific $h$.  For the data sets plotted in Figure \ref{Fig:Simulateddata} (a)-(c), the RFBFs obtained with a specific value of $h$ are illustrated in red in Figure \ref{Fig:Simulateddata} (d)-(f). 
To further investigate the performance, we obtained ten sets of random data for each population flow. The RFBFs are then obtained with a sequence of $h$ for the random data. 
An analysis of the mean errors for the Hausdorff distances between the population flow and RFBFs are summarized in Table \ref{Tab:mean_error}. From the numerical results, we note that the RFBFs are able to capture the variation globally and locally. As we lower $h$, the performance accuracy of the RFBFs improves generally. On the other hand, overfitting may occur as we lower $h$ gradually. }

\RVSD{
For the two non-convex population flows in Figure \ref{Fig:Simulateddata} (b)-(c), we also generated noisy data sets from the whole closed flows. As boundary points are required to obtain RFBFs, we handled these noisy data sets parts by parts. For example, we fitted the noisy six-fold data set quarter by quarter and the noisy two-fold data set half by half. We specified the boundary points for each part of the whole data set and obtained the RFBFs with predetermined values of $h$. The obtained RFBFs are shown in red in Figure \ref{Fig:whole_curve_data}. To compare the performance accuracy, we further applied the level set methods in \citet{liu2017level} to the random data sets and plotted the obtained curves in blue in Figure \ref{Fig:whole_curve_data}. 
In contrast to the level set methods, the RFBFs are able to capture the local variation better, especially at the parts of the curves with high curvature. We also note that the level curve methods reach the locations outside the data cloud at some parts of the two-fold data. }

\begin{figure}[h]
\begin{center}
\subfigure[noisy ``C"-shaped data]{
\includegraphics[width=1.5in]{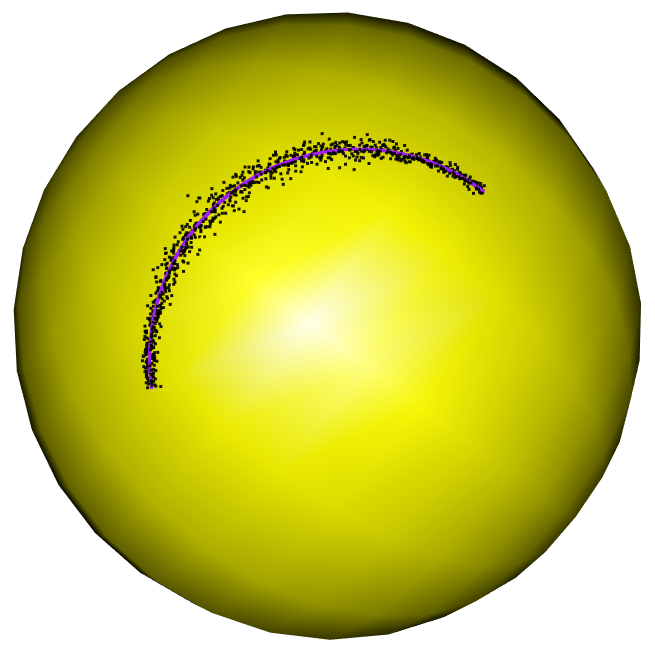}}
\subfigure[noisy six-fold data]{
\includegraphics[width=1.5in]{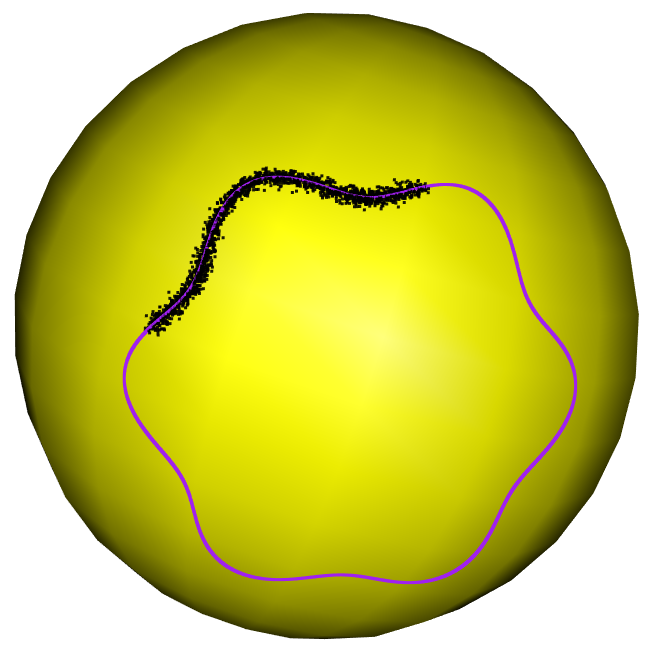}}
\subfigure[noisy two-fold data]{
\includegraphics[width=1.5in]{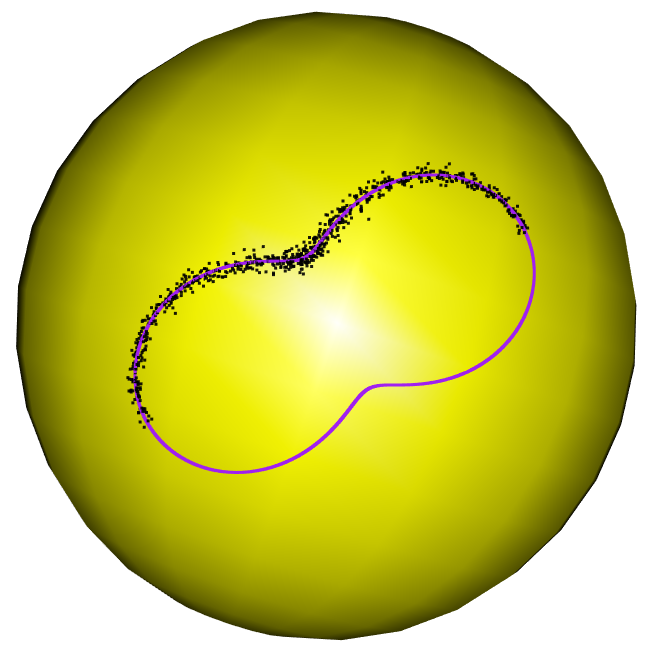}}\\
\subfigure[$h=0.08$]{
\includegraphics[width=1.5in]{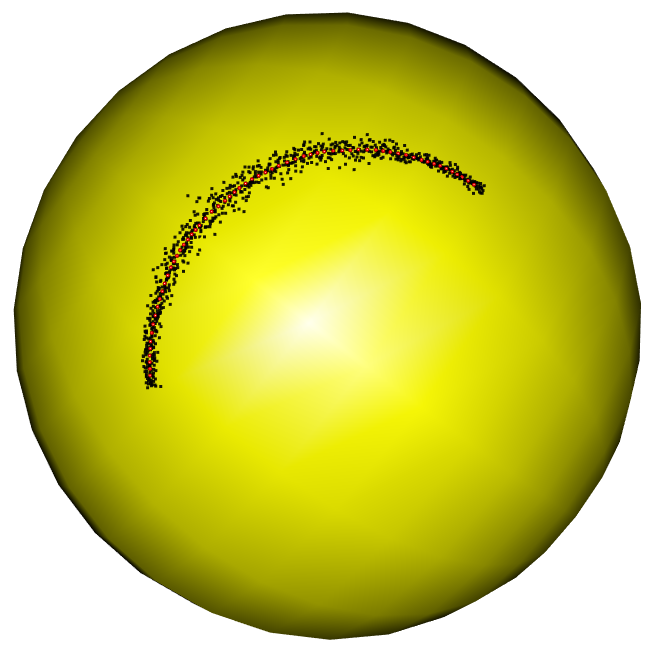}}
\subfigure[$h=0.06$]{
\includegraphics[width=1.5in]{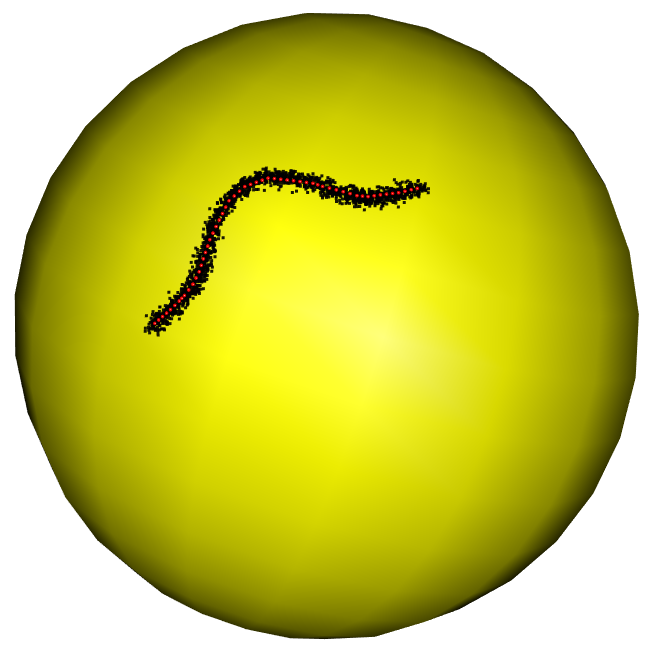}}
\subfigure[$h=0.08$]{
\includegraphics[width=1.5in]{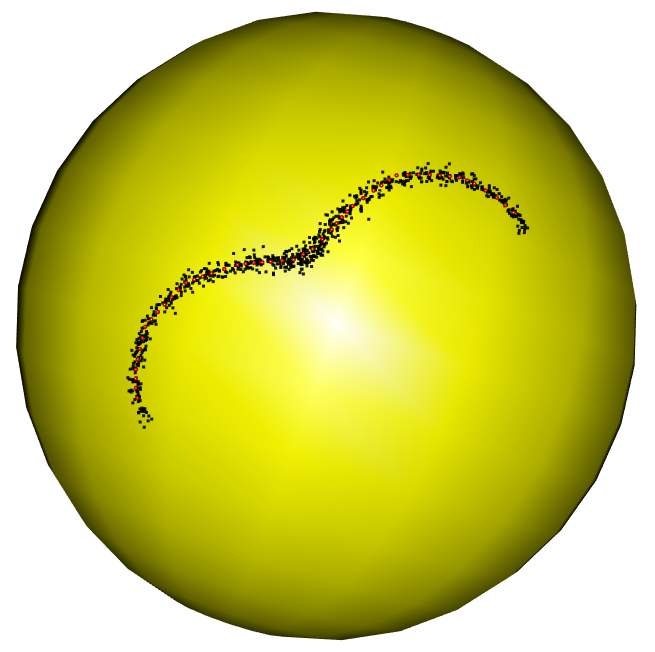}}
\end{center}
\caption{RFBFs (in red) for three random data sets on the unit sphere with the underlying population flows (in purple). 
}\label{Fig:Simulateddata}
\end{figure}

\begin{table}[h]
\centering
\scalebox{0.65}{
\begin{tabular}{c c c c c c| c c c c c c c }
\hline
\multicolumn{6}{c}{Unit Sphere}  & \multicolumn{6}{c}{Right-circular unit cone}\\\hline\hline
$h$ & noisy ``C"-shaped & $h$ &noisy six-fold  &  $h$ &noisy two-fold  & $h$ & noisy band  & $h$ & noisy ``C'' shape  & $h$ & noisy ``S'' shape\\
 \hline
 $0.06$ & 0.0074(0.0017) & $0.04$ & 0.0093(0.0022) & $0.06$ & 0.0122(0.0025)  & 0.12 & 0.0095(0.0021) & 0.06 & 0.0080(0.0018) &  0.06 & 0.0101(0.0025)\\
 $0.08$ & 0.0063(0.0010) & $0.06$ & 0.0090(0.0009) &  $0.08$ & 0.0110(0.0026) & 0.14 & 0.0088(0.0012) & 0.08 & 0.0066(0.0013) & 0.08 & 0.0095(0.0013)\\
 $0.10$ & 0.0063(0.0014) & $0.08$ & 0.0121(0.0012) &  $0.10$ & 0.0133(0.0011) &  0.16 & 0.0094(0.0009) & 0.10 & 0.0067(0.0012) & 0.10 & 0.0126(0.0014) \\
 $0.12$ & 0.0073(0.0011) & $0.10$ & 0.0180(0.0059) &  $0.12$ & 0.0164(0.0009) &  0.18 & 0.0103(0.0009) & 0.12 & 0.0077(0.0009) & 0.12 & 0.0167(0.0013)\\
 $0.14$ & 0.0078(0.0009) & $0.12$ & 0.0257(0.0071) &  $0.14$ & 0.0204(0.0011) &  0.20 & 0.0122(0.0007) & 0.14 & 0.0099(0.0009) & 0.14 & 0.0215(0.0019)\\
 \hline
\end{tabular}}
\caption{Comparison of the mean errors for the Hausdorff distances $d_H(\tilde{\gamma}, \gamma^*)$ with five different values of $h$. Standard deviations based on ten sets of randomly generated data are shown in parentheses.}\label{Tab:mean_error}
\end{table}

\begin{figure}[h]
\begin{center}
\subfigure[noisy six-fold data]{
\includegraphics[width=1.5in]{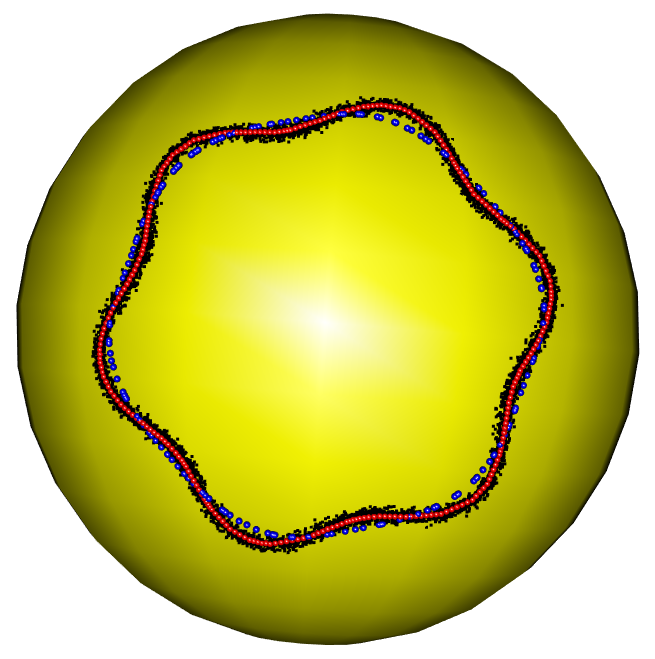}}
\subfigure[noisy two-fold data]{\includegraphics[width=1.5in]{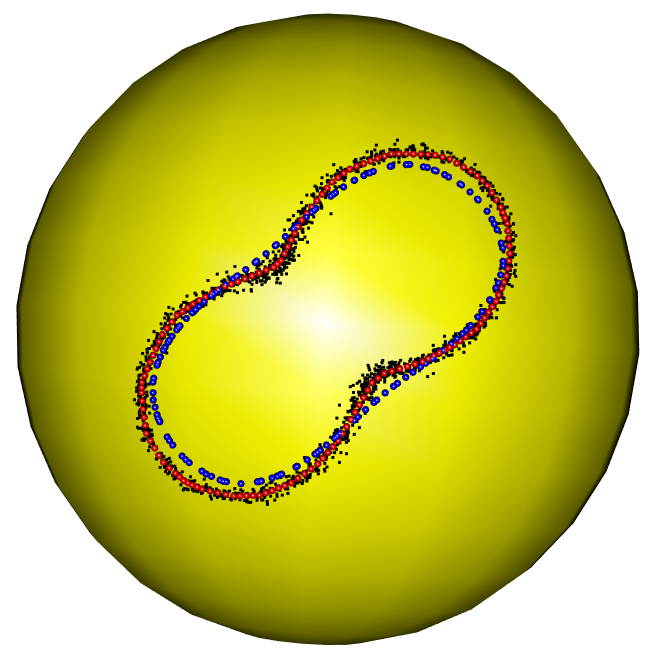}
}
\end{center}
\caption{Comparison of RFBFs and the level curve methods on the unit sphere. Black dots: data points; red dots: RFBFs; blue dots: curves obtained from the level curve methods.
}\label{Fig:whole_curve_data}
\end{figure}

\RVSD{
In the second part of the simulation study, the testing manifold is extended to a right-circular unit cone, with apex at $(0,0,0)$, height $H=1$ and radius $R=1$. Three types of random data sets are generated to examine the performance of RFBFs on the right-circular unit cone. The first data set is concentrated around a band on the cone. For the second and third data sets, they are generated from a ``C''-shaped and ``S''-shaped population flows on the tested manifold. The RFBFs with a predetermined value of $h$ are illustrated in red in Figure \ref{Fig:cone_data}. As the data plotted shown, we observe that the RFBFs work well to capture different types of variations on the cone. 
Similarly, we fitted RFBFs with a sequence of $h$ for ten randomly generated data sets. To examine the performance accuracy, the mean errors of the Hausdorff distances between the population flows and RFBFs are summarized in Table \ref{Tab:mean_error}. As expected, the obtained RFBFs do indeed divine the variation accurately on the cone as we lower $h$. It becomes more challenging to capture the variation accurately for all three types of variation investigated when the variation pattern becomes more complicated. }

\begin{figure}[h]
\begin{center}
\subfigure[noisy band data]{
\includegraphics[width=1.5in]{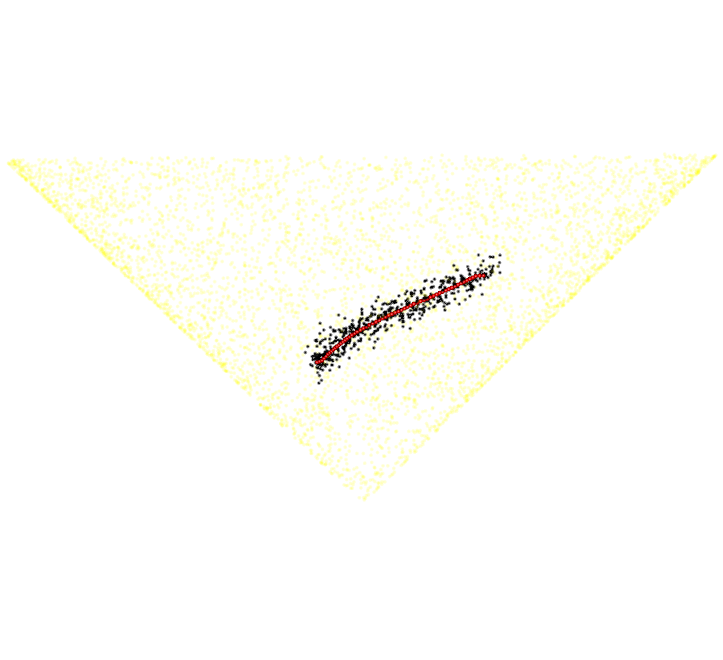}}
\subfigure[noisy ``C'' shape data]{
\includegraphics[width=1.5in]{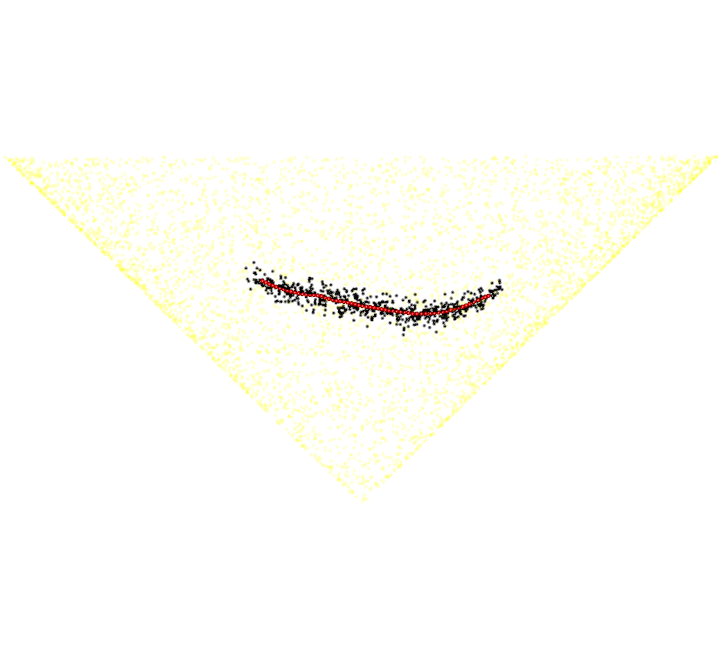}}
\subfigure[noisy ``S'' shape data]{
\includegraphics[width=1.5in]{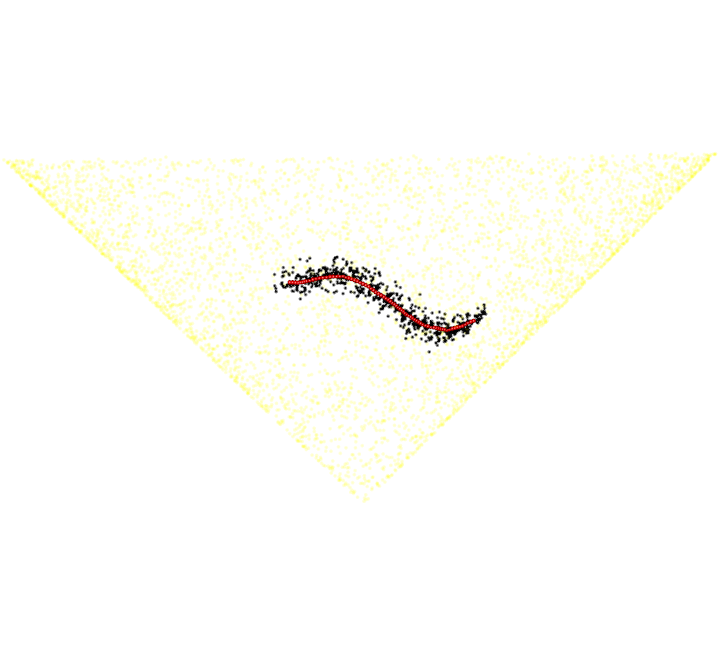}}
\end{center}
\caption{RFBFs (in red) for three different types of random data sets on the right-circular unit cone. Values of $h$ used: (a) $h=0.14$; (b) $h=0.08$; (c) $h=0.08$. }\label{Fig:cone_data}
\end{figure}

\section{Real Data Application}\label{SEC:realdata}

\subsection{Seismological Data}
\RVSD{
Here we explain the full analysis of the previously mentioned seismological events. The data set was sourced from the International Seismological Center (ISC) and features significant earthquakes (magnitude $5.5$ in Richter scale and above, including continental events of magnitude $5.0$) between 1904 and 2015. The earthquake epicentre data is plotted in black in Figure \ref{Fig:earthquake_eigenvalue}. Before we fit the RFBF for the earthquake data, we first investigate the distribution of the first eigenvalue for the data. This is shown in Figure \ref{Fig:earthquake_eigenvalue}, from which we observe that the variation of the first eigenvalue among the earthquake epicentres along the distribution of the earthquakes is quite non-uniform. Furthermore, we also observe that the first eigenvalue changes with different values of $h$, which changes the determination of local variation. Hence, the analysis of seismological events is an example with a varying first eigenvalue and we will investigate the performance of RFBF for this case. }


\begin{figure}[h]
\begin{center}
\subfigure[$h_1$]{
\includegraphics[width=1.5in]{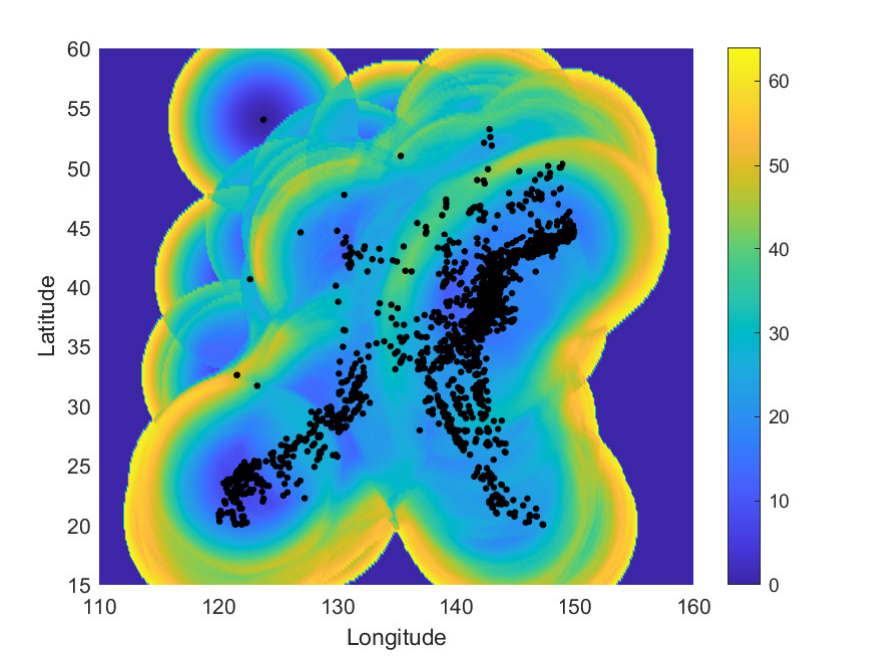}}
\subfigure[$h_2$]{
\includegraphics[width=1.5in]{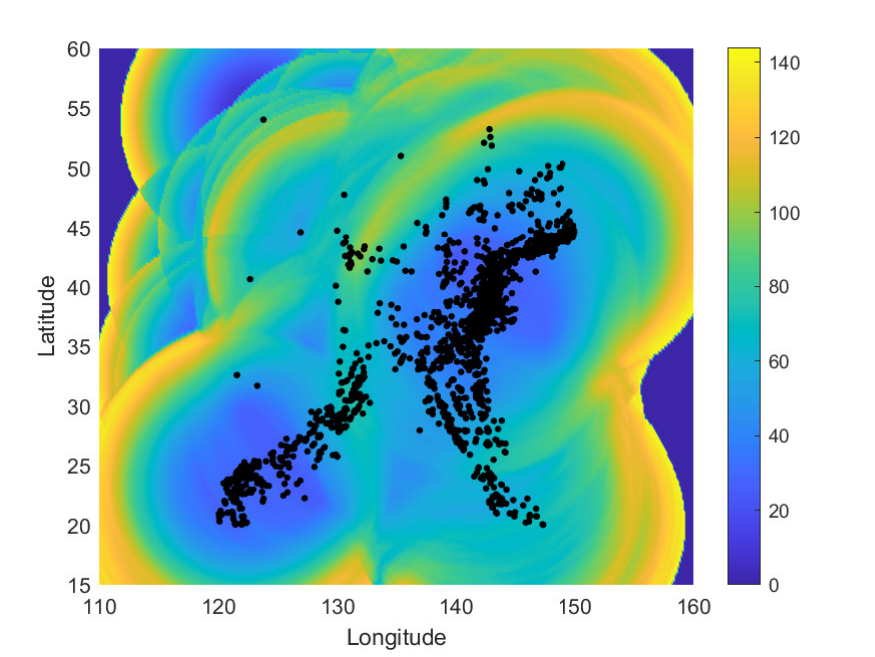}}
\subfigure[$h_3$]{
\includegraphics[width=1.5in]{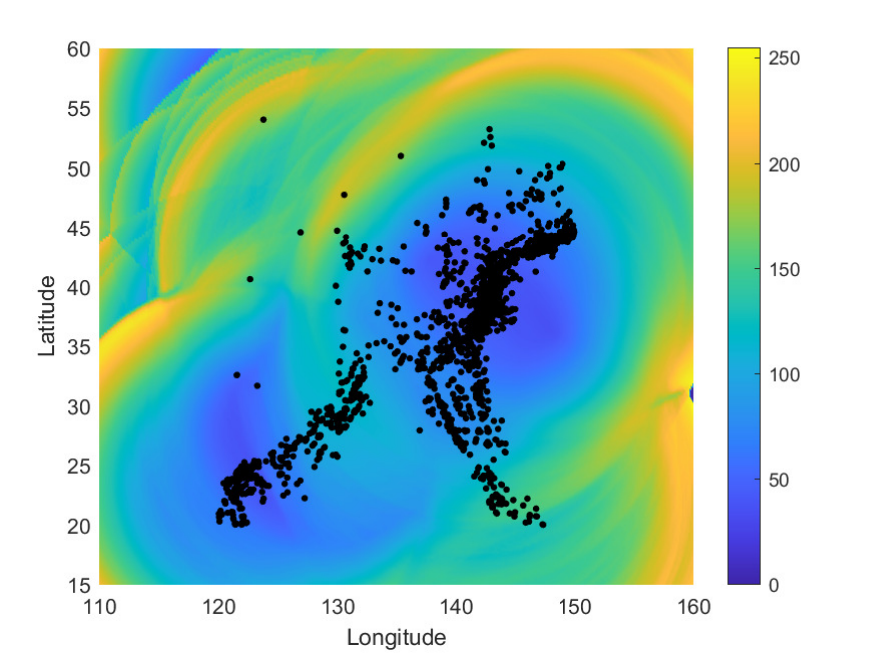}}
\end{center}
\caption{Distribution of the first eigenvalue $\lambda_1$ for the seismological events with various $h$, where $h_1<h_2<h_3$. Black dots: earthquake epicentres; the color bars represent the magnitude of $\lambda_1$.
}\label{Fig:earthquake_eigenvalue}
\end{figure}

\RVSD{
We note that earthquakes tend to occur around the tectonic plate boundaries. As has been mentioned earlier, the shape of the plate boundaries shown in Figure \ref{Fig:plate}(a) carries the global variation (from east to west, or north to south) and the localized variation along different plates. 
If we select $\bar{x}_1$ and $\bar{x}_2$ around the Philippine Sea plate manually, we expect the RFBFs would move along the plate boundary and mirror the blue curve shown in Figure \ref{Fig:plate}(c). At the same time, the movement of the RFBFs will also reflect the local variation pattern of the data, which is captured by $h$. 
In our analysis, we scaled the data onto the unit sphere and selected three different sets of $\bar{x}_1$ and $\bar{x}_2$ along the Philippine Sea plate manually. Figure \ref{Fig:FBF_earthquake} illustrates the earthquake data on a flat world atlas with the three sets of $\bar{x}_1$ and $\bar{x}_2$, namely (a)-(c), (d)-(f) and (g)-(i). 
To visualise and compare the performance, we fit RFBFs using three values of $h$. 
As we expected, the RFBFs move along the boundary of the Philippine Sea plate and capture the variation between the given boundary points. Furthermore, we let $h$ vary and visualize the RFBFs that reflect the various localized variation patterns. 
Given the boundary points, we note that the RFBFs work well in capturing the variation patterns of the data. As we lower $h$, the RFBFs uncover the global and local variation pattern more accurately. For example, when we set $h=0.075$, the RFBFs in Figure \ref{Fig:FBF_earthquake} (a), (d) and (g) move inside the data cloud and trace the global variation from south to north better than the RFBFs in the other plots of Figure \ref{Fig:FBF_earthquake}. When we gradually increase the value of $h$, more data points will be involved in the determination of the local variation and this also influences the trend of the RFBFs. In the last three plots of Figure \ref{Fig:FBF_earthquake}, we select two sets of boundary points with opposite directions. Comparing the results in Figure \ref{Fig:FBF_earthquake} (a)-(c) and (g)-(i), we note that the direction of the boundary points does not inordinately affect the RFBFs with the same $h$. }

\begin{figure}[h]
\begin{center}
\subfigure[$h=0.075$]{
\includegraphics[width=1.2in]{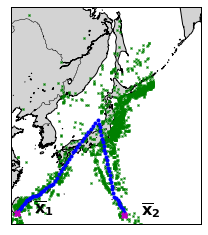}}
\subfigure[$h=0.15$]{
\includegraphics[width=1.2in]{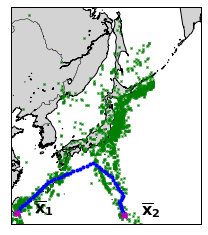}}
\subfigure[$h=0.226$]{
\includegraphics[width=1.2in]{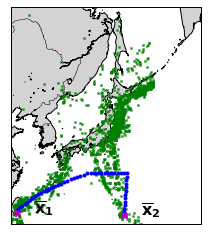}}
\subfigure[$h=0.075$]{
\includegraphics[width=1.2in]{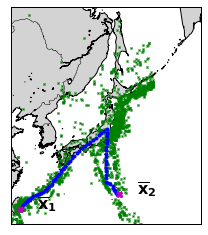}}\\
\subfigure[$h=0.15$]{
\includegraphics[width=1.2in]{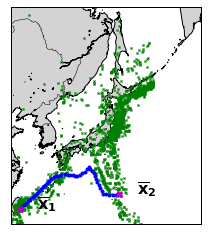}}
\subfigure[$h=0.226$]{
\includegraphics[width=1.2in]{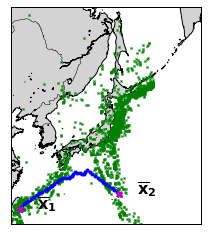}}
\subfigure[$h=0.075$]{
\includegraphics[width=1.2in]{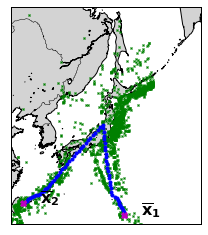}}
\subfigure[$h=0.15$]{
\includegraphics[width=1.2in]{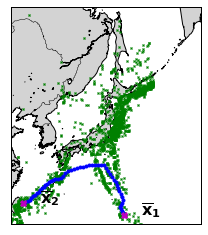}}
\subfigure[$h=0.226$]{
\includegraphics[width=1.2in]{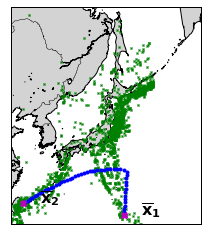}}
\end{center}
\caption{RFBFs for the seismological events with three different sets of boundary points. Green dots: earthquake epicentres; purple crosses: boundary points; blue dots: RFBFs.  
}\label{Fig:FBF_earthquake}
\end{figure}


\subsection{Labeled Faces in the Wild}
In this section, we consider another concrete case -- Labeled Faces in the Wild (LFW) in \citet{LFWTech}. 
The data set comprising face photographs is designed to provide a system of face recognition with over $13,000$ images of faces collected from the web. Each face image is labeled with the name of the person in that image. Note also that among those face images are $1,680$ people who have two or more distinct photographs in the data set.

\RVSD{
In our study, we downloaded $264$ images of $66$ people with four images of each person. To facilitate the analysis, the face region was cropped from the original image and resized to $50\times 37$ pixels. The images of the face region for the $66$ individuals can be found in the Supplementary Materials. As the analysis uses four different images for each individual, the data set can be written as $\{x_i\}_{i=1}^{264}$, where $x_i$ are vectors in the ambient space $\mathbb{R}^{1850}$. We assume that the data points $\{x_i\}_{i=1}^{264}$ lie on the unit sphere $\mathbb{S}^{1849}$, which is embedded in $\mathbb{R}^{1850}$. 
To begin with, we chose two images with the largest distance from one another in the ambient space, setting them as $\bar{x}_1$ and $\bar{x}_2$. As shown in Figure \ref{Fig:face_changing}, the image of Andy Roddick in Figure \ref{Fig:face_changing}(a) is the starting image, and the image of Jack Straw in Figure \ref{Fig:face_changing}(p) is the ending image in our analysis. Then, we fit RFBFs with various values of $h$. }

The obtained RFBFs are discrete flows of face images which capture the variation of facial structure from the starting image to the ending image. 
With the exception of the boundary images on the RFBFs, we generated a sequence of fake faces, which are plotted in Figure \ref{Fig:face_changing} (b)-(o). 
The person plotted in each fake face image is not a real person that can be identified in the given image set. On the contrary, the person is constructed using the characteristics extracted from the local and global variation pattern of the given images. 
The intermediate face images on the RFBF reflect the progressive face changing from the starting image to the ending image. 
There are some noteworthy conclusions that we draw from the RFBFs. 
First, the skin tone of Andy Roddick shown in the starting image of Figure \ref{Fig:face_changing} (a) appears somewhat wheatish, while Jack Straw's face, plotted in the ending image of Figure \ref{Fig:face_changing} (p), possesses a light skin tone. Through the fake faces constructed on the RFBF, we are able to observe the gradual changes of skin tone from dark to light. 
Second, we note that the Andy Roddick dons a cap in the starting image and Jack Straw’s hairstyle features a fringe in the ending image. For the first few images in Figure \ref{Fig:face_changing} (b)-(f), the fake faces on the RFBF are also wearing caps. In the last few images plotted in Figure \ref{Fig:face_changing} (m)-(o), the fake faces of the RFBF have fringes. Hence, we are also able to monitor the change of hairstyle through the intermediate fake faces on the RFBF.  
Although RFBFs are able to reveal some progressive face changes, the characteristics captured by the RFBFs are one-dimensional. Hence, the variation pattern analyzed by the RFBFs is limited when we are dealing with high-dimensional data with large $m$ values. We will consider the extension of RFBFs in the future. 

\begin{figure}[h]
\begin{center}
\subfigure[]{
\includegraphics[width=0.4in]{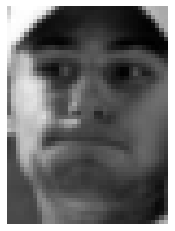}}
\subfigure[]{
\includegraphics[width=0.4in]{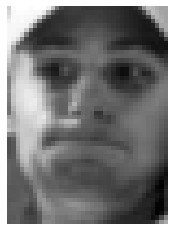}}
\subfigure[]{
\includegraphics[width=0.4in]{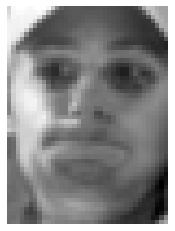}}
\subfigure[]{
\includegraphics[width=0.4in]{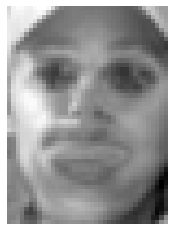}}
\subfigure[]{
\includegraphics[width=0.4in]{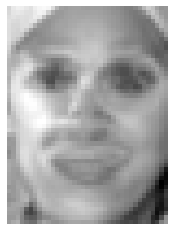}}
\subfigure[]{
\includegraphics[width=0.4in]{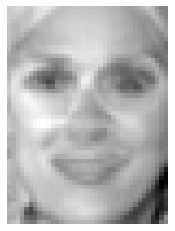}}
\subfigure[]{
\includegraphics[width=0.4in]{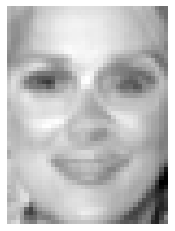}}
\subfigure[]{
\includegraphics[width=0.4in]{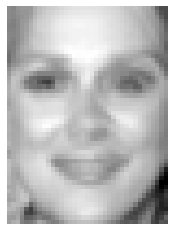}}
\subfigure[]{
\includegraphics[width=0.4in]{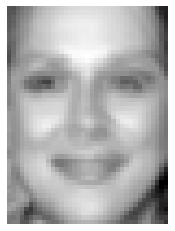}}
\subfigure[]{
\includegraphics[width=0.4in]{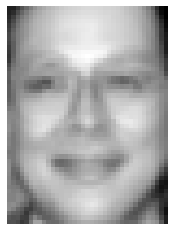}}
\subfigure[]{
\includegraphics[width=0.4in]{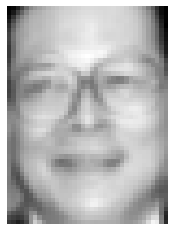}}
\subfigure[]{
\includegraphics[width=0.4in]{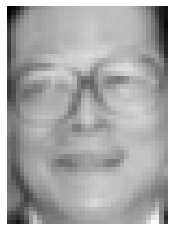}}
\subfigure[]{
\includegraphics[width=0.4in]{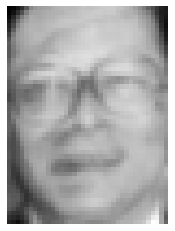}}
\subfigure[]{
\includegraphics[width=0.4in]{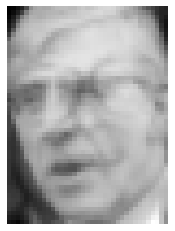}}
\subfigure[]{
\includegraphics[width=0.4in]{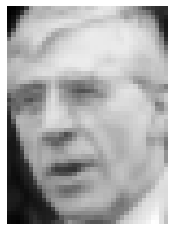}}
\subfigure[]{
\includegraphics[width=0.4in]{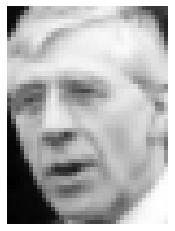}}
\end{center}
\caption{Face images obtained from the RFBFs. Image (a): starting face image; image (p): ending face image; image (b)-(o): fake face images generated by the RFBF.}\label{Fig:face_changing}
\end{figure}


\section{Fixed Boundary Flow for Non-random Data in Euclidean Space}\label{SEC:FBF_EUCLIDEAN}

The aim of this section is to prove that fixed boundary flows for non-random data are canonical, in the sense that they will pass through the usual principal component, in the context of Euclidean spaces. Hereafter, we suppose $\mathcal{M}$ is a linear subspace of $\mathbb{R}^d$, and $h = \infty$, which implies
\[
\Sigma_{n,h}(x) = \frac{1}{n}\sum_{i=1}^n (x_i - x)(x_i - x)^T.
\]
Under this configuration, we will figure out the supremum of $\mathcal{L}(W,\gamma)$ defined in (\ref{opt_original}) subjected to the constraint $\gamma \in \Gamma(\bar{x}_1, \bar{x}_2)$ defined in \eqref{curve_set_original}.

\begin{proposition}\label{prop_easy}
Suppose $\gamma_*:[0,\Xia{C\Delta}] \to \mathcal{M}$ such that
\[
\gamma_*(0) = \bar{x}_1 \quad {\rm and} \quad \dot{\gamma}_*(t) = W(\gamma_*(t)) \quad {\rm for} \ t \in (0,\Xia{C\Delta}] .
\]
If $\gamma_*(\Xia{C\Delta}) = \bar{x}_2$, then $\gamma_*$ is the unique optima of (\ref{opt_original})
\end{proposition}
\begin{proof}
Since $W(\gamma(t))$ and $\dot{\gamma}(t)$ are units for any $\gamma \in \Gamma(\bar{x}_1, \bar{x}_2)$, we have
\begin{align*}
    \sup_{\gamma} \mathcal{L}(W,\gamma) \leq \sup_{\gamma} \int_0^{C\Delta} \langle \dot{\gamma}(t), W(\gamma(t)) \rangle \leq \int_0^{C\Delta} \|\dot{\gamma}(t)\| \| W(\gamma(t))\|dt = C\Delta,
\end{align*}
and the equation holds only if $\dot{\gamma}(t) = W(\gamma(t))$ for any $t \in [0,C\Delta]$. Hence, $\gamma_*$ is the only curve that enables the equation to hold, and is accordingly the unique optima of (\ref{opt_original}).
\end{proof}

Proposition \ref{prop_easy} analyzes the optima of (\ref{opt_original}) under a strict condition that $\gamma_*(t) = \bar{x}_2$ with $t = C\Delta$. If the condition is relaxed to be $t \leq C\Delta$, things are more difficult. For further analysis, we suppose the original point of $\mathcal{M}$ to be $\bar{x} = \sum_{i=1}^n x_i$, and $[v_1, \cdots, v_d]$ to be the basis with $v_1 = W(\bar{x})$.
For convenience, we denote $z_i = v_i^Tx$ to be the $i$-th coordinate of any $z \in \M$ and $V_\bot = [v_2, \cdots, v_d] \in \mathbb{R}^{d \times (d-1)}$ hereafter.


Before giving our final proposition, we define some important sets and curves first. With $\odot$ representing \Xia{Hadamard
multiplication}, 
we denote a subset of $\Gamma(\bar{x}_1,\bar{x}_2)$, where the curves have the same direction with $W$,
\begin{align*}
    \Gamma_+(\bar{x}_1,\bar{x}_2) = \{\gamma \in \Gamma(\bar{x}_1,\bar{x}_2): \dot{\gamma}(t) \odot  W(\gamma(t)) \geq 0 \ {\rm for \ any} \ t \}.
\end{align*}
The red curves in Figure \ref{Fig:settings} (a) demonstrate flows satisfying $\dot{\gamma}(t) \odot W(\gamma(t)) \geq 0$, that is, the curves have the same direction as $W$.
Denote $p_1 = v_1v_1^T\bar{x}_1$ and $p_2 = v_1v_1^T\bar{x}_2$ 
as the projections of $\bar{x}_1$ and $\bar{x}_2$, respectively, onto the first axis. And $\Gamma_+(\bar{x}_1, p_1, v_1)$($\Gamma_+(\bar{x}_2, p_2, v_1)$) as the set of the curves from $\bar{x}_1$($\bar{x}_2$) to $p_1$($p_2$), orthogonal to $v_1$ and satisfying $\dot{\gamma}(t) \odot W(\gamma(t)) \geq 0$. 
We set
\begin{align*}
    \bar{\gamma}_1 = {\rm argsup}_{\gamma \in \Gamma_+(\bar{x}_1,p_1,v_1)} \mathcal{L}(W,\gamma), \quad
    \bar{\gamma}_2 = {\rm argsup}_{\gamma \in \Gamma_+(p_2, \bar{x}_2,v_1)} \mathcal{L}(W,\gamma).
\end{align*}
And we also set $\bar{\gamma}:[0,\|p_1-p_2\|] \to \mathcal{M}$ as the straight line between $p_1$ and $p_2$, that is $\bar{\gamma}(t) = p_1 + \frac{t}{\|p_2-p_1\|}(p_2-p_1)$.

Let $\gamma_s$ be the concatenation of $\bar{\gamma}_1, \bar{\gamma}$ and $\bar{\gamma}_2$, that is $\gamma_s:[0, \ell(\bar{\gamma}_1)+\ell(\bar{\gamma})+\ell(\bar{\gamma}_2)] \to \M$ satisfying
\begin{align*}
    \gamma_s(t) =
    \begin{cases}
    \bar{\gamma}_1(t) & \quad 0 \leq t \leq \ell(\bar{\gamma}_1) \\
    \bar{\gamma}_1(\ell(\bar{\gamma}_1)) + \bar{\gamma}(t) & \quad
    \ell(\bar{\gamma}_1) < t \leq \ell(\bar{\gamma}_1)+\ell(\bar{\gamma}) \\
    \bar{\gamma}_1(\ell(\bar{\gamma}_1)) + \bar{\gamma}(\ell(\bar{\gamma})) + \bar{\gamma}_2(t) & \quad
    \ell(\bar{\gamma}_1)+\ell(\bar{\gamma}) < t \leq \ell(\bar{\gamma}_1)+\ell(\bar{\gamma})+\ell(\bar{\gamma}_2),
    \end{cases}
\end{align*}
then $\gamma_s$ is continuous and in the closure of  $\Gamma_+(\bar{x}_1,\bar{x}_2)$ by Proposition \Xia{4.1 in the Supplementary Materials}. The yellow curve in Figure \ref{Fig:settings} (a) demonstrates $\gamma_s$.

In Figure \ref{Fig:settings} (a), we use the blue arrows to demonstrate an example of the vector field satisfying Assumption \ref{ass:vect_field}. Generally speaking, this refers to the arrows at the left half plane pointing towards $\bar{x}$ and arrows at the right half plane pointing in the opposite direction. Moreover, the arrows straighten horizontally as they approach the second axis. We summarize the assumptions on the vector field in Assumption \ref{ass:vect_field} (b) and (c).

\begin{assumption}\label{ass:vect_field}
$\ $
\begin{itemize}
    \item[(a)] $v^T \bar{x}_1 < 0$ and $v^T \bar{x}_2 > 0$.
    \item[(b)] For any $x \in \mathcal{M}$, $v_1^TW_{n,h}(x) \geq 0$ and $(v_i^TW(x)) * (v_i^Tx)*(v_1^Tx) \geq 0$ for any $i \geq 2$.
    \item[(c)] Suppose $x$ and $x'$ are in $\mathcal{M}$. If $V_\bot^T x = V_\bot^T x'$ and $|v_1^Tx| \leq |v_1^Tx'|\leq \max\{|v_1^T\bar{x}_1|,|v_1^T\bar{x}_2|\}$, then $|v_i^T W(x)| \leq |v_i^T W(x')|$ for any $i \geq 2$.
\end{itemize}
\end{assumption}

\begin{figure}[th]
\begin{center}
\subfigure[]{
\includegraphics[width=0.25\textwidth]{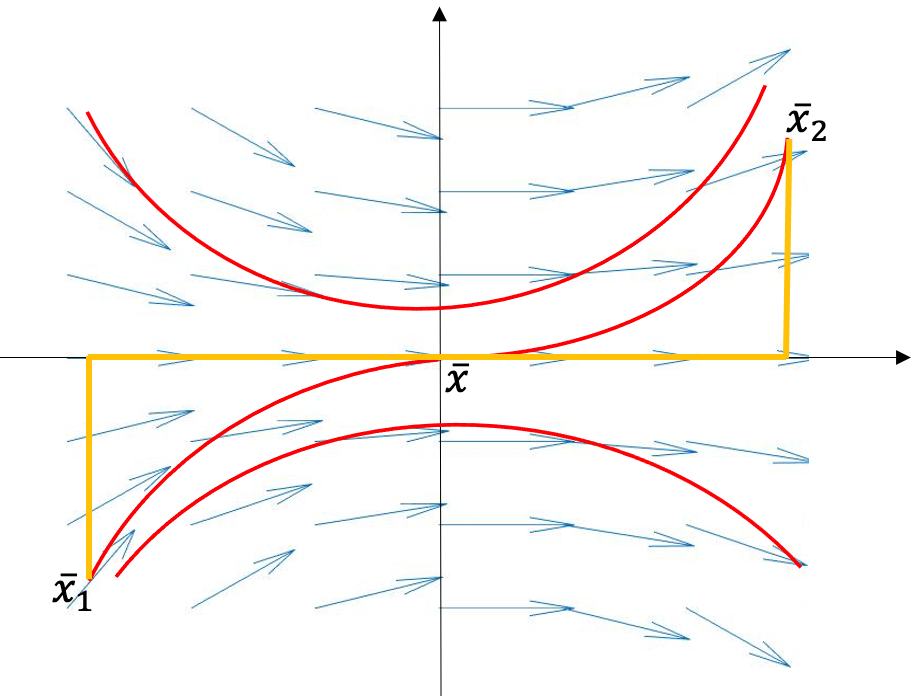}}
\subfigure[]{\includegraphics[width=0.2\textwidth]{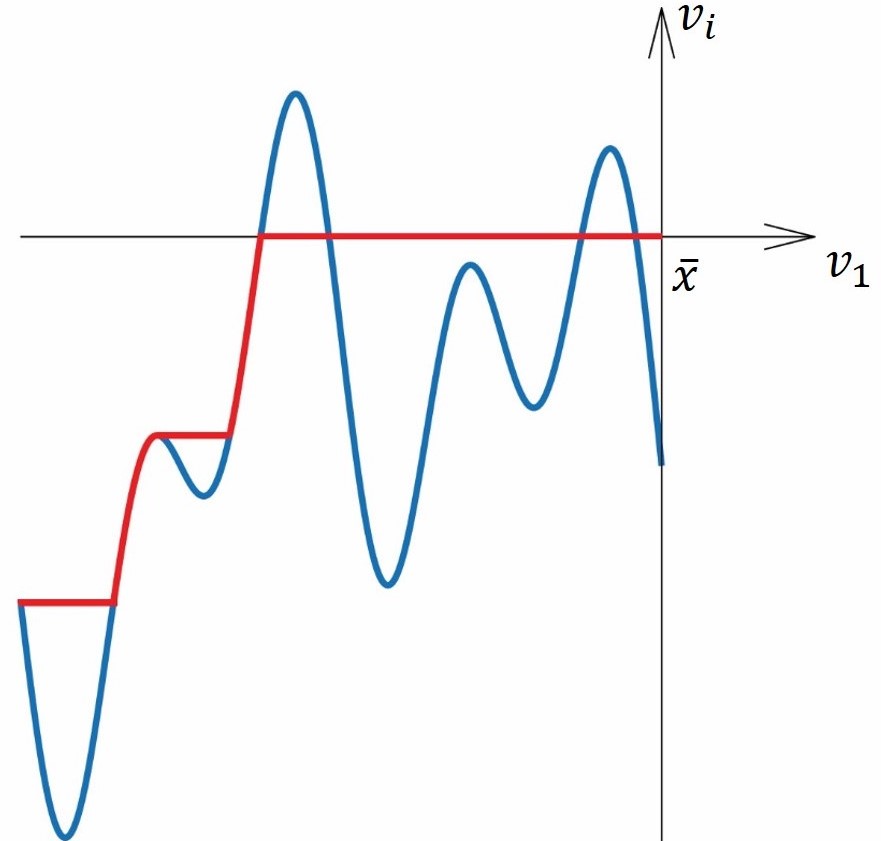}}
\subfigure[]{\includegraphics[width=0.19\textwidth]{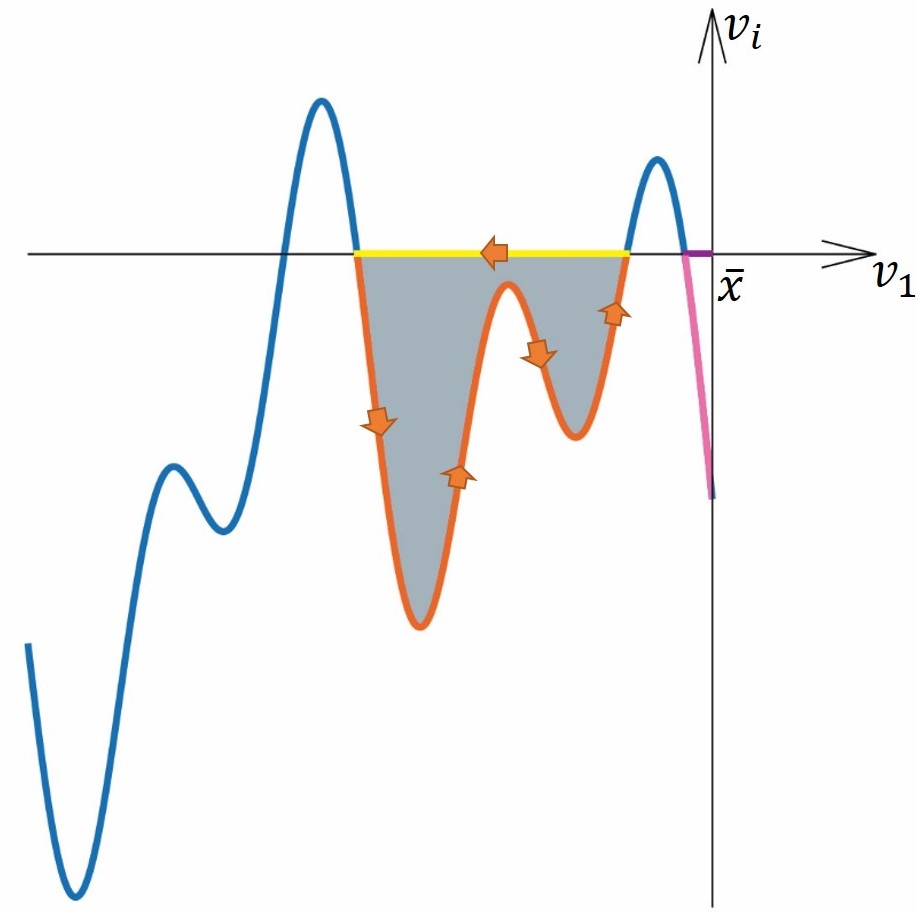}}
\end{center}
\caption{Image (a): The cross sectional area of $\M$ along the first and $i$-th axis. Vector field satisfying Assumption \ref{ass:vect_field} (blue arrows), flows satisfying $\dot{\gamma}(t) \odot W(\gamma(t)) \geq 0$ (red curves) and $\gamma_s$ (yellow curve). Image (b) and (c): Diagram of $\gamma$ (blue curve) and $\gamma_+$ (red curve) in $\M$. The orange and pink curves are segments of $\gamma$ and the yellow and purple curves are segments of $\gamma_+$.}
\label{Fig:settings}
\end{figure}

Assumption \ref{ass:vect_field} is not strict. Figure \ref{Fig:ass_vect_field} illustrates (b) and (c) of Assumption \ref{ass:vect_field} with two data sets, as represented by black points that are concentrated around a ``C''-shaped curve and an ``S''-shaped curve in $\mathbb{R}^2$. The two diagrams in the left-hand panel show the vector fields for the two data sets, both of which satisfy Assumption \ref{ass:vect_field}(b), while the diagrams in the right-hand panel show how $|v_2^TW(x)|$ varies at different points of $x$. Specifically, $|v_2^TW|$ gets larger when the color transitions to yellow, and smaller when the color transitions to blue. One can conclude from the two diagrams in the right panel that the vector field between the two orange lines satisfies Assumption \ref{ass:vect_field}(c).

\begin{figure}[ht]
    \centering
    \includegraphics[width=6in]{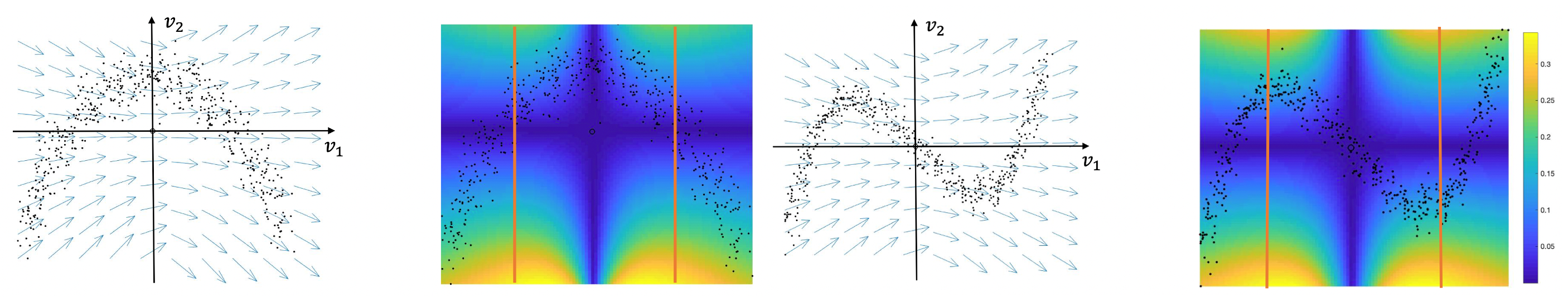}
    \caption{Vector field (left) and $|v_2^TW|$ (right) at different points.}
    \label{Fig:ass_vect_field}
\end{figure}

We can now set out sthe second proposition, which is under the general condition $\gamma(t) = \bar{x}_2$ with $t \leq C\Delta$. This proposition shows that if we restrict $\gamma$ in $\Gamma_+(\bar{x}_1, \bar{x}_2)$, the fixed boundary flow will pass through the usual principal component.

\begin{proposition}\label{PROP:PLUS}
If $\ell(\gamma_s) \leq C\Delta$, then
\begin{align*}
    \mathcal{L}(W, \gamma_s) = \sup_{\gamma \in \Gamma_+(\bar{x}_1,\bar{x}_2)} \mathcal{L}(W,\gamma).
\end{align*}
\end{proposition}


The proof of Proposition \ref{PROP:PLUS} can be found in Appendix D in the Supplementary Materials. Also in Appendix D in the Supplementary Materials,  we further explain
the inequality
\begin{align}\label{Gamma_neg}
\sup_{\gamma \in \Gamma_+(\bar{x}_1, \bar{x}_2)} \ \mathcal{L}(W, \gamma) \quad \geq \quad \sup_{\gamma \in \Gamma(\bar{x}_1, \bar{x}_2)/\Gamma_+(\bar{x}_1, \bar{x}_2)} \ \mathcal{L}(W, \gamma),
\end{align}
Combining the inequality in Proposition \ref{PROP:PLUS} and the inequality (\ref{Gamma_neg}), we conclude that the optimal solution of (\ref{opt_original}) always passes through the usual principal component. The scheme to show the inequality (\ref{Gamma_neg}) is organized as follows. As shown in Figure \ref{Fig:settings}(b) and (c), we construct $\gamma_+$ (the red curve) by any $\gamma$ (the blue curve), and illustrate $\mathcal{L}(W, \gamma_s) \geq \mathcal{L}(W, \gamma)$. In particular, if the dimension of the space is 2, the comparison between $\mathcal{L}(W, \gamma_s) $ and $ \mathcal{L}(W, \gamma)$ can be achieved by calculating the integration over the gray area using Green’s Theorem.

\section{Discussion} \label{SEC:DISS}
The determination of a fixed boundary flow for data points on non-linear manifolds is a very different problem from the case of principal flow. We propose the notion of a fixed boundary flow to define a curve with fixed starting and ending points and a tangent velocity that matches the maximal variation of data in its neighborhood at each point. The local geometry of data variation is represented by the tangent space at the given point, which compels us to use the local vector fields. Based on this choice, we formulate an optimization framework to construct a smooth curve on the manifold, with a tangent vector that always matches the local vector fields. There is no doubt that the solution to the optimization problem, and equivalently, the fixed boundary flow, depends on how a neighborhood is defined at a certain point, just as with principal flow.

The choice of the neighborhood depending on the scale parameter $h$ determines how local or global covariation features are captured by the fixed boundary flow.  Algorithm \ref{RFBF algorithm} provides a way to select a series of decreasing $h^{\sss (k)}$ till $\rho h^{\sss (k)} \leq 4\sqrt{\sigma}$, which obliges us to focus on the global trend of the curve first and the local second. Using this algorithm, we generate a curve represented by $\tilde{\gamma}$.  We discuss below the construction of a ``confidence band'' for the resulting fixed boundary flow $\tilde{\gamma} \in \M$. 
As we define the confidence band for the flow on the manifold, it should be a confidence ellipsoid. Note that the samples in $B_d(x,h)$ roughly lie within an ellipsoid with principal axes of length $h$, $\frac{\sqrt{\lambda_2}}{\sqrt{\lambda_1}}h, \cdots, \frac{\sqrt{\lambda_m}}{\sqrt{\lambda_1}}h$, respectively. 
Thus, we use the formulation of $\{\lambda_i\}_{i=1}^{m}$ to construct the ``confidence band''.  Specifically, 
for any point $x=\tilde{\gamma}(t)$ on the computed fixed boundary flow $\tilde{\gamma}$, we can define an ellipsoid of dimension $(m-1)$ in the intersection of $\mathcal{T}_x \M$ and the normal space at $x$, which could cover most samples in this intersection. By allowing the orthonormal $U(x) \in \mathbb{R}^{d\times m}$ be a basis of $\mathcal{T}_x\M$, the confidence ellipsoid is of dimension $(m-1)$ obeying
\begin{align*}
    \left\{\Pi_{\mathcal{M}}z: (z-\tilde{\gamma}(t))^TV \ {\rm diag} \ (\frac{\lambda_1}{\lambda_2h^2}, \cdots, \frac{\lambda_1}{\lambda_mh^2})V^T(z-\tilde{\gamma}(t)) = 1\right\},
\end{align*}
where $\Pi_{\mathcal{M}}z$ is the projection of $z$ onto $\mathcal{M}$ and $V = \big( U(\tilde{\gamma}(t))U(\tilde{\gamma}(t))^T - \dot{\tilde{\gamma}}(t)\dot{\tilde{\gamma}}(t)^T \big)(e_2, \cdots, e_m)$. Note that $U(x)$ can be estimated with certain theoretical guarantees (see \citet{tyagi2013tangent}). We remark that $\dot{\tilde{\gamma}}(t)$ usually approximates $W(\tilde{\gamma}(t))$, that is, $e_1$. This makes $V$ full column rank, and consequently, the dimension of the ellipsoid is $(m-1)$. If $\dot{\tilde{\gamma}}(t)$ is happened to be orthogonal to $W(\tilde{\gamma}(t))$, the dimension of the ellipsoid would reduce to $(m-2)$. With certain covering ellipsoid conditions for the samples in the neighborhood, one might consider bounding $\dot{\tilde{\gamma}}$ and $\tilde{\gamma}$ under the current setting. Some of the results in \citet{Yao2019principal} will be helpful in this respect. As this is part of our ongoing work, we intend to further investigate it in the future.

\section*{Acknowledgements}
 ZY is grateful for the financial support from his Singapore Ministry of Education (MOE) Tier 1 funding (A-0004809-00-00) and Tier 2 funding (A-0008520-00-00) at the National University of Singapore. ZY thanks Professor Shing-Tung Yau for his comments and discussion and the support from the Center of Mathematical Sciences and Applications at Harvard University.

\section*{Appendix A: Preliminaries}
In this section, we will introduce some preliminaries in Riemannian geometry and review the principal flows. We focus on studying a complete Riemannian manifold 
$\mathcal{M}$ of dimension $m$, equipped with a metric $g$. The smooth Riemannian manifold $\mathcal{M}$ can be isometrically embedded into the Euclidean space $(\mathbb{R}^d,\|\cdot\|)$, $m<d$. Assuming that the embedding is known, there exists a known differentiable function $F: \mathbb{R}^d \to \mathbb{R}^m$, and we have
\begin{eqnarray*}
\mathcal{M}:=\{x\in\mathbb{R}^d: F(x)=0\}.
\end{eqnarray*}
The Riemannian metric $g(\cdot,\cdot)$ on the Riemannian manifold $\mathcal{M}$ is induced by an inner product $\langle \cdot,\cdot\rangle$ defined in the tangent space $T_x\mathcal{M}$ at each point $x\in\mathcal{M}$, and the tangent space $T_x\mathcal{M}$ is denoted by
\begin{eqnarray*}
T_x\mathcal{M}:=\{y\in\mathbb{R}^d: DF y=0\}, ~~x\in \mathcal{M},
\end{eqnarray*}
where $DF$ is the $m\times d$ derivative matrix of $F$ evaluated at $x$. Since the tangent space $T_x\mathcal{M}$ is able to locally approximate the manifold, we define two mappings between the tangent space and the manifold. The exponential map at $x$ takes a tangent vector $v\in T_x\mathcal{M}$ denoted by
\begin{eqnarray*}
\exp_x: T_x\mathcal{M} \to \mathcal{M},
\end{eqnarray*}
and there exists a unique geodesic $\gamma_v$ satisfying $\gamma_v(0)=x$ with initial velocity $\dot{\gamma}_v(0)=v/\|v\|$. Therefore, the exponential map is locally defined by $\exp_x(v)=\gamma_v(\|v\|)$ in the neighborhood of $x$. The inverse of the exponential map, the logarithm map, is denoted by 
\begin{eqnarray*}
\log_x: \mathcal{M} \to  T_x\mathcal{M}.
\end{eqnarray*}

Before we review principal flows, we recall the Fr\'echet mean and tangent space PCA, which are basic elements to construct principal flows. Let $\{x_1,\dots,x_n\}$ be the data points lying on the manifold $\mathcal{M}$, where there exists a connected open set $B\subset\mathcal{M}$, such that $B$ covers $\{x_1,\dots,x_n\}$. 
The Fr\'echet mean $\bar{x}\in B$ is the point that minimizes the sum of square distances under the Riemannian metric
\begin{eqnarray*}
\frac{1}{n}\sum_{i=1}^ng^2(\cdot,x_i).
\end{eqnarray*}
The tangent components $e_1(\bar{x}),\ldots,e_m(\bar{x})$ at $\bar{x}$ form a basis for the tangent space $T_{\bar{x}} \mathcal{M}$, and they are given by the first $r$ eigenvectors of the scale $h$ local tangent covariance matrix 
\begin{eqnarray*}
\Sigma_h(\bar{x})=\frac{1}{\sum_i\kappa_h(x_i,\bar{x})}\sum_{i=1}^n(\log_{\bar{x}}(x_i)\otimes\log_{\bar{x}}(x_i))\kappa_h(x_i,\bar{x}),
\end{eqnarray*} 
where $y\otimes y=yy\trans$, $\kappa_h(x,\bar{x})=K(h^{-1}\|\log_{\bar{x}}x-\bar{x}\|)$ and $h>0$.

\Xia{
When $h = \infty$,  and $\mathcal{M} = \mathbb{R}^d$, the local tangent covariance matrix reduced to be 
\begin{align*}
\Sigma_h(x) & = \frac{1}{n} \sum_{i=1}^n (x_i - \bar{x} + \bar{x} - x)(x_i-\bar{x}+\bar{x}-x)^T \\
& = \Sigma_h(\bar{x}) + (x-\bar{x})(x-\bar{x})^T.
\end{align*}
Noting $W(x)$ is unit length, $\langle W(x), dW(x) \rangle = 0$, we could calculate the derivation of $\lambda_1(x)$ as follows:
\begin{align*}
d \lambda_1(x)  & = d \langle \Sigma_h(x) W(x), W(x) \rangle  \\
& = \langle (d \Sigma_h(x)) W(x), W(x) \rangle +  \langle \Sigma_h(x) d W(x), W(x) \rangle + \langle \Sigma_h(x) W(x), d W(x) \rangle \\
& = \langle (d \Sigma_h(x)) W(x), W(x) \rangle +  \langle d W(x), \Sigma_h^T(x) W(x) \rangle + \langle \Sigma_h(x) W(x), d W(x) \rangle \\
& = \langle (d \Sigma_h(x)) W(x), W(x) \rangle +  \langle d W(x), \lambda_1(x)W(x) \rangle + \langle  \lambda_1(x) W(x), d W(x) \rangle \\
& = \langle (d \Sigma_h(x)) W(x), W(x) \rangle +  2\lambda_1(x) \langle  W(x), d W(x) \rangle \\
& = \langle (d \Sigma_h(x)) W(x), W(x) \rangle = \langle (d \Sigma_h(x)), W(x) W(x)^T \rangle \\
& = \langle d \big( (x-\bar{x})(x-\bar{x})^T \big) , W(x) W(x)^T \rangle \\
& = 2\langle W(x)W(x)^T(x-\bar{x}), dx \rangle.
\end{align*}
}

\section*{Appendix B: Proof of Theorem 3.1}

In subsequent proof, we need a special case of  Theorem 8 by \citet{yaoxia2019}, where the manifold degenerates into a curve, that is, the dimension of the manifold is $1$.  We state this special case of Theorem 8 below, where we use $h$ instead of $r'$ as the scale, to ensure the same as the symbol in this paper.
\begin{thm}[Slight deformation of Theorem 8 by \citet{yaoxia2019}]\label{thm:thm_yaoxia}
Let $z$ be a point off a curve $\gamma$, $z^*$ be the projection of $z$ onto $\gamma$, $d(z,\gamma^\ast)\leq h$. We have
\[
\| \Pi_z - \Pi_{z^*}^*\|_F \leq \frac{C}{h^2}\frac{1}{|I_{z,h}|}\sum_{i \in I_{z,h}}\left(  \|\xi_i\|^4 + \|\xi_i\|^3 + \|\xi_i\|^2 + h\|\xi_i\| \right) + C\left(h+\frac{\|z-z^*\|}{h} + \frac{\|z-z^*\|^2}{h^2} \right) ,
\]
where $\Pi_{z^*}^*$ denotes the orthogonal projection onto the normal space of $z^*$ and $\Pi_z = v_\bot v_\bot^T$. Here $V_\bot$ is the orthogonal component of $v$ and $v$ is the first   eigenvector of $\Sigma_h(z)$.
\end{thm}
To bound the summation of some power of $\|\xi_i\|$ above, we need Proposition 2.3 by \citet{yaoxia2019} as follows.
\begin{proposition}[Proposition 2.3 by \citet{yaoxia2019}]\label{prop:prop_yaoxia}
Suppose $\xi \sim N(0, \sigma^2 I_d)$; then we have, for any positive integer $k$:
\begin{itemize}
    \item[(1)] $\mathbb{E}(\|\xi\|_2^k) = C_1\sigma^k$
    \item[(2)] ${\rm Var}(\|\xi\|_2^k) = C_2 \sigma^{2k} $
    \item[(3)] $\mathbb{E}\big(\|\xi\|_2^k- \mathbb{E}(\|\xi\|_2^k) \big)^3 =C_3\sigma^{3k}$
    \item[(4)] $\|\xi_i\|_2^k$ and $\|\xi_j\|_2^k$ are independent if $\xi_i$ and $\xi_j$ are independent,
\end{itemize}
where $C_1$, $C_2$, and $C_3$ are three constants depending on $d$ and $k$.
\end{proposition}

Based on the above proposition, we obtain the upper bound of the summation of each $\| \xi_i\|^k $ for points lying in a tube surrounding $\gamma^\ast$.

\begin{proposition}\label{prop:sum_xi}
For a given $\delta$, there exists $C_n$ such that if $n \geq C_n \sqrt{\sigma} $, then
\begin{align}
\frac{1}{|\I(x,h)|}\sum_{i \in \I(x,h)} \| \xi_i\|^k \leq C \sigma^k
\end{align}
holds with probability $1-\delta$ for any $(x,h)$ satisfying $h>4\sqrt{\sigma}$, $d(x,\gamma^\ast) \leq \sqrt{\sigma}$, $\|x-\gamma^\ast(0)\|>h/2$ and $\|x-\gamma^\ast(r^*)\| > h/2$.
\end{proposition}

\begin{proof}
Noticing $\{\xi_i\}$ are i.i.d. samples drawn from Gaussian distribution, we can obtain the expectation $\mu_k = C_1\sigma^k$, variance $\sigma_k^2 = C_2\sigma^{2k}$ and the third moment $\rho_k = C_3\sigma^{3k}$ of $\|\xi_i\|^k$ according to Proposition \ref{prop:prop_yaoxia}. By Berry-Esseen Theorem, the cumulative distribution of $\big(\sum_{i \in \I(x,h)} \|\xi_i\|^k - \mu_k |\I(x,h)| \big) / \big( \sigma_k \sqrt{|\I(x,h)|} \big)$ denoted by $F$ satisfies
\[
|F(t) - \Phi(t)| \leq C\rho_k / \sigma_k^3 \leq  C' / \sqrt{|\I(x,h)|},
\]
where $\Phi$ is the cumulative distribution function of standard normal distribution. So, there exists $C$ depending on $d$, $k$ and $\delta$ such that
\[
\frac{1}{|\I(x,h)|}\sum_{i \in \I(x,h)} \|\xi_i\|^k \leq C\sigma^k,
\]
with probability at least $1-\delta/3-C'/\sqrt{|\I(x,h)|}$.

To estimate $|\I(x,h)|$, we calculate the probability of $i \in \I(x,h)$ based on $h > 4\sqrt{\sigma}$,
\begin{align*}
\mathbb{P}(i \in \I(x,h)) & \geq  \mathbb{P}(\|x_i-x\| \leq h) \geq \mathbb{P}(\tilde{x}_i \in \gamma^\ast \cap B_d(x,h/2))\mathbb{P}(\|\xi_i\|\leq h-h/2) \\
 & = \frac{\ell(\gamma^\ast \cap B_d(x,h/2))}{\ell(\gamma^\ast)}\mathbb{P}(\|\xi_i\|\leq 2\sqrt{\sigma}).
\end{align*}
Letting $x^* = \gamma^\ast(t^*)$ be the projection of $x$ onto $\gamma^\ast$, then $x^* \in B_d(x, h/2)$ since $\|x-x^*\| = d(x,\gamma^\ast) \leq \sigma < h/4$. Since $\|x-\gamma^\ast(0)\| > h/2$ and $\|x-\gamma^\ast(r^*)\| > h/2$, there exists $0<t_1<t^*<t_2<r^*$ such that $x_1 = \gamma^\ast(t_1)$ and $x_2=\gamma^\ast(t_2)$ satisfy $\|x-x_1\| = \|x-x_2\| = h/2$. Hence, $\ell(\gamma^\ast \cap B_d(x,h/2)) \geq \|x_1-x^*\| + \|x^*-x_2\| \geq \big(\|x_1-x\|-\|x-x^*\| \big)+ \big(\|x-x_2\|-\|x-x^*\|\big) \geq h/2$. Since each entry of $\xi_i$ obeys Gaussian distribution, $\|\xi_i\|^2/\sigma^2$ obeys Chi-squared distribution. According to the cdf of Chi-squared distribution, we could obtain $\mathbb{P}(\|\xi_i\|\leq 2\sqrt{\sigma}) = O(1)$, and thereby $\mathbb{P}(i \in \I(x,h)) \geq \frac{h/2}{r^*} \cdot O(1) \geq ch = 2c\sqrt{\sigma}$.
Thus, whether $i \in \I(x,h)$ or not can be treated as a Bernoulli distribution with expectation no less than $2c\sqrt{\sigma}$. Applying Berry-Esseen theorem to the $n$ Bernoulli trials, there exists $c' < 1$ such that $|\I(x,h)| \geq c'n\sqrt{\sigma}$ with probability $1-C''/\sqrt{n}$, which implies,
\begin{align*}
\mathbb{P}(\frac{1}{|\I(x,h)|}\sum_{i \in \I(x,h)} \|\xi_i\|^k \leq C\sigma^k)
&\geq 1-\delta/3-C'/\sqrt{|\I(x,h)|} \\
& \geq \big(1-\delta/3-C'/(c'n\sqrt{\sigma}) \big) \mathbb{P}(|\I(x,h)| \geq c'n\sqrt{\sigma}) \\
& \geq \big(1-\delta/3-C'/(c'n\sqrt{\sigma}) \big)(1-C''/\sqrt{n}).
\end{align*}
Setting
\[
n \geq n_0 := \max \{ \frac{9{C''}^2}{\delta^2} , \frac{C'}{c'\sqrt{\sigma}} \frac{1}{1-\frac{\delta}{3} - \frac{1-\delta}{1-\frac{\delta}{3}}} \},
\]
one could verify $1-\delta/3-C'/(c'n\sqrt{\sigma}) \geq (1-\delta)/(1-\delta/3)$ and $ (1-C''/\sqrt{n} \geq (1-\delta/3)$, which implies $\mathbb{P}(\frac{1}{|\I(x,h)|}\sum_{i \in \I(x,h)} \|\xi_i\|^k \leq C\sigma^k) \geq 1-\delta$. Hence, we could take
\[
C_n = \max \{ \frac{9{C''}^2}{\delta^2} , \frac{C'}{c'(1-\frac{\delta}{3} - \frac{1-\delta}{1-\frac{\delta}{3}})} \}
\geq \max \{ \frac{9{C''}^2 \sqrt{\sigma}}{\delta^2} , \frac{C'}{c'(1-\frac{\delta}{3} - \frac{1-\delta}{1-\frac{\delta}{3}})} \}
\]
to complete this proof.
\end{proof}

\begin{proof}[Proof of Theorem 3.1]
\RVSD{
For any $t \in T$, plugging $z = \gamma^\ast(t)$ into Theorem \ref{thm:thm_yaoxia}, we have
\begin{align*}
\| \Pi_z - \Pi_{z^*}^*\|_F & \leq \frac{C}{h^2}\frac{1}{|I_{z,h}|}\sum_{i \in I_{z,h}}\left(  \|\xi_i\|^4 + \|\xi_i\|^3 + \|\xi_i\|^2 + h\|\xi_i\| \right) + Ch \\
& \leq \frac{C}{h^2}\frac{1}{|I_{z,h}|}\sum_{i \in I_{z,h}}\left(  \sigma^4 + \sigma^3 + \sigma^2 + h\sigma \right) + Ch \leq Ch,
\end{align*}
where the second inequality holds by Proposition \ref{prop:sum_xi} with probability $1-\delta$, and the last inequality holds since $h>4\sqrt{\sigma}$.
}

\RVSD{
Let $u = \dot{\gamma}^*(t)$, the tangent vector of $\gamma^\ast$ at $\gamma^\ast(t)$, and $v = W(\gamma^\ast(t))$, the first eigenvector of $\Sigma_h(\gamma^\ast(t))$. By the definition of $\Pi_z$ and $\Pi_{z^*}^*$, we have $\Pi_z = v_\bot v_\bot^T$ and $\Pi_{z^*}^* = u_\bot u_\bot^T$. Hence $\| \Pi_z - \Pi_{z^*}^*\|_F = \| v_\bot v_\bot^T - u_\bot u_\bot^T \| = \|uu^T - vv^T\| \leq Ch$. 
Noting
\begin{align*}
 C^2h^2 \geq \|uu^T - vv^T\|^2 = 2 - 2\langle u, v \rangle^2 = 2(1+\langle u, v \rangle)(1-\langle u, v \rangle) \geq 2(1-\langle u, v \rangle),
\end{align*}
we have $1-\langle u, v \rangle \leq \frac{C^2}{2}h^2$, that is, $\langle u, v \rangle = \langle \dot{\gamma}^*(t), W(\gamma^\ast(t)) \rangle \geq 1 - \frac{C^2}{2}h^2$,
which completes the proof.}
\end{proof}

\section*{Appendix C: Proof of Theorem 3.2 - Theorem 3.4}

\begin{lemma} \label{lma:bound_z}
Suppose 
$h > 4\sqrt{\sigma}$, $\|x-\gamma^\ast(0)\| > h/2$, $\|x-\gamma^\ast(r^*)\| > h/2$ and $x \in X$. For any given $\delta$, there exists $C$ independent on $x$ and $h$ such that
$d(\frac{1}{|\I(x,h)|} \sum_{i \in \I(x,h)} x_i ,  \gamma^\ast) \leq C h^2$
with probability $1-\delta$.
\end{lemma}

\begin{proof}
In this proof, we simplify $\frac{1}{|\I(x,h)|} \sum_{i \in \I(x,h)} x_i$ to be $z$ for convenience. For $i \in \I(x, h)$,
we use $x_i^*$ to represent the projection of $x_i$ onto $\gamma^\ast$, and similarly, use $z^*$ to represent the projection of $z$ onto $\gamma^\ast$. Denoting the tangent space of $\gamma^\ast$ at $z^*$ to be $T_{z^*} \gamma^\ast$, we have
\begin{align*}
d(z, \gamma^\ast) = d(z, T_{z^*}\gamma^\ast) & = \|z - z^*\| \\
& \leq \|z - {x_i}^*\| \leq \|z - x_i\| + \|x_i - {x_i}^*\| < 2.5 h
\end{align*}
where $\|x_i - {x_i}^*\|\leq \sqrt{\sigma} < h/2$ by \RVSD{(3.3) in the main manuscript} and
\[
\|z - x_i\| \leq \frac{1}{|\I(x, h)|} \sum_{i' \in \I(x, h)} \|x_{i'}-x_i\| \leq 2h.
\]

To get a tight bound on $d(z, \gamma^\ast)$, we denote $\Pi_{z^*}^*$ to be the orthogonal projection onto the normal space of $\gamma^\ast$ at $z^*$ and obtain
\begin{align*}
d(z, T_{z^*}\gamma^\ast) & = \| \Pi_{z^*}^* \big( \frac{1}{|\I(x,h)|} \sum_{i \in \I(x,h)} x_i - z^* \big) \|
\\
& \leq \frac{1}{|\I(x,h)|} \|\Pi_{z^*}^*\| \sum_{i \in \I(x,h)} \|x_i -{x_i}^*\| + \frac{1}{|\I(x,h)|} \sum_{i \in \I(x,h)} \|\Pi_{z^*}^*(x_i^* - z^*)\| \\
& \leq \frac{1}{|\I(x,h)|} \sum_{i \in \I(x,h)} \|\xi_i\| + \frac{1}{|\I(x,h)|} \sum_{i \in \I(x,h)} \frac{\|x_i^* - z^*\|^2}{\tau},
\end{align*}
where the second term of the last inequality follows Theorem 4.18 by \citet{Federer1959}. Noting
\[
\|x_i^*-z^*\| \leq \|x_i^*-z\|+\|z-z^*\| < 2.5h+2.5h = 5h,
\]
and $\frac{1}{|\I(x,h)|} \sum_{i \in \I(x,h)} \|\xi_i\| \leq C\sigma$ with high probability by Proposition \ref{prop:sum_xi}, we could bound
\[
d(z, \gamma^\ast) = d(z, T_{z^*} \gamma^\ast) \leq C \sigma + \frac{25}{\tau} h^2 = O(h^2),
\]
which completes the proof.
\end{proof}

\begin{lemma}\label{lma:first_eigvec}
Let $z$ be a point off $\gamma^\ast$, $z^*$ be the projection of $z$ onto $\gamma^\ast$, and $v$ be the tangent vector of $\gamma^\ast$ at $z^*$. If $d(z,\gamma^\ast) \leq C_1h^2$ and $\|z-\gamma^\ast(t)\| > h/2$ for $t=0,1$, then for any given $\delta$, there exists $C$ such that $\|vv^T-e_1(z) e_1(z)^T\|\leq Ch$ with probability $1-\delta$.
\end{lemma}
\begin{proof}
\RVSD{
Since $d(z,\gamma^\ast) \leq C_1h^2$, $\|z-z^*\| = d(z,\gamma^\ast) \leq C_1h^2$. Plugging $z$ into Theorem \ref{thm:thm_yaoxia}, we have
\begin{align*}
\| \Pi_z - \Pi_{z^*}^*\| & \leq \frac{C}{h^2}\frac{1}{|\I(z,h)|}\sum_{i \in \I(z,h)}\left(  \|\xi_i\|^4 + \|\xi_i\|^3 + \|\xi_i\|^2 + h\|\xi_i\| \right) + C(h + C_1 h + C_1^2 h^2) \\
& \leq \frac{C}{h^2}\frac{1}{|\I(z,h)|}\sum_{i \in \I(z,h)}\left(  \sigma^4 + \sigma^3 + \sigma^2 + h\sigma \right) + Ch + Ch^2\leq Ch,
\end{align*}
where the second inequality holds by Proposition \ref{prop:sum_xi} with probability $1-\delta$, and the last inequality holds since $4\sqrt{\sigma}< h \leq 1$.
}

By the definition of $\Pi_z$ and $\Pi_{z^*}^*$, we have $\Pi_z = I - e_1(z) e_1(z)^T$ and $\Pi_{z^*}^* = I-vv^T$. Hence $\| \Pi_z - \Pi_{z^*}^*\|_F =  \| vv^T - e_1(z) e_1(z)^T\| \leq Ch$. 
\end{proof}

\begin{proposition} \label{prop:uv}
Suppose $u$ and $v$ are normal vectors, then $\|uu^T-vv^T\| = \sqrt{2} \|(I-vv^T)u\|$.
\end{proposition}
\begin{proof}
To prove this proposition, we calculate $\|uu^T-vv^T\|^2$ and $\|(I-vv^T)u\|^2$ respectively as per
$
\|uu^T-vv^T\|^2 = \langle uu^T, uu^T \rangle + \langle vv^T, vv^T \rangle - 2 \langle uu^T, vv^T \rangle = 2 - 2 \langle uu^T, vv^T \rangle,
$
and
$
\|(I-vv^T)u\|^2 = \langle (I-vv^T)u, (I-vv^T)u \rangle = \langle I, uu^T \rangle - \langle uu^T, vv^T \rangle = 1 - \langle uu^T, vv^T \rangle,
$
which complete the proof.
\end{proof}

\begin{proposition}\label{prop:dist_tangent_gamma}
  If $a \in \gamma^\ast$ and $b \in T_a \gamma^*$, then $d(b, \gamma^\ast) \leq \frac{1}{2\tau}\|a-b\|^2$.
\end{proposition}
\begin{proof}
  Let $a = \gamma^\ast(t_0)$ and $\Delta t = \langle b-a, {\gamma^\ast}^{\prime}(t_0) \rangle$. By Taylor's expansion,
  \begin{align*}
  \gamma^\ast(t_0 + \Delta t) & = \gamma^\ast(t_0) + \dot{\gamma}^\ast(t_0) \Delta + \frac{\ddot{\gamma}^{\ast}(t')}{2} \Delta t^2 \\
  & = \gamma^\ast(t_0) + \dot{\gamma}^\ast(t_0){\dot{\gamma}^\ast(t_0)}^T(b-a) + \frac{\ddot{\gamma}^\ast(t')}{2} \Delta t^2 \\
  & = a + (b-a) + \frac{\ddot{\gamma}^\ast(t')}{2} \Delta t^2 = b + \frac{\ddot{\gamma}^\ast(t')}{2} \Delta t^2,
  \end{align*}
  where $t' \in [t_0, t_0 + \Delta t]$ and $\dot{\gamma}^\ast(t_0){\dot{\gamma}^\ast(t_0)}^T(b-a) = b-a$ since $b \in T_{a^*}\gamma^\ast$ and $\|\dot{\gamma}^\ast\|=1$. Hence, we could obtain $d(b, \gamma^\ast) \leq \|b - \gamma^\ast(t_0+\Delta t)\| \leq \frac{1}{2\tau} \|b-a\|^2$.
\end{proof}

\begin{proposition}\label{prop:line_gamma}
Suppose $x = z_i^{{\sss (k)}} + \alpha e_1(z_i^{{\sss (k)}},h^{{\sss (k)}})$, $d(z_i^{{\sss (k)}},\gamma^\ast) \leq C_1{h^{{\sss (k)}}}^2$,  $\|z_i^{{\sss (k)}} - \gamma^\ast(0)\| > h^{{\sss (k)}}/2$ and $\|z_i^{{\sss (k)}} - \gamma^\ast(r^*)\| > h^{{\sss (k)}}/2$. 
If $|\alpha| \leq C_2 h^{{\sss (k)}}$, then $d(x,\gamma^\ast) \leq C{h^{{\sss (k)}}}^2$.
\end{proposition}

\begin{proof}
This proof is conducted in two steps: First, we show that there is $\bar{x} \in T_{{z_i^{{\sss (k)}}}^*} \gamma^\ast$ such that $d(x, T_{{z_i^{{\sss (k)}}}^*}\gamma^\ast) = \|x-\bar{x}\| \leq C{h^{{\sss (k)}}}^2$ and then we bound $d(\bar{x}, \gamma^\ast) \leq C{h^{{\sss (k)}}}^2$. These two claims conclude $d(x, \gamma^\ast) \leq \|x-\bar{x}\|+ d(\bar{x},\gamma^\ast) \leq C{h^{{\sss (k)}}}^2$.

To begin with,
\begin{align*}
d(x, T_{{z_i^{{\sss (k)}}}^*}\gamma^\ast) & = \|\Pi^*_{{z_i^{{\sss (k)}}}^*}(x -{z_i^{{\sss (k)}}}^*)\| \leq \|z_i^{{\sss (k)}}-{z_i^{{\sss (k)}}}^*\| + \|\Pi^*_{{z_i^{{\sss (k)}}}^*}(x - z_i^{{\sss (k)}})\| \\
& = \|z_i^{{\sss (k)}}-{z_i^{{\sss (k)}}}^*\| + |\alpha|\|(I-vv^T)e_1(z_i^{{\sss (k)}})\| \\
&= d(z_i^{{\sss (k)}},\gamma^\ast) + \frac{|\alpha|}{\sqrt{2}}\|vv^T-e_1(z_i^{{\sss (k)}}) e_1(z_i^{{\sss (k)}})^T\| \\
& \leq C_1{h^{{\sss (k)}}}^2 + \frac{C_2h^{{\sss (k)}}}{\sqrt{2}} Ch^{{\sss (k)}} \leq C{h^{{\sss (k)}}}^2,
\end{align*}
where $\|vv^T-e_1(z_i^{{\sss (k)}}) e_1(z_i^{{\sss (k)}})^T\| \leq Ch^{{\sss (k)}}$ by Lemma \ref{lma:first_eigvec}.
Denote the projection of $x$ onto $T_{{z_i^{{\sss (k)}}}^*} \gamma^\ast$ to be $\bar{x}$, then 
\begin{align*}
\|\bar{x}-{z_i^{{\sss (k)}}}^*\| 
& \leq \|\bar{x}-x\|+\|x-z_i^{{\sss (k)}}\|+\|z_i^{{\sss (k)}}-{z_i^{{\sss (k)}}}^*\|  \\ 
& \leq d(x, T_{{z_i^{{\sss (k)}}}^*}\gamma^\ast)+\|x-z_i^{{\sss (k)}}\|+\|z_i^{{\sss (k)}}-{z_i^{{\sss (k)}}}^*\| \leq C{h^{{\sss (k)}}}^2+C_2h^{{\sss (k)}}+C_1{h^{{\sss (k)}}}^2\leq Ch^{{\sss (k)}}.
\end{align*}
Taking $a = {z_i^{{\sss (k)}}}^*$ and $b = \bar{x}$ in Proposition \ref{prop:dist_tangent_gamma}, we have $d(\bar{x}, \gamma^\ast) \leq C{h^{{\sss (k)}}}^2$,
which completes the proof.
\end{proof}

We impose constrains to $\tilde{\gamma}^{{\sss (k)}}(t_i)$ rather than $z_i^{{\sss (k)}}$ and $\alpha$, and obtain the following Lemma. We use $s(a,b)$ to denote the segment between $a$ and $b$ hereafter, that is, $s(a,b) = \{\beta a+(1-\beta)b : \beta \in[0,1]\}$.

\RVSD{
\begin{lemma}\label{lma:iteration_k}
Suppose the discrete curve at the $k$-th iteration satisfies the following conditions:
\begin{itemize}
    \item[(a)] 
    $d(\tilde{\gamma}^{{\sss (k)}}(t_{i}),\gamma^\ast) \leq C_1 h^{{\sss (k)}}$,
    \item[(b)] 
    $\|\tilde{\gamma}^{{\sss (k)}}(t_{i+1}) - \tilde{\gamma}^{{\sss (k)}}(t_i)\| \leq C_2 h^{{\sss (k)}}$ for any $i<2N^{{\sss (k)}}$, 
    \item[(c)] $\|\tilde{\gamma}^{{\sss (k)}}(t_{2j+1})-\gamma^\ast(0)\|\geq (2C_1+3.5)h^{{\sss (k)}}$ and $\|\tilde{\gamma}^{{\sss (k)}}(t_{2j+1})-\gamma^\ast(r^*)\|\geq (2C_1+3.5)h^{{\sss (k)}}$ for any $j = 0, 1, \cdots, N^{{\sss (k)}}-1$.
\end{itemize}
For any given $\delta$, there exists $C$ such that the segments $s\big( \gamma_{{\rm proj},2j+1}^{{\sss (k)}}(t_{2j}),\gamma_{{\rm proj},2j+1}^{{\sss (k)}}(t_{2j+2}) \big)$
for any $j = 0, 1, \cdots, N^{{\sss (k)}}-1$ are within Hausdorff distance $C{h^{{\sss (k)}}}^2$ to $\gamma^\ast$ with probability $1-\delta$. 
\end{lemma}
}

\begin{proof}
\RVSD{
Since $d(\tilde{\gamma}^{{\sss (k)}}(t_{2j+1}),\gamma^\ast) \leq C_1h^{{\sss (k)}}$ for each $j = 0, 1, \cdots, N^{{\sss (k)}}-1$, $B_d(\tilde{\gamma}^{{\sss (k)}}(t_{2j+1}), 2C_1 h^{{\sss (k)}}) \cap X $ is non-empty according to Assumption 3.2 in the main manuscript. As the closest point to $\tilde{\gamma}^{{\sss (k)}}(t_{2j+1})$,  
$x^{{\sss (k)}}_{\rm proj}(t_{2j+1})$ exists and $\|\tilde{\gamma}^{{\sss (k)}}(t_{2j+1}) - x^{{\sss (k)}}_{\rm proj}(t_{2j+1}) \| \leq 2C_1 h^{{\sss (k)}}$. 
Moveover, $\|\bar{x}^{{\sss (k)}}_{\rm proj}(t_{2j+1}) - x^{{\sss (k)}}_{\rm proj}(t_{2j+1})\| \leq h^{\sss (k)}$ since $\bar{x}^{{\sss (k)}}_{\rm proj}(t_{2j+1})$ is the mean of $B_d(x^{{\sss (k)}}_{\rm proj}(t_{2j+1}) , h^{\sss (k)})$. Noting $x^{{\sss (k)}}_{\rm proj}(t_{2j+1}) \in B_d(\bar{x}^{{\sss (k)}}_{\rm proj}(t_{2j+1}) , h^{\sss (k)})$, the nearest sample in the data cloud to $\bar{x}^{{\sss (k)}}_{\rm proj}(t_{2j+1})$, which is the final projection of $\tilde{\gamma}^{{\sss (k)}}(t_{2j+1})$ to the data cloud denoted by $\gamma^{{\sss (k)}}_{\rm proj}(t_{2j+1})$,  satisfies 
\begin{align*}
\|\gamma^{{\sss (k)}}_{\rm proj}(t_{2j+1})-\tilde{\gamma}^{{\sss (k)}}(t_{2j+1})\| 
& \leq \|x^{{\sss (k)}}_{\rm proj}(t_{2j+1})-\tilde{\gamma}^{{\sss (k)}}(t_{2j+1})\| + 
\|\bar{x}^{{\sss (k)}}_{\rm proj}(t_{2j+1})-x^{{\sss (k)}}_{\rm proj}(t_{2j+1})\| \\
& + \|\bar{x}^{{\sss (k)}}_{\rm proj}(t_{2j+1})-\gamma^{{\sss (k)}}_{\rm proj}(t_{2j+1})\| \\
& \leq 2C_1 h^{{\sss (k)}} + h^{{\sss (k)}} + h^{{\sss (k)}} = (2C_1+2)h^{{\sss (k)}}.
\end{align*}
Hence, the distance between $\gamma^{{\sss (k)}}_{\rm proj}(t_{2j+1})$ and the ends of $\gamma^*$ can be bounded below as
\begin{align*}
\|\gamma^{{\sss (k)}}_{\rm proj}(t_{2j+1}) - \gamma^\ast(t)\| 
& \geq \| \tilde{\gamma}^{{\sss (k)}}(t_{2j+1})-\gamma^\ast(t)\| - \|\tilde{\gamma}^{{\sss (k)}}(t_{2j+1}) - \gamma^{{\sss (k)}}_{\rm proj}(t_{2j+1})\| \\
& \geq (2C_1+3.5)h^{{\sss (k)}} - (2C_1+2)h^{{\sss (k)}} = 1.5h^{{\sss (k)}} >  0.5h^{{\sss (k)}},    \quad {\rm for} \ t= 0,r^*.
\end{align*}
}
Plugging $x = \gamma^{{\sss (k)}}_{\rm proj}(t_{2j+1})$ into Lemma \ref{lma:bound_z}, we conclude $d(z_{2j+1}^{{\sss (k)}},\gamma^\ast) = O({h^{{\sss (k)}}}^2)$ with probability $1-\delta$. Moreover,
\begin{align*}
\|z_{2j+1}^{{\sss (k)}} - \gamma^\ast(t)\|
 \geq \|\gamma^{{\sss (k)}}_{\rm proj}(t_{2j+1}) - \gamma^\ast(t)\| - \| \gamma^{{\sss (k)}}_{\rm proj}(t_{2j+1}) - z_{2j+1}^{{\sss (k)}}\| \geq 1.5h^{{\sss (k)}} - h^{{\sss (k)}} = 0.5h^{{\sss (k)}},
\end{align*}
for $t=0,r^*$.
Hence, the conditions on $z_{2j+1}^{{\sss (k)}}$ in Proposition \ref{prop:line_gamma} are satisfied for $j=0,1,\cdots,N^{{\sss (k)}}-1$ with probability $1-\delta$. 
We utilize Proposition \ref{prop:line_gamma} to prove the conclusion for the segment between $\gamma_{{\rm proj},2j+1}^{{\sss (k)}}(t_{2j})$ and $\gamma_{{\rm proj},2j+1}^{{\sss (k)}}(t_{2j+2})$. Since $\gamma_{{\rm proj},2j+1}^{{\sss (k)}}(t_{2j})$ is the projection of $\tilde{\gamma}^{{\sss (k)}}(t_{2j})$ onto the line passing $z_{2j+1}^{{\sss (k)}}$ along direction $e_1(z_{2j+1}^{{\sss (k)}}, h^{{\sss (k)}})$, $\gamma_{{\rm proj},2j+1}^{{\sss (k)}}(t_{2j})$ can be written as $\gamma_{{\rm proj},2j+1}^{{\sss (k)}}(t_{2j}) = z_{2j+1}^{{\sss (k)}} + \alpha_1 e_1(z_{2j+1}^{{\sss (k)}},h^{{\sss (k)}})$ with $|\alpha_1| = \|\gamma_{{\rm proj},2j+1}^{{\sss (k)}}(t_{2j})-z_{2j+1}^{{\sss (k)}}\| \leq \|\tilde{\gamma}^{{\sss (k)}}(t_{2j})-z_{2j+1}^{{\sss (k)}}\|$. Thus,
\begin{align}
\begin{split}\label{bound:tildegamma_z1}
|\alpha_1| & \leq \|\tilde{\gamma}^{{\sss (k)}}(t_{2j})-z_{2j+1}^{{\sss (k)}}\| \\
& \leq \|\tilde{\gamma}^{{\sss (k)}}(t_{2j}) - \tilde{\gamma}^{{\sss (k)}}(t_{2j+1})\| + \|\tilde{\gamma}^{{\sss (k)}}(t_{2j+1}) - \gamma^{{\sss (k)}}_{\rm proj}(t_{2j+1}) \| + \|\gamma^{{\sss (k)}}_{\rm proj}(t_{2j+1}) - z_{2j+1}^{{\sss (k)}}\|\\
& \leq (C_2 + 2C_1 + 1)h^{{\sss (k)}}.
\end{split}
\end{align}
Analogically, $\gamma_{{\rm proj},2j+1}^{{\sss (k)}}(t_{2j+2})$ can be written as $\gamma_{{\rm proj},2j+1}^{{\sss (k)}}(t_{2j+2})= z_{2j+1}^{{\sss (k)}} + \alpha_2 e_1(z_{2j+1}^{{\sss (k)}},h^{{\sss (k)}})$ with
\begin{align}\label{bound:tildegamma_z2}
|\alpha_2| \leq \| \tilde{\gamma}^{{\sss (k)}}(t_{2j+2})-z_{2j+1}^{{\sss (k)}} \| \leq (C_2 + C_1 + 1)h^{{\sss (k)}}.
\end{align}
Any point on the segment between $\gamma_{{\rm proj},2j+1}^{{\sss (k)}}(t_{2j})$ and $\gamma_{{\rm proj},2j+1}^{{\sss (k)}}(t_{2j+2})$ is a convex combination of $\gamma_{{\rm proj},2j+1}^{{\sss (k)}}(t_{2j})$ and $\gamma_{{\rm proj},2j+1}^{{\sss (k)}}(t_{2j+2})$, that is, such a point equals
\[
\beta \gamma_{{\rm proj},2j+1}^{{\sss (k)}}(t_{2j}) + (1-\beta) \gamma_{{\rm proj},2j+1}^{{\sss (k)}}(t_{2j+2}) = z_{2j+1}^{{\sss (k)}} + (\beta \alpha_1 + (1-\beta) \alpha_2) e_1(z_{2j+1}^{{\sss (k)}},h^{{\sss (k)}}),
\]
with a certain $\beta \in [0,1]$. Also, we could verify that $|\beta \alpha_1 + (1-\beta) \alpha_2| \leq \beta |\alpha_1| + (1-\beta)|\alpha_2| \leq (C_2 + 2C_1 + 1)h^{{\sss (k)}}$. Hence, any point on the segment between $\gamma_{{\rm proj},2j+1}^{{\sss (k)}}(t_{2j})$ and $\gamma_{{\rm proj},2j+1}^{{\sss (k)}}(t_{2j+2})$ satisfies the condition of Proposition \ref{prop:line_gamma} and thereby it has a distance less than $C{h^{{\sss (k)}}}^2$ to $\gamma^\ast$, where $C$ is independent on $h^{{\sss (k)}}$.
\end{proof}


\begin{proposition}\label{bound:convex_com}
  If $d(a, \gamma^\ast) = O(h^2)$, $d(b, \gamma^\ast) = O(h^2)$ and $\|a-b\| = O(h)$, then $d(c,\gamma^\ast) = O(h^2)$ for any $c \in s(a, b)$.
\end{proposition}
\begin{proof}
  Letting $a^*$ and $b^*$ be the projections of $a$ and $b$ onto $\gamma^\ast$ respectively, then $\|a^*-b^*\| \leq \|a^*-a\|+\|a-b\|+\|b-b^*\| = O(h)$. Letting $T_{a^*}\gamma^\ast$ be the tangent space of $\gamma^\ast$ at $a^*$,
  \[
  d(b, T_{a^*}\gamma^\ast) \leq \|b-b^*\| + d(b^*, T_{a^*}\gamma^\ast) = O(h^2)
  \]
  by Theorem 4.18 in \citet{Federer1959}.

  Based on the above inequalities, we start to bound $d(c, T_{a^*}\gamma^\ast)$ for $c=\beta a + (1-\beta)b$. Denoting $u$ be the normalized tangent vector of $\gamma^\ast$ at $a^*$, then
  \begin{align*}
    d(c,T_{a^*}\gamma^\ast)
    & = \|(I-uu^T)(c-a^*)\| = \|(I-uu^T)\big(\beta a + (1-\beta)b - a^* \big) \| \\
    & \leq  \beta \|(I-uu^T)(a - a^*)\| + (1-\beta)\|(I-uu^T)(b-a^*)\| \\
    & \leq \beta d(a, \gamma^\ast) + (1-\beta)d(b,T_{a^*}\gamma^\ast) = O(h^2).
  \end{align*}
  Let the projection of $c$ onto $T_{a^*}\gamma^\ast$ be $c^*$. We could bound $\|a^*-c^*\|$ as
  \[
    \|a^*-c^*\| \leq \|a^*-a\|+\|a-c\|+\|c-c^*\| = O(h),
  \]
  and thereby we could bound $d(c^*, \gamma^\ast) = O(h^2)$ by Proposition \ref{prop:dist_tangent_gamma}. Hence, we have
  \begin{align*}
  d(c, \gamma^\ast) \leq \|c-c^*\| + d(c^*, \gamma^\ast) \leq d(c, T_{a^*}\gamma^\ast)+d(c^*,\gamma^\ast) = O(h^2).
  \end{align*}
\end{proof}


\begin{proof}[Proof of Theorem 3.2]
For any $j = 1, \cdots, N^{{\sss (k)}}$, we will show that as the two projections of $\tilde{\gamma}^{{\sss (k)}}(t_{2j})$ to $e_1(z_{2j-1}^{{\sss (k)}},h^{{\sss (k)}})$ and $e_1(z_{2j+1}^{{\sss (k)}},h^{{\sss (k)}})$ respectively, $\gamma_{{\rm proj},2j-1}^{{\sss (k)}}(t_{2j})$ and $\gamma_{{\rm proj},2j+1}^{{\sss (k)}}(t_{2j})$  are not far away from each other:  
\begin{align*}
    \|\gamma_{{\rm proj},2j-1}^{{\sss (k)}}(t_{2j})-\gamma_{{\rm proj},2j+1}^{{\sss (k)}}(t_{2j}) \|
     & \leq \| \tilde{\gamma}^{{\sss (k)}}(t_{2j})-\gamma_{{\rm proj},2j-1}^{{\sss (k)}}(t_{2j}) \| + \|\tilde{\gamma}^{{\sss (k)}}(t_{2j})-\gamma_{{\rm proj},2j+1}^{{\sss (k)}}(t_{2j})\| \\
     & \leq \| \tilde{\gamma}^{{\sss (k)}}(t_{2j}) - z_{2j-1}^{{\sss (k)}} \| + \|\tilde{\gamma}^{{\sss (k)}}(t_{2j}) - z_{2j+1}^{{\sss (k)}}\| = O(h^{{\sss (k)}}),
  \end{align*}
where the last inequality holds by (\ref{bound:tildegamma_z1}) and (\ref{bound:tildegamma_z2}). According to Lemma \ref{lma:iteration_k}, the two segments $s\big( \gamma_{{\rm proj},2j-1}^{{\sss (k)}}(t_{2j-2}),\gamma_{{\rm proj},2j-1}^{{\sss (k)}}(t_{2j}) \big)$ and $s\big( \gamma_{{\rm proj},2j+1}^{{\sss (k)}}(t_{2j}),\gamma_{{\rm proj},2j+1}^{{\sss (k)}}(t_{2j+2}) \big)$ are within Hausdorff distance $C{h^{\sss (k)}}^2$ to $\gamma^\ast$ with probability $1-\delta$. And thereby the two points $\gamma_{{\rm proj},2j-1}^{{\sss (k)}}(t_{2j})$ and $\gamma_{{\rm proj},2j+1}^{{\sss (k)}}(t_{2j})$ are within Hausdorff distance $C{h^{\sss (k)}}^2$ to $\gamma^\ast$. Plugging $a = \gamma_{{\rm proj},2j-1}^{{\sss (k)}}(t_{2j})$ and $b=\gamma_{{\rm proj},2j+1}^{{\sss (k)}}(t_{2j})$ into Proposition \ref{bound:convex_com}, we could prove the Hausdorff distance between $s(\gamma_{{\rm proj},2j-1}^{{\sss (k)}}(t_{2j}), \gamma_{{\rm proj},2j+1}^{{\sss (k)}}(t_{2j})$ and $\gamma^\ast$ is less than $O({h^{{\sss (k)}}}^2)$ for any $j=1,\cdots,N^{{\sss (k)}} $. Similarly, we could also prove $s(\bar{x}_1, \gamma_{{\rm proj},1}^{{\sss (k)}}(t_{0})$ and $s(\gamma_{{\rm proj},2N^{{\sss (k)}}-1}^{{\sss (k)}}(t_{2N^{{\sss (k)}}}), \bar{x}_2)$ are within the same Hausdorff distance.
 As the union of all the segments, $d_H(\mathcal{S}, \gamma^\ast) = O({h^{{\sss (k)}}}^2)$. 
\end{proof}

\begin{proposition}\label{prop:init}
Given the initial discrete curve $\{\tilde{\gamma}^{{\sss (0)}}(t_i)\}_{i=0}^{2N^{{\sss (0)}}}$, if there exists constants $h^{{\sss (0)}}$, $C_1$ and $C_2$ satisfies
\begin{itemize}
\item[(a)] $Ch^{{\sss (0)}} \leq C_1 \rho$ and $(4C_1+7)h^{{\sss (0)}} < \|\gamma^\ast(0) - \gamma^\ast(r^*)\| $,
\item[(b)] $d(\tilde{\gamma}^{{\sss (0)}}(t_{i}),\gamma^\ast) \leq C_1 h^{{\sss (0)}}$, for any $i = 1, \cdots, 2N^{{\sss (0)}}-1 $
    \item[(c)] $\|\tilde{\gamma}^{{\sss (0)}}(t_{i+1}) - \tilde{\gamma}^{{\sss (0)}}(t_i)\| \leq C_2 h^{{\sss (0)}}$ for any $i<2N^{{\sss (0)}}$, 
    \item[(d)] $\|\tilde{\gamma}^{{\sss (0)}}(t_{2j+1})-\gamma^\ast(0)\|\geq (2C_1+3.5)h^{{\sss (0)}}$ and $\|\tilde{\gamma}^{{\sss (0)}}(t_{2j+1})-\gamma^\ast(r^*)\|\geq (2C_1+3.5)h^{{\sss (0)}}$ for any $j = 0, 1, \cdots, N^{{\sss (0)}}-1$.
\item[(e)] $C_2 > 4C_1+7$
\end{itemize}
then the three conditions of Lemma \ref{lma:iteration_k} hold with probability $(1-\delta)^{k}$ for any $k \geq 1$.
\end{proposition}
\begin{proof}
Since $\rho \leq 1$, $h^{{\sss (k)}} = \rho^{k}h^{\sss (0)} < h^{\sss (0)}$, the conditions $(4C_1+7)h^{{\sss (k)}} < \|\gamma^\ast(0) - \gamma^\ast(r^*)\| $ and $Ch^{{\sss (k)}} \leq C_1 \rho$ hold for any $k$. In order to prove the three conditions hold for $k \geq 1$ if they hold for $k=0$, we only need to prove the three conditions hold for $k+1$ if they hold for $k$.
If the conditions hold for $k$, then Lemma \ref{lma:iteration_k} holds. We obtain the segments $\{s\big( \gamma_{{\rm proj},2j+1}^{{\sss (k)}}(t_{2j}),\gamma_{{\rm proj},2j+1}^{{\sss (k)}}(t_{2j+2}) \big)\}$ for any $j = 0, 1, \cdots, N^{{\sss (k)}}$ are within Hausdorff distance $O({h^{{\sss (k)}}}^2)$ to $\gamma^\ast$ with probability $1-\delta$. Hence, by Proposition \ref{bound:convex_com}, we obtain
  \[
  d_H(\tilde{\gamma}^{\sss (k+1)}, \gamma^\ast) \leq C({h^{{\sss (k)}}}^2) \leq C_1 \rho h^{{\sss (k)}} = C_1 h^{\sss (k+1)},   \  \mbox{ with \ probability} \ 1-\delta
  \]
  which is the first condition of  Lemma \ref{lma:iteration_k} with probability $1-\delta$.
  
 $\tilde{\gamma}^{\sss (k+1)}$ is a continuous polyline starting at $\bar{x}_1$ and ending at $\bar{x}_2$. Removing the points in $B_d(\gamma^\ast(0), (2C_1+3.5)h^{\sss (k+1)})$ and $B_d(\gamma^\ast(r^*), (2C_1+3.5)h^{\sss (k+1)})$ from $\tilde{\gamma}^{\sss (k+1)}$, we obtain
 \begin{align}\label{def:tilde_S_0}
\tilde{\gamma}^{\sss (k+1)}_0 = \tilde{\gamma}^{\sss (k+1)} \setminus \big( B_d(\gamma^\ast(0), (2C_1+3.5)h^{\sss (k+1)}) \cup B_d(\gamma^\ast(r^*), (2C_1+3.5)h^{\sss (k+1)}) \big)
\end{align}
Since $\|\gamma^\ast(0) - \gamma^\ast(r^*)\|> (4C_1+7)h^{\sss (k+1)}$, 
$\tilde{\gamma}^{\sss (k+1)}_0$ is non-empty. Hence, we could select a discrete curve $\tilde{\gamma}^{\sss (k+1)}(t_i)$, $i = 1, \cdots, 2N^{\sss (k+1)}-1$, from $\tilde{\gamma}_0^{{\sss (k+1)}}$, such that nearby points are within distance $C_2h^{\sss (k+1)}$, that is, $\|\tilde{\gamma}^{\sss (k+1)}(t_{i+1}) - \tilde{\gamma}^{\sss (k+1)}(t_i)\| \leq C_2h^{\sss (k+1)}$ for any $i<2N^{\sss (k+1)}$, the second condition of Lemma \ref{lma:iteration_k} at the $(k+1)$-th iteration. Moreover, since $\{\tilde{\gamma}^{\sss (k+1)}(t_i) \} \subset \tilde{\gamma}_0^{\sss (k+1)}$, we have $\|\tilde{\gamma}^{\sss (k+1)}(t_i) - \gamma^\ast(t)\| \geq (2C_1+3.5)h^{\sss (k+1)}$ for $t=1,2$, which is the third condition of Lemma \ref{lma:iteration_k}.
\end{proof}

\begin{proof}[Proof of Theorem 3.3]
By Proposition \ref{prop:init}, the three conditions of Lemma \ref{lma:iteration_k} hold with probability $(1-\delta)^k$ for any $k \geq 1$. When the conditions hold, we have 
\[
d_H \Big( \tilde{\gamma}^{\sss (k+1)} , \gamma^\ast \Big) = O({h^{{\sss (k)}}}^2), \ \mbox{with \ probability} \ 1-\delta
\]
by Theorem 3.2. Hence, 
\[
 P \Big( d_H(\tilde{\gamma}^{\sss (k+1)} , \gamma^\ast ) = O({h^{{\sss (k)}}}^2) \Big) = (1-\delta)^k(1-\delta) = (1-\delta)^{k+1}.
 \]
\end{proof}

\begin{proof}[Proof of Theorem 3.4]
Since $\gamma^\ast \subset \M$, we have $d(x,\M)\leq d(x,\gamma^\ast)$. Using this inequality, we could obtain the following inequalities:
  \[
    d(\tilde{x}, \gamma^\ast) \leq d(\tilde{x},x)+d(x,\gamma^\ast) \leq d(x,\M)+d(x,\gamma^\ast) \leq d(x,\gamma^\ast) + d(x,\gamma^\ast) = O(h).
  \]
\end{proof}

\section*{Appendix D: Data set of Labelled Faces in the Wild in Section 5.2}

In our study, we downloaded $264$ images of $66$ people with four images of each person. The images of the face region for the $66$ individuals are shown in Figure \ref{Fig:face_data}.

\begin{figure}[H]
\begin{center}
\begin{minipage}[t]{0.8\textwidth}
\includegraphics[width=0.4in]{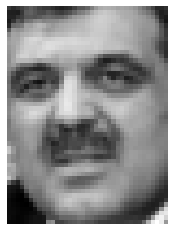}
\includegraphics[width=0.4in]{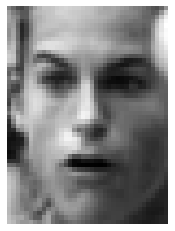}
\includegraphics[width=0.4in]{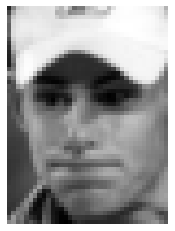}
\includegraphics[width=0.4in]{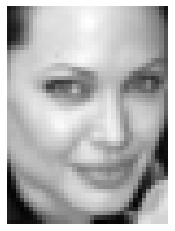}
\includegraphics[width=0.4in]{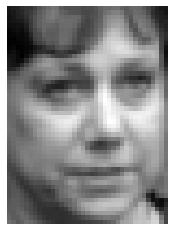}
\includegraphics[width=0.4in]{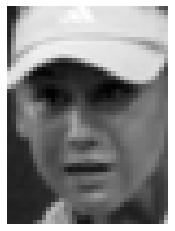}
\includegraphics[width=0.4in]{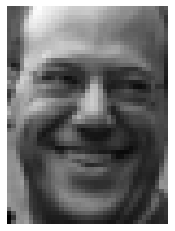}
\includegraphics[width=0.4in]{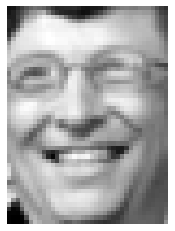}
\includegraphics[width=0.4in]{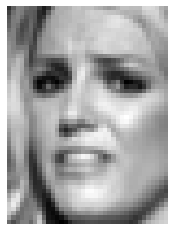}
\includegraphics[width=0.4in]{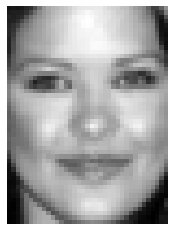}
\includegraphics[width=0.4in]{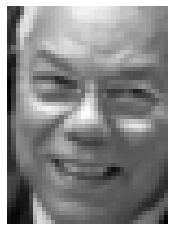}
\end{minipage}
\begin{minipage}[t]{0.8\textwidth}
\includegraphics[width=0.4in]{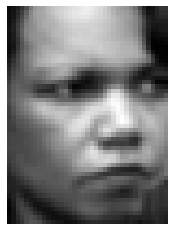}
\includegraphics[width=0.4in]{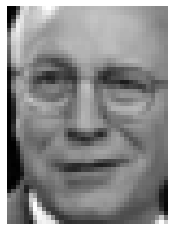}
\includegraphics[width=0.4in]{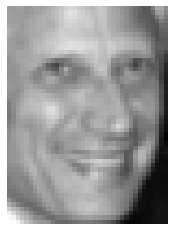}
\includegraphics[width=0.4in]{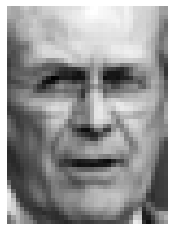}
\includegraphics[width=0.4in]{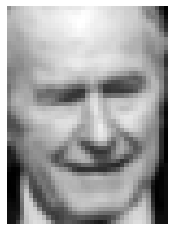}
\includegraphics[width=0.4in]{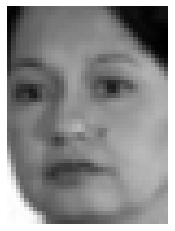}
\includegraphics[width=0.4in]{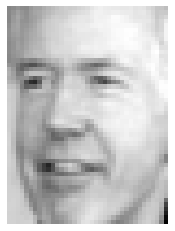}
\includegraphics[width=0.4in]{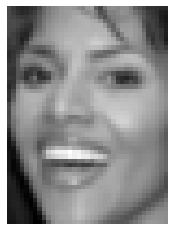}
\includegraphics[width=0.4in]{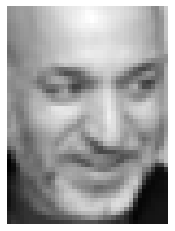}
\includegraphics[width=0.4in]{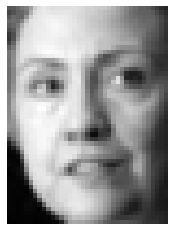}
\includegraphics[width=0.4in]{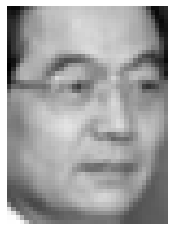}
\end{minipage}
\begin{minipage}[t]{0.8\textwidth}
\includegraphics[width=0.4in]{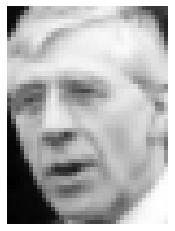}
\includegraphics[width=0.4in]{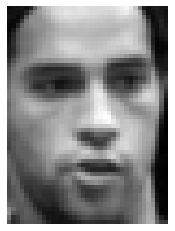}
\includegraphics[width=0.4in]{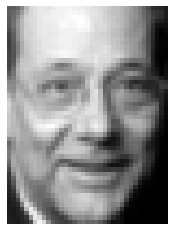}
\includegraphics[width=0.4in]{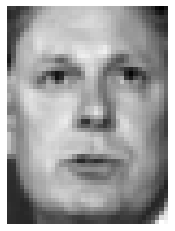}
\includegraphics[width=0.4in]{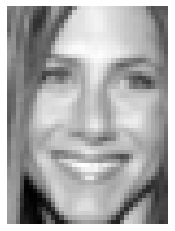}
\includegraphics[width=0.4in]{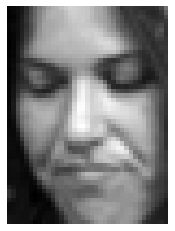}
\includegraphics[width=0.4in]{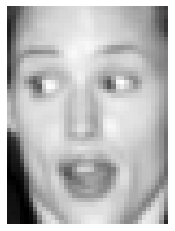}
\includegraphics[width=0.4in]{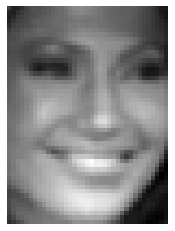}
\includegraphics[width=0.4in]{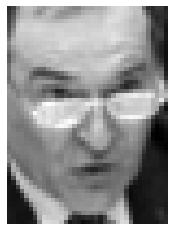}
\includegraphics[width=0.4in]{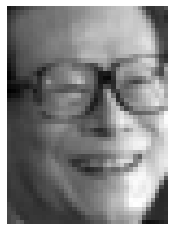}
\includegraphics[width=0.4in]{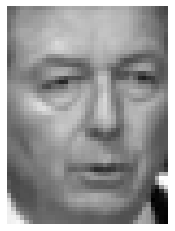}
\end{minipage}
\begin{minipage}[t]{0.8\textwidth}
\includegraphics[width=0.4in]{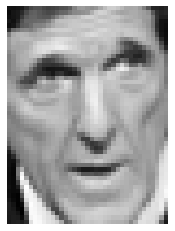}
\includegraphics[width=0.4in]{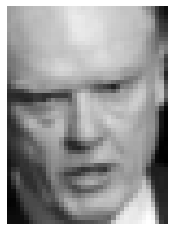}
\includegraphics[width=0.4in]{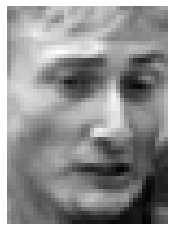}
\includegraphics[width=0.4in]{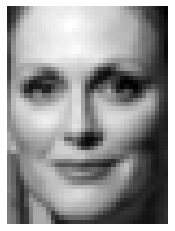}
\includegraphics[width=0.4in]{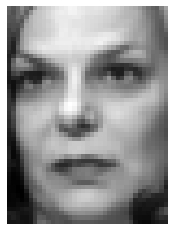}
\includegraphics[width=0.4in]{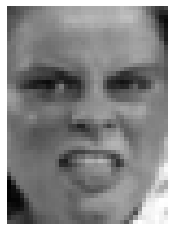}
\includegraphics[width=0.4in]{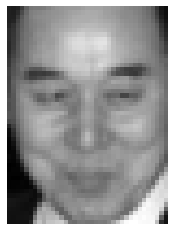}
\includegraphics[width=0.4in]{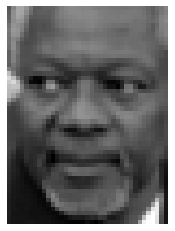}
\includegraphics[width=0.4in]{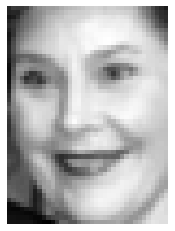}
\includegraphics[width=0.4in]{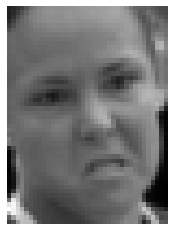}
\includegraphics[width=0.4in]{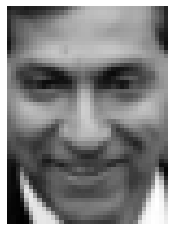}
\end{minipage}
\begin{minipage}[t]{0.8\textwidth}
\includegraphics[width=0.4in]{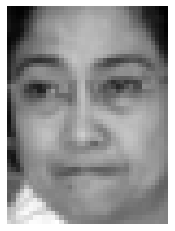}
\includegraphics[width=0.4in]{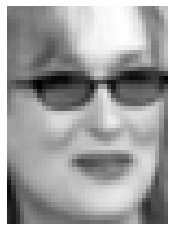}
\includegraphics[width=0.4in]{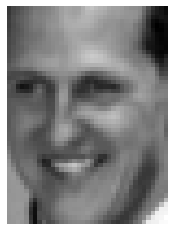}
\includegraphics[width=0.4in]{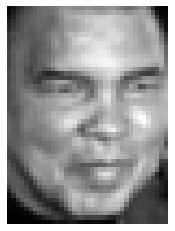}
\includegraphics[width=0.4in]{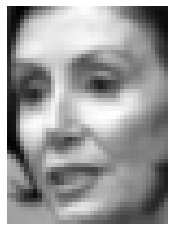}
\includegraphics[width=0.4in]{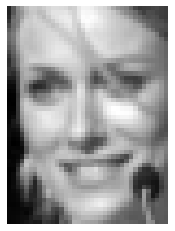}
\includegraphics[width=0.4in]{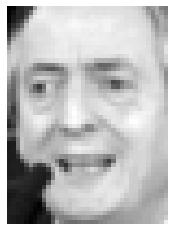}
\includegraphics[width=0.4in]{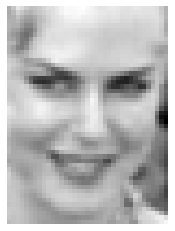}
\includegraphics[width=0.4in]{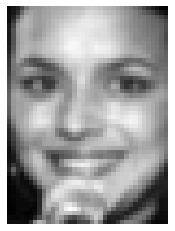}
\includegraphics[width=0.4in]{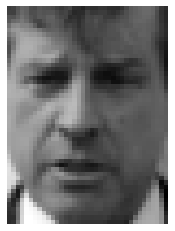}
\includegraphics[width=0.4in]{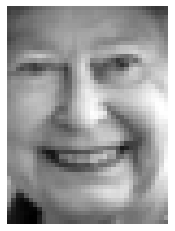}
\end{minipage}
\begin{minipage}[t]{0.8\textwidth}
\includegraphics[width=0.4in]{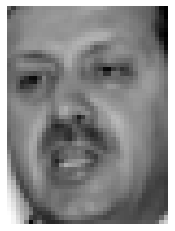}
\includegraphics[width=0.4in]{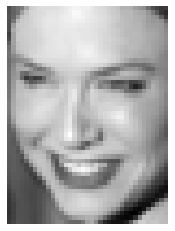}
\includegraphics[width=0.4in]{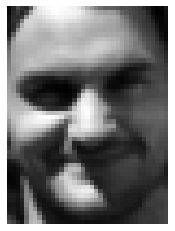}
\includegraphics[width=0.4in]{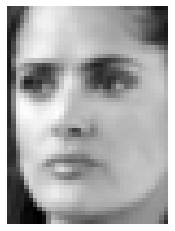}
\includegraphics[width=0.4in]{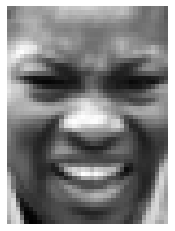}
\includegraphics[width=0.4in]{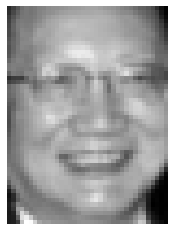}
\includegraphics[width=0.4in]{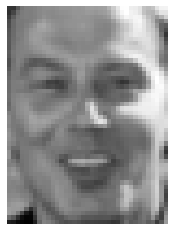}
\includegraphics[width=0.4in]{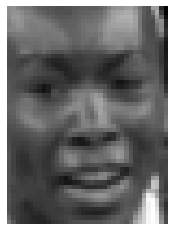}
\includegraphics[width=0.4in]{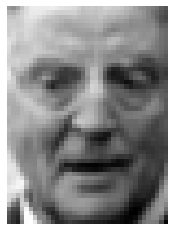}
\includegraphics[width=0.4in]{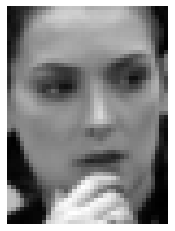}
\includegraphics[width=0.4in]{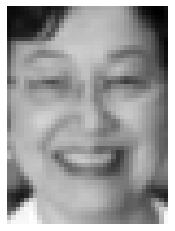}
\end{minipage}
\end{center}
\caption{Images of $66$ individuals’ faces.}\label{Fig:face_data}
\end{figure}

\section*{Appendix E: Proof of Proposition 6.2}

\begin{proposition}\label{PROP:GAMMA_S}
If $\ell(\bar{\gamma}_1)+\ell(\bar{\gamma})+\ell(\bar{\gamma}_2) \leq C$, then $\gamma_s$ belongs to the closure $\Gamma_+(\bar{x}_1,\bar{x}_2)$.
\end{proposition}
\begin{proof}
For simplicity, we denote $\bar{S}$ to be the closure of a set $S$. Let $P$ be the set of polynomials and $C$ be the set of continuous functions. Based on the Stone-Weierstrass Theorem, we have $P \subset C^2 \subset \bar{P} = C$, which implies that the closure of $C^2$ is $C$.
Based on this conclusion, we have
\begin{eqnarray}
\bar{\Gamma}(\bar{x}_1,\bar{x}_2)&=\{&\gamma: [0,r]\to \mathcal{M}: \gamma\in C([0,r]), r\in(0,\RVSD{C\Delta}], \nonumber\\
&& \gamma(0)=\bar{x}_1, \gamma(r)=\bar{x}_2, \gamma(s)\neq \gamma(s^\prime)~\text{for}~s\neq s^\prime, \nonumber \\
&& \ell(\gamma[0,t])=t, \text{for~all}~ t\in[0,r]\}, \nonumber
\end{eqnarray}
and
\begin{align*}
    \bar{\Gamma}_+(\bar{x}_1,\bar{x}_2) = \{\gamma \in \bar{\Gamma}(\bar{x}_1,\bar{x}_2): \dot{\gamma}(t) \odot  W(\gamma(t)) \geq 0 \ {\rm for \ any} \ t \}.
\end{align*}
Since $\gamma_s$ is continuous and satisfies $\dot{\gamma}(t) \odot  W(\gamma(t)) \geq 0$, we conclude that $\gamma_s \in \bar{\Gamma}_+(\bar{x}_1,\bar{x}_2)$.
\end{proof}

\begin{proof}[Proof of Proposition 6.2]
For any $\gamma \in \Gamma_+(\bar{x}_1,\bar{x}_2)$, there exists $t_0$ such that $v_1^T(\gamma(t_0)) = 0$, since $v_1^T(\gamma(0)) = v_1^T\bar{x}_1 < 0$ and $v_1^T(\gamma(r)) = v_1^T\bar{x}_2 > 0$. Define $\gamma_1:[0,t_0] \to \mathcal{M}$ by $\gamma_1(t) = \gamma(t)+ v_1v_1^T(\gamma(0)-\gamma(t))$ and $\gamma_2:[0,r-t_0] \to \mathcal{M}$ by $\gamma_2(t) = \gamma(t-t_0) + v_1v_1^T(\gamma(r)-\gamma(t-t_0))$.
It is easy to verify that $V_\bot^T \gamma(t) = V_\bot^T \gamma_1(t)$,  $V_\bot^T \dot{\gamma}(t) = V_\bot^T \dot{\gamma_1}(t)$ and thereby
\[
0 \leq \langle V_\bot^T\dot{\gamma}(t), V_\bot^TW(\gamma(t)) \rangle \leq \langle V_\bot^T\dot{\gamma}(t), V_\bot^TW(\gamma_1(t)) \rangle = \langle V_\bot^T\dot{\gamma_1}(t), V_\bot^TW(\gamma_1(t)) \rangle
\]
by Assumption \RVSD{6.1} (b) and (c).

Considering that $v_1^T\gamma_1(t)$ and $v_1^T\gamma_2(t)$ are constant, we have $v_1^T\dot{\gamma_1}(t) = v_1^T\dot{\gamma_2}(t)$ and thereby
\begin{align*}
         \int_{0}^{t_0} \langle V_\bot^T\dot{\gamma}(t), V_\bot^TW(\gamma(t)) \rangle dt
    & \leq \int_{0}^{t_0} \langle V_\bot^T\dot{\gamma_1}(t), V_\bot^TW(\gamma_1(t)) \rangle dt
    = \int_{0}^{t_0} \langle \dot{\gamma_1}(t), W(\gamma_1(t)) \rangle dt \\
    & = \mathcal{L}(W,\gamma_1) \leq \mathcal{L}(W,\gamma_1) + \mathcal{L}(W,\gamma_0),
\end{align*}
where $\gamma_0 = \arg\sup_{\gamma \in \Gamma_+(\gamma_1(t_0),p_1,v_1^Tp_1)} \mathcal{L}(W,\gamma)$. Since we can always reparameterize $\gamma_1$ to be a unit speed one, say $\tilde{\gamma}_1$, the concatenate of $\tilde{\gamma}_1$ and $\gamma_0$ belongs to $\bar{\Gamma}_+(\bar{x}_1, p_1, v_1^Tp_1)$. Hence,
\[
\int_{0}^{t_0} \langle V_\bot^T\dot{\gamma}(t), V_\bot^TW(\gamma(t)) \rangle dt
\leq \mathcal{L}(W,\gamma_1) + \mathcal{L}(W,\gamma_0)
\leq \mathcal{L}(W,\bar{\gamma}_1)
\]
We can similarly verify that
\begin{align*}
         \int_{t_0}^{r} \langle V_\bot^T\dot{\gamma}(t), V_\bot^TW(\gamma(t)) \rangle dt
    \leq \mathcal{L}(W,\bar{\gamma}_2).
\end{align*}
Moreover, by $\|W(\gamma(t))\|=1$, we have $ v_1^TW(\gamma(t)) \leq 1$ and thereby
\begin{align}\label{ineq:int_1}
\int_{0}^r v_1^T\dot{\gamma}(t) \cdot v_1^TW(\gamma(t)) dt \leq \int_{0}^r v_1^T \dot{\gamma}(t) dt \leq \mathcal{L}(W,\bar{\gamma}).
\end{align}
Hence,
\begin{align*}
      \int_{0}^r \langle \dot{\gamma}(t), W(\gamma(t)) \rangle dt
    = & \int_{0}^r \big( v_1^T\dot{\gamma}(t) \cdot v_1^TW(\gamma(t)) + \langle V_\bot^T\dot{\gamma}(t), V_\bot^TW(\gamma(t)) \rangle \big) dt \\
    \leq & \mathcal{L}(W,\bar{\gamma}) + \mathcal{L}(W,\bar{\gamma}_1)
    + \mathcal{L}(W,\bar{\gamma}_2) = \mathcal{L}(W,\gamma_s).
\end{align*}
Since $\gamma_s \in \bar{\Gamma}_+(\bar{x}_1, \bar{x}_2)$, the supremum can be achieved, which completes the proof.
\end{proof}

\begin{figure}[th]
    \centering
    \includegraphics[width=0.3\textwidth]{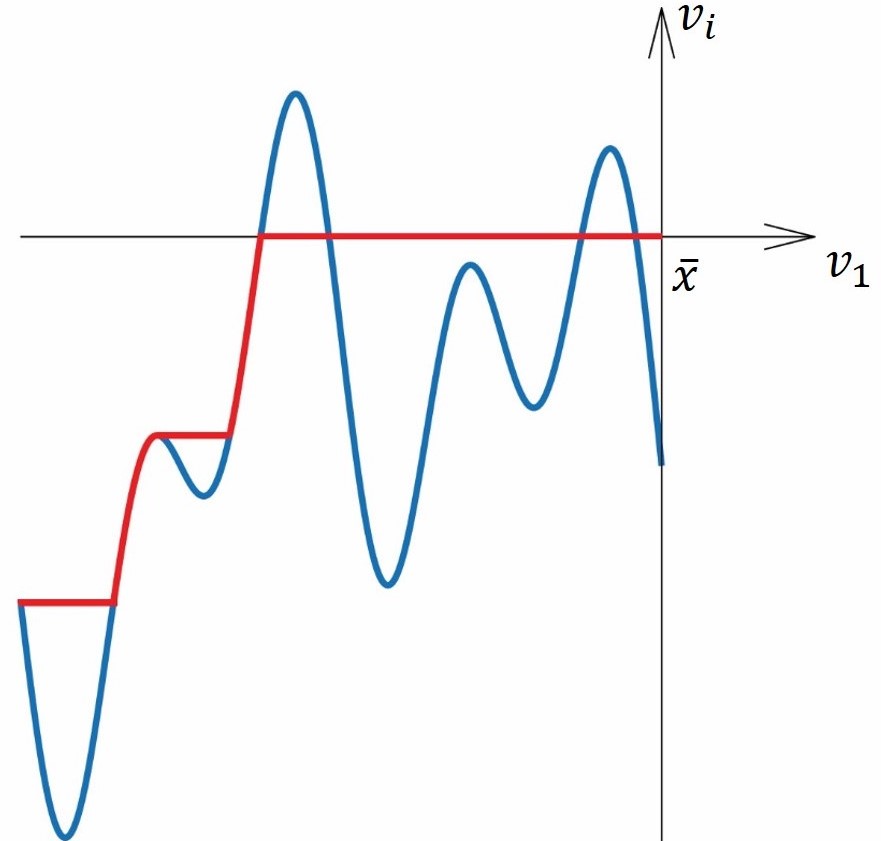}
    \includegraphics[width=0.29\textwidth]{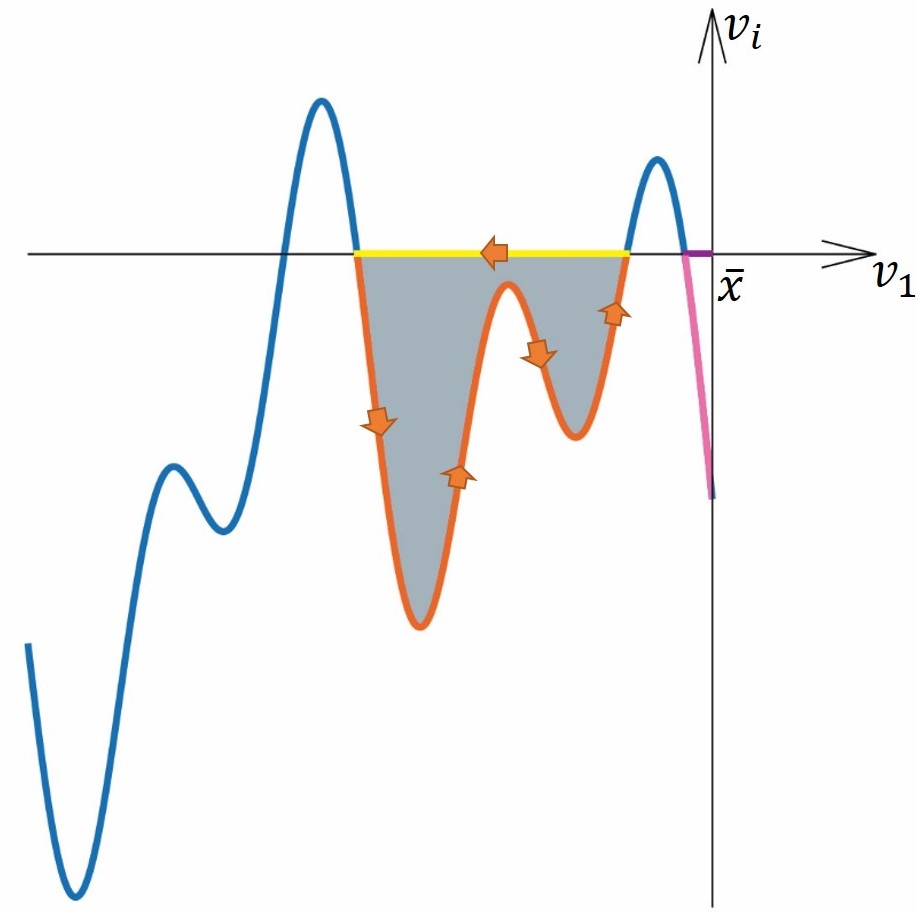}
    \caption{Diagram of $\gamma$ (the blue curve) and $\gamma_+$ (the red curve) in $\M$. The orange and pink curves are segments of $\gamma$ and the yellow and purple curves are segments of $\gamma_+$.}
    \label{Fig:gamma_plus}
\end{figure}

Next, we will discuss the inequality 
\begin{align}\label{Gamma_neg}
\sup_{\gamma \in \Gamma_+(\bar{x}_1, \bar{x}_2)} \ \mathcal{L}(W, \gamma) \quad \geq \quad \sup_{\gamma \in \Gamma(\bar{x}_1, \bar{x}_2)/\Gamma_+(\bar{x}_1, \bar{x}_2)} \ \mathcal{L}(W, \gamma).
\end{align}
Actually, if $\gamma \in \Gamma(\bar{x}_1, \bar{x}_2)/\Gamma_+(\bar{x}_1, \bar{x}_2)$ satisfies $\langle v_1^T\dot{\gamma}(t), v_1^TW(\gamma(t)) \rangle \geq 0 $,
then we define $\gamma_+$ by $\{v_i^T \gamma_+(t)\}_{i=1}^m$. We specially set $v_1^T\gamma_+(t) = v_1^T\gamma(t)$, and for $i \geq 2$ we set
\begin{align} \label{eq:gamma_plus}
    v_i^T \gamma_+ (t) =
    \begin{cases}
    \min\{\max_{s \leq t} \gamma(t), 0\} & v_i^T\bar{x}_1 \leq 0, t \leq t_0 \\
    \min\{\max_{s \geq t} \gamma(t), 0\} & v_i^T\bar{x}_1 \leq 0, t \geq t_0 \\
    \max\{\min_{s \leq t} \gamma(t), 0\} & v_i^T\bar{x}_1 \geq 0, t \leq t_0 \\
    \max\{\min_{s \geq t} \gamma(t), 0\} & v_i^T\bar{x}_2 \geq 0, t \geq t_0
    \end{cases}
\end{align}
where $t_0$ is defined in the proof of Proposition 6.2. Using Assumption \RVSD{6.1} (b), we can verify that $\dot{\gamma}_+(t) \odot W(\gamma_+(t)) \geq 0$.

In Figure \ref{Fig:gamma_plus}, we display the cross sectional area of $\M$ along the first and $i$-th axis for $i \geq 2$. In the left panel, the blue curve is $\gamma$ and the red curve is $\gamma_+$. Without loss of generality, we focus on $v_i^T \bar{x}_1 \leq 0$ and $t \leq t_0$. The other three cases in (\ref{eq:gamma_plus}) can be similarly verified.

First, we compare the integrals over the orange curve $\mathcal{C}_1$ and the yellow curve $\mathcal{C}_2$ in Figure \ref{Fig:gamma_plus}. Then the integral on $\mathcal{C}_1$ denoted by $I_1$ is
\[
I_1 = \int_{t_1}^{t_2} v_i^T \dot{\gamma}(t) \cdot v_i^TW(\gamma(t)) dt  = \int_{t_1}^{t_2} v_i^TW(\gamma(t)) v_i^Td \gamma(t) = \int_{\mathcal{C}_1} v_i^TW(z) dz_i,
\]
and $I_2 = \int_{\mathcal{C}_2} v_i^TW(z) dz_i$. Then, $I_1 - I_2$ is the integral of $v_i^TW(z)$ over the closed anticlockwise curve consisting of  $\mathcal{C}_1$ and the inverse of $\mathcal{C}_2$. When $d=2$, such integral is equal to an integral over the gray region denoted by $\mathcal{D}$ shown in the right panel of Figure \ref{Fig:gamma_plus} by Green Theorem, that is,
\[
I_1 - I_2 = \iint_\mathcal{D} \frac{\partial v_2^T W(z)}{\partial z_1} dz_1dz_2 \leq 0,
\]
since $ \frac{\partial v_2^T W(z)}{\partial z_1} \leq 0$ for $z_1 \leq 0$ based on Assumption \RVSD{6.1}(c). For $i \geq 2$, if $\frac{\partial v_i^TW(z)}{\partial z_j} \geq 0$ holds for any $j>i$ and $\frac{\partial v_i^TW(z)}{\partial z_j}<0$ holds for any $j<i$, the conclusion can be extended to a higher dimension by the Stokes' theorem.

Second, we compare the integrals over the purple and pink curve in Figure \ref{Fig:gamma_plus}. By Assumption \RVSD{6.1} (b), the integral of $v_i^TW$ over the purple curve is negative, while the integral over the pink curve is zero. So, the integral of $v_i^TW$ over the purple curve is less than the pink curve. The above discussion summarizes
$
    \int_{0}^{t_0} v_i^T\dot{\gamma}(t) \cdot v_i^T W(\gamma(t)) dt  \leq \int_{0}^{t_0} v_i^T\dot{\gamma}_+(t) \cdot v_i^T W(\gamma_+(t)) dt,
$
for any $i \geq 2$. Moreover,
\begin{align*}
    \int_{0}^{t_0} \langle V_\bot^T\dot{\gamma}(t) , V_\bot^T W(\gamma_+(t))\rangle
    & = \sum_{i=2}^m \int_{0}^{t_0} v_i^T\dot{\gamma}(t) \cdot v_i^T W(\gamma(t)) dt
    \leq \sum_{i=2}^m \int_{0}^{t_0} v_i^T\dot{\gamma}_+(t) \cdot v_i^T W(\gamma_+(t)) dt \\
    & = \int_{0}^{t_0} \langle V_\bot^T\dot{\gamma}_+(t) , V_\bot^T W(\gamma_+(t))\rangle dt \leq \mathcal{L}(W,\bar{\gamma}_1),
\end{align*}
where the last inequality can be verified by similar proof of Proposition 6.2. Implementing the above discussion for $t \geq t_0$ analogically, we also have
\begin{align*}
    \int_{t_0}^r \langle V_\bot^T\dot{\gamma}(t) , V_\bot^T W(\gamma(t))\rangle dt \leq \mathcal{L}(W,\bar{\gamma}_2).
\end{align*}
Along with (\ref{ineq:int_1}) we conclude
\begin{align*}
\mathcal{L}(W,\gamma) & = \int_{0}^r v_1^T\dot{\gamma}(t) \cdot v_1^T W(\gamma(t)) dt + \sum_{i=2}^m \int_{0}^r v_i^T\dot{\gamma}(t) \cdot v_i^T W(\gamma(t)) dt \\
& \leq \mathcal{L}(W,\bar{\gamma})+\mathcal{L}(W,\bar{\gamma})+\mathcal{L}(W,\bar{\gamma}) = \sup_{\gamma \in \Gamma_+(\bar{x}_1, \bar{x}_2)}  \ \mathcal{L}(W, \gamma),
\end{align*}
which supports the inequality (\ref{Gamma_neg}).




\spacingset{1.5}
\bibliographystyle{Chicago}

\bibliography{FBFreferences}
\end{document}